\documentclass[12pt]{article}%
\usepackage{amsmath}
\usepackage{amsfonts}
\usepackage{amssymb}
\usepackage{graphicx}%
\providecommand{\U}[1]{\protect\rule{.1in}{.1in}}
\newtheorem{theorem}{Theorem}
\newtheorem{acknowledgement}[theorem]{Acknowledgement}

\newtheorem{lemma}[theorem]{Lemma}

\newtheorem{remark}[theorem]{Remark}

\newenvironment{proof}[1][Proof]{\noindent\textbf{#1.} }{\ \rule{0.5em}{0.5em}}
\begin{document}

\title{Discrete semiclassical orthogonal polynomials of class 2}
\author{Diego Dominici \thanks{e-mail: diego.dominici@dk-compmath.jku.at}\\Johannes Kepler University Linz\\Doktoratskolleg \textquotedblleft Computational Mathematics\textquotedblright\\Altenberger Stra\ss e 69, 4040 Linz, Austria.\\Permanent address: Department of Mathematics\\State University of New York at New Paltz\\1 Hawk Dr., New Paltz, NY 12561-2443, USA
\and Francisco Marcell{\'a}n \thanks{e-mail: pacomarc@ing.uc3m.es}\\Departamento de Matem\'aticas \\Universidad Carlos III de Madrid \\Escuela Polit\'ecnica Superior \\Av. Universidad 30 \\28911 Legan\'es\\Spain}
\maketitle

\begin{abstract}
In this contribution, discrete semiclassical orthogonal polynomials of class
$s\leq2$ are studied. By considering all possible solutions of the Pearson
equation, we obtain the canonical families in each class. We also consider
limit relations between these and other families of orthogonal polynomials.

\end{abstract}

\strut

\strut

\section{\strut Introduction}

Semiclassical linear functionals with respect to the derivative operator have
received increasing attention in the literature of orthogonal polynomials
taking into account their connections with many interesting problems in
mathematical physics and numerical quadratures. From the so called Pearson
equation associated with those linear functionals, ladder operators for the
corresponding sequences of orthogonal polynomials can be deduced from the
structure relations they satisfy (see \cite{MR803215},\cite{MR932783},
\cite{MR1270222}). As a direct consequence, a simple computation allows to
deduce the corresponding holonomic equations, i.e. second order linear
differential equations with polynomial coefficients of bounded degrees
(depending on the degree of the polynomial) associated with the class of the
linear functional (see \cite{MR2191786}, \cite{MR1270222}).\newline

The case of semiclassical linear functionals with respect to the difference
operator in a uniform lattice is related to discretizations of holonomic
equations, a fact pointed out in the classical case by Nikiforov, Uvarov and
Suslov in a stimulating monograph on classical discrete orthogonal polynomials
\cite{MR1149380}.\newline

In both situations, a key problem is the classification of semiclassical
linear functionals by using a hierarchy according to their class. This
represents an alternative method to the Askey tableau, where classical
orthogonal polynomials appear as a hierarchy of hypergeometric polynomials
(see \cite{MR2656096}).\newline

Examples of semiclassical linear functionals with respect to the derivative
operator when the class is either $s=1$ (see \cite{MR1186737}) or $s=2$ (see
\cite{MR2929632}) have been studied in the literature. Nevertheless, some of
them are related to perturbations of classical linear functionals by the
addition of Dirac functionals, or their derivatives, supported on convenient
points (\cite{MR2061464}, \cite{MR1199643}). The important fact is that these
orthogonal polynomials, the so called Krall-type orthogonal polynomials, are
eigenfunctions of higher order differential operators. They are related to
bispectral problems and have been intensively used in the generation of
exceptional polynomials, an useful tool in the study of integrable
systems.\newline

In the discrete case, there is a limited development in these topics (see
\cite{MR1665164}). The case $s=0$ has been deeply studied by many authors (see
\cite{MR1464669}, \cite{MR1340932} ) and the approach to the case $s=1$ has
been done by Dominici and Marcell\'{a}n in a recent paper (see
\cite{MR3227440}). Some examples of semiclassical discrete polynomials for
$s=1$ appeared in \cite{MR3765964}.\newline

The aim of this contribution is twofold. First, we deal with three different
perturbations of discrete linear functionals based on the so called linear
spectral transformations (Uvarov, Christoffel and Geronimus) \cite{MR1482157},
truncations and symmetrization processes, respectively. Second, from the above
transformations we are able to generate semiclassical linear functionals of
class $s=1$ and $s=2.$ Some of them appear in a natural way in the literature
but others are new. We put attention only in the representation of the
corresponding Stieltjes functions and, as a consequence, we deduce the
corresponding class. Notice that an interesting and open problem is to analyze
the coefficients of the three term recurrence relation they satisfy, in
particular the so called Laguerre-Freud equations. For generalized Charlier
polynomials such equations appeared in the monograph by W. Van Assche
\cite{MR3729446} as well as in \cite{MR1737084}. They lead to limiting cases
of the discrete Painlev\'{e} equation corresponding to $D_{4}^{c}$ in the
Sakai's classification. A generalization of the Krawtchouk polynomials and the
fifth Painlev\'{e} equation are studied in \cite{MR3173496}. Some interesting
examples of semiclassical extensions of Meixner polynomials appeared in
\cite{MR2749070} and \cite{MR2861208}. They yield a limit case of asymmetric
discrete Painlev\'{e} IV equations. Discrete orthogonal polynomials with
respect to the hypergeometric weight have been analyzed in \cite{MR3846851}
and \cite{MR3813279}.\newline

The structure of the paper is the following. In Section 2, a basic background
about discrete semiclassical linear functionals is given. In particular, the
class is defined from the degrees of polynomials appearing in the discrete
Pearson equation they satisfy. We emphasize the role of the Stieltjes function
associated with every linear functional. In the semiclassical case, the
corresponding Stieltjes function satisfies a first order linear difference
equation with polynomial coefficients. They provide the key information about
the class of the linear functional. Section 3 deals with the three canonical
perturbations of linear functionals we will deal in the sequel. The
corresponding Stieltjes functions are deduced. In Section 4, classical
discrete linear functionals are revisited. In Section 5 we analyze
semiclassical linear functionals of class $s=1$ in such a way that the results
given in \cite{MR3227440} are completed. Finally, in Section 6, we study
examples of discrete semiclassical linear functionals of class $s=2$ generated
by the three canonical perturbations introduced in Section 3.\newline

\section{Basic background on discrete semiclassical linear functionals}

Let $\mathbb{N}_{0}$ denote the set $\mathbb{N}\cup\left\{  0\right\}
=0,1,2,\ldots.$ Throughout the paper we will use the notation
\[
\mathbf{a}=\left(  a_{1},\ldots,a_{p}\right)  ,\quad\mathbf{b}=\left(
b_{1},\ldots,b_{q}\right)  ,\quad p,q\in\mathbb{N}_{0},
\]%
\begin{equation}
x+\mathbf{a=}%
%TCIMACRO{\dprod \limits_{i=1}^{p}}%
%BeginExpansion
{\displaystyle\prod\limits_{i=1}^{p}}
%EndExpansion
\left(  x+a_{i}\right)  ,\quad x+\mathbf{b}=%
%TCIMACRO{\dprod \limits_{i=1}^{q}}%
%BeginExpansion
{\displaystyle\prod\limits_{i=1}^{q}}
%EndExpansion
\left(  x+b_{i}\right)  , \label{x+a}%
\end{equation}
and%
\[
\left(  \mathbf{a}\right)  _{n}=%
%TCIMACRO{\dprod \limits_{i=1}^{p}}%
%BeginExpansion
{\displaystyle\prod\limits_{i=1}^{p}}
%EndExpansion
\left(  a_{i}\right)  _{n},\quad\left(  \mathbf{b}\right)  _{n}=%
%TCIMACRO{\dprod \limits_{i=1}^{q}}%
%BeginExpansion
{\displaystyle\prod\limits_{i=1}^{q}}
%EndExpansion
\left(  b_{i}\right)  _{n},\quad n\in\mathbb{N}_{0},
\]
where the \emph{Pochhammer polynomial} $\left(  x\right)  _{n}$ is defined by
$\left(  x\right)  _{0}=1$ and \cite[18:12]{MR2466333}%
\begin{equation}
\left(  x\right)  _{n}=%
%TCIMACRO{\dprod \limits_{j=0}^{n-1}}%
%BeginExpansion
{\displaystyle\prod\limits_{j=0}^{n-1}}
%EndExpansion
\left(  x+j\right)  ,\quad n\in\mathbb{N}. \label{Poch}%
\end{equation}
From (\ref{Poch}), we immediately obtain the recurrence%
\begin{equation}
\left(  x\right)  _{n+1}=\left(  x+n\right)  \left(  x\right)  _{n},\quad
n\in\mathbb{N}_{0}. \label{Poch1}%
\end{equation}
On the other hand, if we consider the analytic continuation with $\alpha
\in\mathbb{C},$ then%
\begin{equation}
\left(  x\right)  _{\alpha}=\frac{\Gamma\left(  x+\alpha\right)  }%
{\Gamma\left(  x\right)  },\quad-\left(  x+\alpha\right)  \notin\mathbb{N}%
_{0}, \label{Gamma}%
\end{equation}
where $\Gamma\left(  x\right)  $ denotes the Gamma function \cite[Chapter
5.]{MR2723248}.

In this work, we deal with linear functionals defined in $\mathbb{C}\left[
x\right]  ,$ the linear space of polynomials with complex coefficients.
Indeed,
\begin{equation}
L\left[  r\right]  =%
%TCIMACRO{\dsum \limits_{x=0}^{\infty}}%
%BeginExpansion
{\displaystyle\sum\limits_{x=0}^{\infty}}
%EndExpansion
r\left(  x\right)  \varrho\left(  x\right)  ,\quad r\in\mathbb{C}\left[
x\right]  , \label{Ldisc}%
\end{equation}
where the weight function $\varrho:\mathbb{N}_{0}\rightarrow\mathbb{C}$ is
\begin{equation}
\varrho\left(  x\right)  =\frac{\left(  \mathbf{a}\right)  _{x}}{\left(
\mathbf{b}+1\right)  _{x}}\frac{z^{x}}{x!}. \label{rho}%
\end{equation}
Using (\ref{Poch1}), we see that the weight function $\varrho\left(  x\right)
$ satisfies the \emph{Pearson equation}%
\begin{equation}
\frac{\varrho\left(  x+1\right)  }{\varrho\left(  x\right)  }=\frac
{\eta\left(  x\right)  }{\sigma\left(  x+1\right)  }, \label{pearson}%
\end{equation}
where the polynomials $\eta\left(  x\right)  ,\sigma\left(  x\right)  $ are
\begin{equation}
\eta\left(  x\right)  =z\left(  x+\mathbf{a}\right)  ,\quad\sigma\left(
x\right)  =x\left(  x+\mathbf{b}\right)  . \label{eta sigma}%
\end{equation}
Note that from (\ref{x+a}) we have $\deg\left(  \eta\right)  =p$ and
$\deg\left(  \sigma\right)  =q+1$ and we assume they are coprime.

The linear functional $L$ (and also $\varrho)$ is said to be
\emph{semiclassical}. The class of $L$ (and also $\varrho)$ is the
non-negative integer number
\[
s=\max\left\{  \deg\left(  \sigma\right)  -2,\ \deg\left(  \sigma-\eta\right)
-1\right\}  .
\]
The linear functionals of class $s=0$ are called \emph{discrete classical
\cite{MR1340932}} linear functionals.

The number $s$ depends on $p,q,$ and $z.$ It follows from (\ref{eta sigma})
that we have four cases to consider:

\begin{enumerate}
\item $p>q+1.$ In this case, $\deg\left(  \sigma-\eta\right)  -1=p-1>q$ and
$s=\max\left\{  q-1,\ p-1\right\}  =p-1.$

\item $p<q+1.$ In this case, $\deg\left(  \sigma-\eta\right)  -1=q$ and
$s=\max\left\{  q-1,\ q\right\}  =q.$

\item $q+1=p,$ $z\neq1.$ In this case, $\deg\left(  \sigma-\eta\right)  -1=q$
and $s=\max\left\{  q-1,\ q\right\}  =q.$

\item $q+1=p,$ $z=1.$ In this case, $\deg\left(  \sigma-\eta\right)  -1=q-1$
and $s=\max\left\{  q-1,\ q-1\right\}  =q-1.$
\end{enumerate}

Thus, we conclude that%
\begin{equation}
s=\left\{
\begin{array}
[c]{ll}%
p-1, & p>q+1,\\
q, & p<q+1,\\
q, & p=q+1,\quad z\neq1,\\
q-1, & p=q+1,\quad z=1.
\end{array}
\right.  \label{s}%
\end{equation}
Note that we can also write $s$ as%
\begin{equation}
s=\left\{
\begin{array}
[c]{ll}%
\max\left\{  \deg\left(  \sigma\right)  -1,\ \deg\left(  \sigma-\eta\right)
-1\right\}  , & \deg\left(  \sigma-\eta\right)  \neq\deg\left(  \sigma\right)
-1,\\
\deg\left(  \sigma\right)  -2, & \deg\left(  \sigma-\eta\right)  =\deg\left(
\sigma\right)  -1.
\end{array}
\right.  \label{s1}%
\end{equation}

The \emph{Stieltjes transform} of the linear functional $L$ is defined by%
\begin{equation}
S\left(  t\right)  =L\left[  \frac{1}{t-x}\right]  . \label{Stieltjes}%
\end{equation}
Here the linear functional is acting on $x.$ The connection between
$\eta\left(  x\right)  ,\sigma\left(  x\right)  $ and $S\left(  t\right)  $ is
given by the following result.

\begin{lemma}
The Stieltjes transform $S\left(  t\right)  $ of the linear functional $L$
satisfies the difference equation%
\begin{equation}
\sigma\left(  t+1\right)  S\left(  t+1\right)  -\eta\left(  t\right)  S\left(
t\right)  =\xi\left(  t\right)  , \label{ST1}%
\end{equation}
where $\xi\left(  t\right)  $ is a polynomial of degree $s$ that is the class
of the linear functional.
\end{lemma}

\begin{proof}
From (\ref{Stieltjes}), we have%
\begin{gather*}
\sigma\left(  t+1\right)  S\left(  t+1\right)  -\eta\left(  t\right)  S\left(
t\right)  =\sigma\left(  t+1\right)
%TCIMACRO{\dsum \limits_{x=0}^{\infty}}%
%BeginExpansion
{\displaystyle\sum\limits_{x=0}^{\infty}}
%EndExpansion
\frac{\varrho\left(  x\right)  }{t+1-x}-\eta\left(  t\right)
%TCIMACRO{\dsum \limits_{x=0}^{\infty}}%
%BeginExpansion
{\displaystyle\sum\limits_{x=0}^{\infty}}
%EndExpansion
\frac{\varrho\left(  x\right)  }{t-x}\\
=\frac{\sigma\left(  t+1\right)  }{t+1}+\sigma\left(  t+1\right)
%TCIMACRO{\dsum \limits_{x=1}^{\infty}}%
%BeginExpansion
{\displaystyle\sum\limits_{x=1}^{\infty}}
%EndExpansion
\frac{\varrho\left(  x\right)  }{t+1-x}-\eta\left(  t\right)
%TCIMACRO{\dsum \limits_{x=0}^{\infty}}%
%BeginExpansion
{\displaystyle\sum\limits_{x=0}^{\infty}}
%EndExpansion
\frac{\varrho\left(  x\right)  }{t-x}\\
=\frac{\sigma\left(  t+1\right)  }{t+1}+\sigma\left(  t+1\right)
%TCIMACRO{\dsum \limits_{x=0}^{\infty}}%
%BeginExpansion
{\displaystyle\sum\limits_{x=0}^{\infty}}
%EndExpansion
\frac{\varrho\left(  x+1\right)  }{t-x}-\eta\left(  t\right)
%TCIMACRO{\dsum \limits_{x=0}^{\infty}}%
%BeginExpansion
{\displaystyle\sum\limits_{x=0}^{\infty}}
%EndExpansion
\frac{\varrho\left(  x\right)  }{t-x}.
\end{gather*}
Using the Pearson equation (\ref{pearson}), we get%
\begin{gather*}
\sigma\left(  t+1\right)  S\left(  t+1\right)  -\eta\left(  t\right)  S\left(
t\right)  =\frac{\sigma\left(  t+1\right)  }{t+1}+\sigma\left(  t+1\right)
%TCIMACRO{\dsum \limits_{x=0}^{\infty}}%
%BeginExpansion
{\displaystyle\sum\limits_{x=0}^{\infty}}
%EndExpansion
\frac{\eta\left(  x\right)  }{\sigma\left(  x+1\right)  }\frac{\varrho\left(
x\right)  }{t-x}-\eta\left(  t\right)
%TCIMACRO{\dsum \limits_{x=0}^{\infty}}%
%BeginExpansion
{\displaystyle\sum\limits_{x=0}^{\infty}}
%EndExpansion
\frac{\varrho\left(  x\right)  }{t-x}\\
=\frac{\sigma\left(  t+1\right)  }{t+1}+%
%TCIMACRO{\dsum \limits_{x=0}^{\infty}}%
%BeginExpansion
{\displaystyle\sum\limits_{x=0}^{\infty}}
%EndExpansion
\frac{\sigma\left(  t+1\right)  \eta\left(  x\right)  -\eta\left(  t\right)
\sigma\left(  x+1\right)  }{\sigma\left(  x+1\right)  }\frac{\varrho\left(
x\right)  }{t-x}.
\end{gather*}
But%
\[
\lambda_{x}\left(  t\right)  =\frac{\sigma\left(  t+1\right)  \eta\left(
x\right)  -\eta\left(  t\right)  \sigma\left(  x+1\right)  }{t-x}%
\]
is a polynomial in $t$ (with coefficients depending on $x)$ of degree
$\deg\left(  \sigma-\eta\right)  -1.$ Since $\eta\left(  x\right)  \neq
\sigma\left(  x+1\right)  $ (otherwise $\varrho$ is a constant function), we
have $\deg\left(  \sigma-\eta\right)  \neq\deg\left(  \sigma\right)  -1.$

Therefore,%
\begin{equation}
\xi\left(  t\right)  =\frac{\sigma\left(  t+1\right)  }{t+1}+%
%TCIMACRO{\dsum \limits_{x=0}^{\infty}}%
%BeginExpansion
{\displaystyle\sum\limits_{x=0}^{\infty}}
%EndExpansion
\lambda_{x}\left(  t\right)  \frac{\varrho\left(  x\right)  }{\sigma\left(
x+1\right)  } \label{xi}%
\end{equation}
is a polynomial in $t$ of degree $\max\left\{  \deg\left(  \sigma\right)
-1,\ \deg\left(  \sigma-\eta\right)  -1\right\}  .$ From (\ref{s1}) we
conclude that $\deg\left(  \xi\right)  =s.$
\end{proof}

Let the \emph{falling factorial polynomials} be defined by $\phi_{0}\left(
x\right)  =1$ and
\begin{equation}
\phi_{n}\left(  x\right)  =%
%TCIMACRO{\dprod \limits_{j=0}^{n-1}}%
%BeginExpansion
{\displaystyle\prod\limits_{j=0}^{n-1}}
%EndExpansion
\left(  x-j\right)  ,\quad n\in\mathbb{N}. \label{phidef}%
\end{equation}
Comparing (\ref{phidef}) with (\ref{Poch}), we see that%
\begin{equation}
\phi_{n}\left(  x\right)  =\left(  -1\right)  ^{n}\left(  -x\right)
_{n}=\left(  x-n+1\right)  _{n}, n\in\mathbb{N}_{0}, \label{equal}%
\end{equation}
or, using (\ref{Gamma}),%
\begin{equation}
\phi_{n}\left(  x\right)  =\frac{\Gamma\left(  x+1\right)  }{\Gamma\left(
x-n+1\right)  }=\frac{x!}{\left(  x-n\right)  !}=n!\binom{x}{n},\quad
n\in\mathbb{N}_{0}. \label{phi-gamma}%
\end{equation}

We define the \emph{moments} of $L$ with respect to the basis $\{\phi
_{n}(x)\}_{n\geq0}$ by%
\[
\nu_{n}=L\left[  \phi_{n}\right]  ,\quad n\in\mathbb{N}_{0}.
\]
From (\ref{rho}) and (\ref{phi-gamma}), we have%
\[
\nu_{n}=%
%TCIMACRO{\dsum \limits_{x=n}^{\infty}}%
%BeginExpansion
{\displaystyle\sum\limits_{x=n}^{\infty}}
%EndExpansion
\frac{\left(  \mathbf{a}\right)  _{x}}{\left(  \mathbf{b}+1\right)  _{x}}%
\frac{z^{x}}{\left(  x-n\right)  !}=%
%TCIMACRO{\dsum \limits_{x=0}^{\infty}}%
%BeginExpansion
{\displaystyle\sum\limits_{x=0}^{\infty}}
%EndExpansion
\frac{\left(  \mathbf{a}\right)  _{x+n}}{\left(  \mathbf{b}+1\right)  _{x+n}%
}\frac{z^{x+n}}{x!}.
\]
If we use the recurrence \cite[18:5:12]{MR2466333}%
\begin{equation}
\left(  x\right)  _{n+m}=\left(  x\right)  _{n}\left(  x+n\right)  _{m},\quad
n,m\in\mathbb{N}_{0},\label{poch2}%
\end{equation}
we get%
\begin{equation}
\nu_{n}=z^{n}\frac{\left(  \mathbf{a}\right)  _{n}}{\left(  \mathbf{b}%
+1\right)  _{n}}%
%TCIMACRO{\dsum \limits_{x=0}^{\infty}}%
%BeginExpansion
{\displaystyle\sum\limits_{x=0}^{\infty}}
%EndExpansion
\frac{\left(  \mathbf{a}+n\right)  _{x}}{\left(  \mathbf{b}+1+n\right)  _{x}%
}\frac{z^{x}}{x!}=z^{n}\frac{\left(  \mathbf{a}\right)  _{n}}{\left(
\mathbf{b}+1\right)  _{n}}\ _{p}F_{q}\left(
\begin{array}
[c]{c}%
\mathbf{a}+n\\
\mathbf{b}+1+n
\end{array}
;z\right)  ,\label{moment}%
\end{equation}
where$\ _{p}F_{q}$ denotes the \emph{generalized hypergeometric function}
defined by \cite[16.2]{MR2723248}%
\begin{equation}
\ _{p}F_{q}\left(
\begin{array}
[c]{c}%
a_{1},\ldots,a_{p}\\
b_{1},\ldots,b_{q}%
\end{array}
;z\right)  =%
%TCIMACRO{\dsum \limits_{k=0}^{\infty}}%
%BeginExpansion
{\displaystyle\sum\limits_{k=0}^{\infty}}
%EndExpansion
\frac{\left(  a_{1}\right)  _{k}\cdots\left(  a_{p}\right)  _{k}}{\left(
b_{1}\right)  _{k}\cdots\left(  b_{q}\right)  _{k}}\frac{z^{k}}{k!}%
.\label{series}%
\end{equation}
The moments of discrete semiclassical orthogonal polynomials of class $s\leq1$
were studied in \cite{MR3492864}.

\begin{remark}
\label{Hyper}The convergence of the series (\ref{series}) depends on the
values of $p$ and $q.$ Three different situations appear \cite[16.2]%
{MR2723248}:

\begin{enumerate}
\item If $p<q+1$, then $\ _{p}F_{q}$ is an entire function of $z.$

\item If $p=q+1$, then $\ _{p}F_{q}$ is analytic inside the unit circle,
$\left\vert z\right\vert <1,$ and can be extended by analytic continuation to
the cut plane $\mathbb{C}\setminus\lbrack1,\infty).$ Let
\begin{equation}
\gamma=b_{1}+\cdots+b_{q}-\left(  a_{1}+\cdots+a_{q+1}\right)  . \label{gamma}%
\end{equation}
On the unit circle $\left\vert z\right\vert =1,$ the series (\ref{series}) is

(i) absolutely convergent if $\operatorname{Re}\left(  \gamma\right)  >0,$

(ii) convergent except at $z=1$ if $\operatorname{Re}\left(  \gamma\right)
\in(-1,0],$

and

(iii) divergent if $\operatorname{Re}\left(  \gamma\right)  \leq-1.$

\item If $p>q+1,$ then $\ _{p}F_{q}$ diverges for all $z\neq0,$ up to
$a_{i}=-N,$ with $N\in\mathbb{N}$ for some $1\leq i\leq p.$ In this
case,$\ _{p}F_{q}$ becomes a polynomial of degree $N.$
\end{enumerate}
\end{remark}

%Our present contribution is focussed on discrete semiclassical polynomials of class $s\leq2.$
%In a previous work \cite{MR3227440}, we have classified the discrete
%semiclassical polynomials of class $s\leq1.$ However, we have considered only
%one type of modification of functionals (the addition of mass points). Now we include three additional processes: the so called spectral linear transformations,
%truncation and symmetrization,   in such a way we extend the particular cases
%listed in \cite{MR3227440}. Therefore, for completion purposes we list all
%families of discrete semiclassical polynomials of class $s\leq2.$

\section{Modification of functionals}

In this section, we describe three ways in order to change the class of a
linear functional.

\subsection{Rational spectral transformations}

Let $L$ be a linear functional and $S\left(  t\right)  $ denote its Stieltjes
transform introduced in (\ref{Stieltjes}). A \emph{rational spectral
transformation} of $S\left(  t\right)  $ is defined by \cite{MR1482157}%
\[
\widetilde{L}\left[  \frac{1}{t-x}\right]  =\widetilde{S}\left(  t\right)
=\frac{A\left(  t\right)  S\left(  t\right)  +B\left(  t\right)  }{C\left(
t\right)  S\left(  t\right)  +D\left(  t\right)  },
\]
where $A\left(  t\right)  ,B\left(  t\right)  ,C\left(  t\right)  ,D\left(
t\right)  $ are polynomials such that $A(z)D(z)-B(z)C(z)\neq0.$ See also
\cite{MR2055354}, \cite{MR1140648}, \cite{MR1934914}, \cite{MR1199259},
\cite{MR0262764}, and \cite{MR1482157}, among others.

It was shown in \cite{MR1482157} what are the families of semiclassical
orthogonal polynomials related by \emph{spectral linear transformations} (with
$C=0)$. They can be written as a composition of two basic transformations:

\begin{enumerate}
\item The \emph{Christoffel transformation}, (see \cite{MR2324861}%
,\cite{MR2055354}, \cite{MR1579059}, \cite{MR1939588}) is defined by%
\begin{equation}
\frac{L_{C}\left[  r\right]  }{\lambda_{C}}=L\left[  \left(  x-\omega\right)
r\left(  x\right)  \right]  ,\quad r\in\mathbb{C}\left[  x\right]  ,
\label{LC}%
\end{equation}
or, equivalently, by%
\begin{equation}
\frac{S_{C}\left(  z\right)  }{\lambda_{C}}=\left(  z-\omega\right)  S\left(
z\right)  -L\left[  1\right]  , \label{S Chris}%
\end{equation}
where $L\left(  x-\omega\right)  \neq0, n\geq0,$ and%
\[
\lambda_{C}=\frac{L_{C}\left[  1\right]  }{L\left[  x-\omega\right]  }.
\]

\item The \emph{Geronimus transformation } (see \cite{MR3208415},
\cite{MR3264577}, \cite{MR0004339}, \cite{MR0004340}, \cite{MR1934914},
\cite{MR1105709}) is defined by%
\begin{equation}
\frac{L_{G}\left[  r\right]  }{\lambda_{G}}=L\left[  \frac{r\left(  x\right)
}{x-\omega}\right]  +Mr\left(  \omega\right)  ,\quad r\in\mathbb{C}\left[
x\right]  , \label{LG}%
\end{equation}
or, equivalently, by%
\begin{equation}
\left(  z-\omega\right)  \frac{S_{G}\left(  z\right)  }{\lambda_{G}}=S\left(
z\right)  -S\left(  \omega\right)  +M, \label{S ger}%
\end{equation}
with%
\[
\lambda_{G}=\frac{L_{G}\left[  1\right]  }{M-S\left(  \omega\right)  },
\]
where we assume that $S\left(  z\right)  $ is analytic at $z=\omega$ and
$M-S\left(  \omega\right)  \neq0.$
\end{enumerate}

On the other hand, the \emph{Uvarov transformation} (see \cite{MR2064407},
\cite{MR1199643}, \cite{MR0262764}) is defined by%
\begin{equation}
\frac{L_{U}\left[  r\right]  }{\lambda_{U}}=L\left[  r\right]  +Mr\left(
\omega\right)  ,\quad r\in\mathbb{C}\left[  x\right]  , \label{LU}%
\end{equation}
or by%
\begin{equation}
\frac{S_{U}\left(  z\right)  }{\lambda_{U}}=S\left(  z\right)  +\frac
{M}{z-\omega}, \label{S uva}%
\end{equation}
where $M+L\left[  1\right]  \neq0$ and%
\[
\lambda_{U}=\frac{L_{U}\left[  1\right]  }{L\left[  1\right]  +M}.
\]

Note that since%
\[
r\left(  x\right)  =%
%TCIMACRO{\dsum \limits_{k=0}^{\deg\left(  r\right)  }}%
%BeginExpansion
{\displaystyle\sum\limits_{k=0}^{\deg\left(  r\right)  }}
%EndExpansion
\frac{r^{\left(  k\right)  }\left(  \omega\right)  }{k!}\left(  x-\omega
\right)  ^{k},\quad r\in\mathbb{C}\left[  x\right]  ,
\]
we can rewrite (\ref{LG}) as
\[
\frac{L_{G}\left[  r\right]  }{\lambda_{G}}=\left[  M-S\left(  \omega\right)
\right]  r\left(  \omega\right)  +%
%TCIMACRO{\dsum \limits_{k=1}^{\deg\left(  r\right)  }}%
%BeginExpansion
{\displaystyle\sum\limits_{k=1}^{\deg\left(  r\right)  }}
%EndExpansion
\frac{r^{\left(  k\right)  }\left(  \omega\right)  }{k!}L\left[  \left(
x-\omega\right)  ^{k}\right]  ,
\]
which is well defined for all $r\in\mathbb{C}\left[  x\right]  .$

We denote by $L_{T_{2}T_{1}},S_{T_{2}T_{1}},\varrho_{T_{2}T_{1}},$ etc., the
object obtained by applying the transformation $T_{1}$ followed by the
transformation $T_{2}$ (in other words, by the composition $T_{2}\circ
T_{1}).$

Let us assume that $\lambda_{U}=\lambda_{C}=\lambda_{G}=1.$ If we apply a
Geronimus transformation followed by a Christoffel transformation, we get
\begin{equation}
L_{CG}\left[  r\right]  =L_{G}\left[  \left(  x-\omega\right)  r\left(
x\right)  \right]  =L\left[  \frac{\left(  x-\omega\right)  r\left(  x\right)
}{x-\omega}\right]  +M\left[  \left(  x-\omega\right)  r\left(  x\right)
\right]  _{x=\omega}=L\left[  r\right]  . \label{CG}%
\end{equation}
On the other hand, changing the order in the composition, we obtain%
\begin{equation}
L_{GC}\left[  r\right]  =L_{C}\left[  \frac{r\left(  x\right)  }{x-\omega
}\right]  +Mr\left(  \omega\right)  =L\left[  \frac{\left(  x-\omega\right)
r\left(  x\right)  }{x-\omega}\right]  +Mr\left(  \omega\right)  =L\left[
r\right]  +Mr\left(  \omega\right)  =L_{U}\left[  r\right]  . \label{GC}%
\end{equation}
Thus, the Christoffel and Geronimus transformations are almost inverses of
each other in the sense that $L_{CG}$ is the identity and $L_{GC}$ is an
Uvarov transformation.

\subsubsection{\strut Uvarov transformations}

Let $L$ be given by (\ref{Ldisc}) and consider the linear functional $L_{U}$
defined by (\ref{LU}) with
\[
L_{U}\left[  1\right]  =L\left[  1\right]  +M\neq0.
\]
It follows that the moments of $L_{U}$ with respect to the basis $\{\phi
_{n}(x)\}_{n\geq0}$ are
\begin{equation}
\nu_{n}^{U}=\nu_{n}+M\phi_{n}\left(  \omega\right)  , \quad n\in\mathbb{N}%
_{0}. \label{moment Dirac}%
\end{equation}

Using (\ref{S uva}) in (\ref{ST1}), we obtain%
\begin{equation}
\sigma\left(  t+1\right)  \left[  S_{U}\left(  t+1\right)  -\frac
{M}{t+1-\omega}\right]  -\eta\left(  t\right)  \left[  S_{U}\left(  t\right)
-\frac{M}{t-\omega}\right]  =\xi\left(  t\right)  . \label{STT}%
\end{equation}
Therefore, the Stieltjes transform of $L_{U}$ satisfies the difference
equation%
\[
\left(  t+1-\omega\right)  \left(  t-\omega\right)  \sigma\left(  t+1\right)
S_{U}\left(  t+1\right)  -\left(  t+1-\omega\right)  \left(  t-\omega\right)
\eta\left(  t\right)  S_{U}\left(  t\right)  =\xi_{U}\left(  t\right)  ,
\]
where%
\[
\xi_{U}\left(  t\right)  =\left(  t-\omega\right)  \left(  t-\omega+1\right)
\xi\left(  t\right)  +M\left[  \left(  t-\omega\right)  \sigma\left(
t+1\right)  -\left(  t-\omega+1\right)  \eta\left(  t\right)  \right]  ,
\]
and the modified weight function $\varrho_{U}\left(  x\right)  $ satisfies the
Pearson equation%
\[
\frac{\varrho_{U}\left(  x+1\right)  }{\varrho_{U}\left(  x\right)  }%
=\frac{\left(  x-\omega\right)  \left(  x+1-\omega\right)  \eta\left(
x\right)  }{\left(  x-\omega\right)  \left(  x+1-\omega\right)  \sigma\left(
x+1\right)  }.
\]
We conclude that if $L$ was of class $s,$ then $L_{U}$ will be of class at
most $s+2.$

Notice that we have two special cases to consider:

\begin{enumerate}
\item If $\sigma\left(  \omega\right)  =0$ and $\eta\left(  \omega\right)
\neq0,$ with $\sigma\left(  x\right)  =\left(  x-\omega\right)  \sigma
_{1}\left(  x\right)  ,$ then%
\[
\frac{\varrho_{U}\left(  x+1\right)  }{\varrho_{U}\left(  x\right)  }%
=\frac{\left(  x-\omega\right)  \left(  x+1-\omega\right)  \eta\left(
x\right)  }{\left(  x-\omega\right)  \left(  x+1-\omega\right)  \sigma\left(
x+1\right)  }=\frac{\left(  x-\omega\right)  \eta\left(  x\right)  }{\left(
x-\omega\right)  \sigma\left(  x+1\right)  },
\]%
\begin{equation}
\sigma\left(  t+1\right)  S_{U}\left(  t+1\right)  -\eta\left(  t\right)
S_{U}\left(  t\right)  =\xi\left(  t\right)  +M\left[  \sigma_{1}\left(
t+1\right)  -\frac{\eta\left(  t\right)  }{t-\omega}\right]  ,
\label{reduced1}%
\end{equation}
and $L_{U}$ will be of class $s+1.$

\item If $\sigma\left(  \omega\right)  \neq0$ and $\eta\left(  \omega\right)
=0,$ with $\eta\left(  x\right)  =\left(  x-\omega\right)  \eta_{1}\left(
x\right)  ,$ then%
\[
\frac{\varrho_{U}\left(  x+1\right)  }{\varrho_{U}\left(  x\right)  }%
=\frac{\left(  x-\omega\right)  \left(  x+1-\omega\right)  \eta\left(
x\right)  }{\left(  x-\omega\right)  \left(  x+1-\omega\right)  \sigma\left(
x+1\right)  }=\frac{\left(  x+1-\omega\right)  \eta\left(  x\right)  }{\left(
x+1-\omega\right)  \sigma\left(  x+1\right)  },
\]%
\[
\sigma\left(  t+1\right)  S_{U}\left(  t+1\right)  -\eta\left(  t\right)
S_{U}\left(  t\right)  =\xi\left(  t\right)  +M\left[  \frac{\sigma\left(
t+1\right)  }{t+1-\omega}-\eta_{1}\left(  t\right)  \right]  ,
\]
and $L_{U}$ will be of class $s+1.$
\end{enumerate}

Perturbations of discrete semiclassical functionals by Dirac masses were
studied in \cite{MR2061464}, \cite{MR2064407}, \cite{MR1340932},
\cite{MR1467146}, as well as in \cite{MR1379116}, \cite{MR1353079}, and
\cite{MR1410602}.

\subsubsection{Christoffel transformations}

Let $L$ be given by (\ref{Ldisc}), and consider the functional $L_{C}$ defined
by (\ref{LC}), with
\[
L_{C}\left[  1\right]  =L\left[  x-\omega\right]  \neq0.
\]
It follows that the moments of $L_{C}$ with respect to the basis $\{\phi
_{n}(x)\}_{n\geq0}$ are
\begin{equation}
\nu_{n}^{C}=\nu_{n+1}+\left(  n-\omega\right)  \nu_{n}, \quad n\in
\mathbb{N}_{0} \label{Moments C}%
\end{equation}
since we see from (\ref{phidef}) that%
\begin{equation}
\phi_{n+1}\left(  x\right)  =\left(  x-n\right)  \phi_{n}\left(  x\right)
=\left(  x-\omega\right)  \phi_{n}\left(  x\right)  +\left(  \omega-n\right)
\phi_{n}\left(  x\right)  . \label{req phi}%
\end{equation}

Using (\ref{S Chris}) in (\ref{ST1}), we obtain%
\[
\sigma\left(  t+1\right)  \frac{S_{C}\left(  t+1\right)  +\nu_{0}}{t+1-\omega
}-\eta\left(  t\right)  \frac{S_{C}\left(  t\right)  +\nu_{0}}{t-\omega}%
=\xi\left(  t\right)  .
\]
Therefore, the Stieltjes transform of $L_{C}$ satisfies the difference
equation%
\[
\left(  t-\omega\right)  \sigma\left(  t+1\right)  S_{C}\left(  t+1\right)
-\left(  t+1-\omega\right)  \eta\left(  t\right)  S_{C}\left(  t\right)
=\xi_{C}\left(  t\right)  ,
\]
where%
\[
\xi_{C}\left(  t\right)  =\left(  t-\omega\right)  \left(  t-\omega+1\right)
\xi\left(  t\right)  -\left[  \left(  t-\omega\right)  \sigma\left(
t+1\right)  -\left(  t-\omega+1\right)  \eta\left(  t\right)  \right]  \nu
_{0},
\]
and the modified weight function $\varrho_{C}\left(  x\right)  $ satisfies the
Pearson equation%
\[
\frac{\varrho_{C}\left(  x+1\right)  }{\varrho_{C}\left(  x\right)  }%
=\frac{\left(  x+1-\omega\right)  \eta\left(  x\right)  }{\left(
x-\omega\right)  \sigma\left(  x+1\right)  }.
\]
We conclude that if $L$ was of class $s,$ then $L_{C}$ will be of class at
most $s+1.$

Note that we have two special cases to consider:

\begin{enumerate}
\item If $\sigma\left(  \omega\right)  =0$ and $\eta\left(  \omega\right)
\neq0,$ with $\sigma\left(  x\right)  =\left(  x-\omega\right)  \sigma
_{1}\left(  x\right)  ,$ then%
\[
\frac{\varrho_{C}\left(  x+1\right)  }{\varrho_{C}\left(  x\right)  }%
=\frac{\eta\left(  x\right)  }{\left(  x-\omega\right)  \sigma_{1}\left(
x+1\right)  },
\]%
\[
\left(  t-\omega\right)  \sigma_{1}\left(  t+1\right)  S_{C}\left(
t+1\right)  -\eta\left(  t\right)  S_{C}\left(  t\right)  =\left(
t-\omega\right)  \xi\left(  t\right)  -\left[  \left(  t-\omega\right)
\sigma_{1}\left(  t+1\right)  -\eta\left(  t\right)  \right]  \nu_{0},
\]
and $L_{C}$ could be of class $s.$

\item If $\sigma\left(  \omega\right)  \neq0$ and $\eta\left(  \omega\right)
=0,$ with $\eta\left(  x\right)  =\left(  x-\omega\right)  \eta_{1}\left(
x\right)  ,$ then%
\[
\frac{\varrho_{C}\left(  x+1\right)  }{\varrho_{C}\left(  x\right)  }%
=\frac{\left(  x+1-\omega\right)  \eta_{1}\left(  x\right)  }{\sigma\left(
x+1\right)  },
\]%
\[
\sigma\left(  t+1\right)  S_{C}\left(  t+1\right)  -\left(  t+1-\omega\right)
\eta_{1}\left(  t\right)  S_{C}\left(  t\right)  =\left(  t-\omega+1\right)
\xi\left(  t\right)  -\left[  \sigma\left(  t+1\right)  -\left(
t-\omega+1\right)  \eta_{1}\left(  t\right)  \right]  \nu_{0},
\]
and $L_{C}$ could be of class $s.$
\end{enumerate}

Christoffel transformations in the discrete case have been studied in
\cite{MR1939588}.

\subsubsection{Geronimus transformations}

Let $L$ be given by (\ref{Ldisc}), and consider the functional $L_{G}$ defined
by (\ref{LG}), with%
\[
L_{G}\left[  1\right]  =M-S\left(  \omega\right)  \neq0.
\]
Since
\[
S\left(  \omega\right)  =%
%TCIMACRO{\dsum \limits_{x=0}^{\infty}}%
%BeginExpansion
{\displaystyle\sum\limits_{x=0}^{\infty}}
%EndExpansion
\frac{1}{\omega-x}\frac{\left(  \mathbf{a}\right)  _{x}}{\left(
\mathbf{b}+1\right)  _{x}}\frac{z^{x}}{x!}%
\]
is a meromorphic function, we need $\omega\notin\mathbb{N}_{0}$ for
$L_{G}\left[  1\right]  $ to be well defined.

It follows from (\ref{LG}) that the moments of $L_{G}$ with respect to the
basis $\{\phi_{n}(x)\}_{n\geq0}$ satisfy%
\begin{equation}
\nu_{n+1}^{G}+\left(  n-\omega\right)  \nu_{n}^{G}=\nu_{n}, \quad
n\in\mathbb{N}_{0}, \label{DE MG}%
\end{equation}
since from (\ref{req phi}) we see that
\[
L_{G}\left[  \phi_{n+1}-\left(  \omega-n\right)  \phi_{n}\right]
=L_{G}\left[  \left(  x-\omega\right)  \phi_{n}\left(  x\right)  \right]
=L\left[  \phi_{n}\right]  .
\]
As it is well known, the general solution of the initial value problem%
\[
y_{n+1}=c_{n}y_{n}+g_{n},\quad y_{n_{0}}=y_{0},
\]
is \cite[1.2.4]{MR2128146}%
\[
y_{n}=y_{0}%
%TCIMACRO{\dprod \limits_{j=n_{0}}^{n-1}}%
%BeginExpansion
{\displaystyle\prod\limits_{j=n_{0}}^{n-1}}
%EndExpansion
c_{j}+%
%TCIMACRO{\dsum \limits_{k=n_{0}}^{n-1}}%
%BeginExpansion
{\displaystyle\sum\limits_{k=n_{0}}^{n-1}}
%EndExpansion
\left(  g_{k}%
%TCIMACRO{\dprod \limits_{j=k+1}^{n-1}}%
%BeginExpansion
{\displaystyle\prod\limits_{j=k+1}^{n-1}}
%EndExpansion
c_{j}\right)  .
\]
Solving (\ref{DE MG}) with $\nu_{0}^{G}=M-S\left(  \omega\right)  ,$ we get
\[
\nu_{n}^{G}=\phi_{n}\left(  \omega\right)  \left[  \nu_{0}^{G}+%
%TCIMACRO{\dsum \limits_{k=0}^{n-1}}%
%BeginExpansion
{\displaystyle\sum\limits_{k=0}^{n-1}}
%EndExpansion
\frac{\nu_{k}}{\phi_{k+1}\left(  \omega\right)  }\right]  ,
\]
since from (\ref{phidef}) we have%
\[%
%TCIMACRO{\dprod \limits_{j=0}^{n-1}}%
%BeginExpansion
{\displaystyle\prod\limits_{j=0}^{n-1}}
%EndExpansion
\left(  \omega-j\right)  =\phi_{n}\left(  \omega\right)  ,\quad%
%TCIMACRO{\dprod \limits_{j=k+1}^{n-1}}%
%BeginExpansion
{\displaystyle\prod\limits_{j=k+1}^{n-1}}
%EndExpansion
\left(  \omega-j\right)  =\frac{\phi_{n}\left(  \omega\right)  }{\phi
_{k+1}\left(  \omega\right)  }.
\]

Replacing (\ref{S ger}) in (\ref{ST1}), we obtain%
\[
\sigma\left(  t+1\right)  \left[  \left(  t+1-\omega\right)  S_{G}\left(
t+1\right)  -\nu_{0}^{G}\right]  -\eta\left(  t\right)  \left[  \left(
t-\omega\right)  S_{G}\left(  t\right)  -\nu_{0}^{G}\right]  =\xi\left(
t\right)  .
\]
Therefore, the Stieltjes transform of $L_{G}$ satisfies the difference
equation%
\[
\left(  t+1-\omega\right)  \sigma\left(  t+1\right)  S_{G}\left(  t+1\right)
-\left(  t-\omega\right)  \eta\left(  t\right)  S_{G}\left(  t\right)
=\xi_{G}\left(  t\right)  ,
\]
where%
\[
\xi_{G}\left(  t\right)  =\xi\left(  t\right)  +\left[  \sigma\left(
t+1\right)  -\eta\left(  t\right)  \right]  \nu_{0}^{G},
\]
and the modified weight function $\varrho_{G}\left(  x\right)  $ satisfies the
Pearson equation%
\[
\frac{\varrho_{G}\left(  x+1\right)  }{\varrho_{G}\left(  x\right)  }%
=\frac{\left(  x-\omega\right)  \eta\left(  x\right)  }{\left(  x+1-\omega
\right)  \sigma\left(  x+1\right)  }.
\]
We conclude that if $L$ was of class $s,$ then $L_{G}$ will be of class at
most $s+1.$

\begin{remark}
For all these transformations, it is understood that%
\begin{align*}
\eta\left(  \omega\right)   &  =0\Rightarrow\frac{\eta\left(  x\right)
}{x-\omega}\times\left(  x-\omega\right)  =\eta\left(  x\right)  ,\\
\sigma\left(  \omega\right)   &  =0\Rightarrow\frac{x+1-\omega}{\sigma\left(
x+1\right)  }\times\frac{1}{x+1-\omega}=\frac{1}{\sigma\left(  x+1\right)  }.
\end{align*}
Thus, if we write%
\[
\frac{\varrho_{G}\left(  x+1\right)  }{\varrho_{G}\left(  x\right)  }%
=\frac{\left(  x-\omega\right)  \eta\left(  x\right)  }{\left(  x+1-\omega
\right)  \sigma\left(  x+1\right)  }=\frac{\eta_{G}\left(  x\right)  }%
{\sigma_{G}\left(  x+1\right)  },
\]
then clearly $\eta_{G}\left(  \omega\right)  =0=\sigma_{G}\left(
\omega\right)  .$ Hence,%
\[
\frac{\varrho_{CG}\left(  x+1\right)  }{\varrho_{CG}\left(  x\right)  }%
=\frac{\eta_{G}\left(  x\right)  }{\sigma_{G}\left(  x+1\right)  }\times
\frac{x-\omega}{x+1-\omega}=\frac{\eta\left(  x\right)  }{\sigma\left(
x+1\right)  }=\frac{\varrho\left(  x+1\right)  }{\varrho\left(  x\right)  },
\]
in agreement with (\ref{CG}).

On the other hand, if we write%
\[
\frac{\varrho_{C}\left(  x+1\right)  }{\varrho_{C}\left(  x\right)  }%
=\frac{\left(  x+1-\omega\right)  \eta\left(  x\right)  }{\left(
x-\omega\right)  \sigma\left(  x+1\right)  }=\frac{\eta_{C}\left(  x\right)
}{\sigma_{C}\left(  x+1\right)  },
\]
then $\eta_{G}\left(  \omega\right)  \sigma_{G}\left(  \omega\right)  \neq0,$
and therefore%
\begin{align*}
\frac{\varrho_{GC}\left(  x+1\right)  }{\varrho_{GC}\left(  x\right)  }  &
=\frac{\eta_{C}\left(  x\right)  }{\sigma_{C}\left(  x+1\right)  }\times
\frac{x+1-\omega}{x-\omega}\\
&  =\frac{\left(  x-\omega\right)  \left(  x+1-\omega\right)  \eta\left(
x\right)  }{\left(  x-\omega\right)  \left(  x+1-\omega\right)  \sigma\left(
x+1\right)  }=\frac{\varrho_{U}\left(  x+1\right)  }{\varrho_{U}\left(
x\right)  },
\end{align*}
in agreement with (\ref{GC}).
\end{remark}

\subsection{Truncated linear functionals}

Let $N\in\mathbb{N}_{0}$ and define the \emph{truncated functional} $L_{T}$ of
the linear functional $L$ by
\begin{equation}
L_{T}\left[  r\right]  =%
%TCIMACRO{\dsum \limits_{x=0}^{N}}%
%BeginExpansion
{\displaystyle\sum\limits_{x=0}^{N}}
%EndExpansion
r\left(  x\right)  \varrho\left(  x\right)  ,\quad r\in\mathbb{C}\left[
x\right]  , \label{L trunc}%
\end{equation}
where $\varrho\left(  x\right)  $ satisfies (\ref{pearson}) with $\eta\left(
N\right)  \neq0.$ We define the truncated weight function $\varrho_{T}\left(
x\right)  $ by%
\begin{equation}
\varrho_{T}\left(  x\right)  =\varrho\left(  x\right)  \left[  1-H\left(
x-N\right)  \right]  , \label{rho trunc}%
\end{equation}
where $H\left(  x\right)  $ is the Heaviside function \cite[1.16(iv)]%
{MR2723248} defined by%
\[
H\left(  x\right)  =\left\{
\begin{array}
[c]{c}%
1,\quad x>0\\
0,\quad x\leq0
\end{array}
\right.  .
\]
Multiplying (\ref{rho trunc}) by $\frac{1}{\left(  N-x\right)  !}$, we obtain%
\begin{equation}
\frac{1}{\left(  N-x\right)  !}\varrho_{T}\left(  x\right)  =\frac{1}{\left(
N-x\right)  !}\varrho\left(  x\right)  ,\quad x\in\mathbb{N}_{0}.
\label{rho-trunc1}%
\end{equation}

Using (\ref{rho-trunc1}) in (\ref{pearson}), we get%
\[
\left(  N-x\right)  \frac{\varrho_{T}\left(  x+1\right)  }{\varrho_{T}\left(
x\right)  }=\frac{\frac{1}{\left(  N-x-1\right)  !}\varrho_{T}\left(
x+1\right)  }{\frac{1}{\left(  N-x\right)  !}\varrho_{T}\left(  x\right)
}=\frac{\frac{1}{\left(  N-x-1\right)  !}\varrho\left(  x+1\right)  }{\frac
{1}{\left(  N-x\right)  !}\varrho\left(  x\right)  }=\left(  N-x\right)
\frac{\eta\left(  x\right)  }{\sigma\left(  x+1\right)  }.
\]
We conclude that the modified weight function $\varrho_{T}\left(  x\right)  $
satisfies the Pearson equation%
\begin{equation}
\frac{\varrho_{T}\left(  x+1\right)  }{\varrho_{T}\left(  x\right)  }%
=\frac{\eta\left(  x\right)  \left(  x-N\right)  }{\sigma\left(  x+1\right)
\left(  x-N\right)  }, \label{Pearson trunc}%
\end{equation}
and $L_{T}$ will be of class $s+1.$

The moments of $L_{T}$ with respect to the basis $\{\phi_{n}(x)\}_{n\geq0}$
are
\[
\nu_{n}^{T}\left(  z\right)  =%
%TCIMACRO{\dsum \limits_{x=0}^{N}}%
%BeginExpansion
{\displaystyle\sum\limits_{x=0}^{N}}
%EndExpansion
\phi_{n}\left(  x\right)  \frac{\left(  \mathbf{a}\right)  _{x}}{\left(
\mathbf{b}+1\right)  _{x}}\frac{z^{x}}{x!}, \quad n\in\mathbb{N}_{0}.
\]
A similar calculation leading to (\ref{moment}), gives%
\begin{equation}
\nu_{n}^{T}\left(  z\right)  =z^{n}\frac{\left(  \mathbf{a}\right)  _{n}%
\ }{\left(  \mathbf{b}+1\right)  _{n}}%
%TCIMACRO{\dsum \limits_{x=0}^{N-n}}%
%BeginExpansion
{\displaystyle\sum\limits_{x=0}^{N-n}}
%EndExpansion
\frac{\left(  \mathbf{a}+n\right)  _{x}}{\left(  \mathbf{b}+1+n\right)  _{x}%
}\frac{z^{x}}{x!}, \quad n\in\mathbb{N}_{0}. \label{trunc moment}%
\end{equation}
Using the identity \cite[16.2.4]{MR2723248}%
\[%
%TCIMACRO{\dsum \limits_{x=0}^{K}}%
%BeginExpansion
{\displaystyle\sum\limits_{x=0}^{K}}
%EndExpansion
\frac{\left(  \mathbf{a}\right)  _{x}}{\left(  \mathbf{b}+1\right)  _{x}}%
\frac{z^{x}}{x!}=\frac{\left(  \mathbf{a}\right)  _{K}\ }{\left(
\mathbf{b}+1\right)  _{K}}\frac{z^{K}}{K!}\ _{q+2}F_{p}\left[
\begin{array}
[c]{c}%
-K,1,\ -K-\mathbf{b}\\
1-K-\ \mathbf{a}%
\end{array}
;\frac{\left(  -1\right)  ^{p+q+1}}{z}\right]  ,\quad K\in\mathbb{N},
\]
we can rewrite (\ref{trunc moment}) as%
\begin{equation}
\nu_{n}^{T}\left(  z\right)  =\frac{\left(  \mathbf{a}\right)  _{N}\ }{\left(
\mathbf{b}+1\right)  _{N}}\frac{z^{N}}{\left(  N-n\right)  !}\ _{q+2}%
F_{p}\left[
\begin{array}
[c]{c}%
n-N,1,\ -N-\mathbf{b}\\
1-N-\ \mathbf{a}%
\end{array}
;\frac{\left(  -1\right)  ^{p+q+1}}{z}\right]  . \label{moment truncated}%
\end{equation}

\subsection{Symmetrized functionals}

Let $\varrho\left(  x\right)  $ be given by (\ref{rho}) and $m\in\mathbb{N}$.
We consider the symmetric weight function $\varrho_{\varsigma}\left(
x\right)  $ defined by%
\begin{equation}
\varrho_{\varsigma}\left(  x\right)  =C\varrho\left(  x+m\right)
\varrho\left(  -x+m\right)  , \label{rho symm}%
\end{equation}
where $C$ is a constant factor to be determined. Replacing (\ref{pearson}) in
(\ref{rho symm}), we get
\[
\frac{\varrho_{\varsigma}\left(  x+1\right)  }{\varrho_{\varsigma}\left(
x\right)  }=\frac{\varrho\left(  x+1+m\right)  }{\varrho\left(  x+m\right)
}\frac{\varrho\left(  -x-1+m\right)  }{\varrho\left(  -x+m\right)  }%
=\frac{\eta\left(  x+m\right)  }{\sigma\left(  x+m+1\right)  }\frac
{\sigma\left(  -x+m\right)  }{\eta\left(  -x-1+m\right)  }.
\]
Therefore,%
\[
\frac{\varrho_{\varsigma}\left(  x+1\right)  }{\varrho_{\varsigma}\left(
x\right)  }=\frac{x+m+\mathbf{a}}{\left(  x+m+\mathbf{b}+1\right)  \left(
x+m+1\right)  }\frac{\left(  m+\mathbf{b}-x\right)  \left(  m-x\right)
}{m+\mathbf{a}-x-1},
\]
or%
\begin{equation}
\frac{\varrho_{\varsigma}\left(  x+1\right)  }{\varrho_{\varsigma}\left(
x\right)  }=\frac{\left(  -1\right)  ^{p+q+1}\left(  x+m+\mathbf{a}\right)
\left(  x-m-\mathbf{b}\right)  \left(  x-m\right)  }{\left(  x+1-m-\mathbf{a}%
\right)  \left(  x+1+m+\mathbf{b}\right)  \left(  x+1+m\right)  }.
\label{rho pear symm}%
\end{equation}

Note that, from (\ref{rho pear symm}), we have $\varrho_{\varsigma}\left(
x\right)  =\varrho_{1}\left(  x+m\right)  ,$ with%
\[
\frac{\varrho_{1}\left(  x+1\right)  }{\varrho_{1}\left(  x\right)  }%
=\frac{\left(  -1\right)  ^{p+q+1}\left(  x+\mathbf{a}\right)  \left(
x-2m-\mathbf{b}\right)  \left(  x-2m\right)  }{\left(  x+1-2m-\mathbf{a}%
\right)  \left(  x+1+\mathbf{b}\right)  \left(  x+1\right)  }.
\]
If we write $N=2m,$ then%
\[
\frac{\varrho_{1}\left(  x+1\right)  }{\varrho_{1}\left(  x\right)  }%
=\frac{\left(  -1\right)  ^{p+q+1}\left(  x+\mathbf{a}\right)  \left(
x-N-\mathbf{b}\right)  \left(  x-N\right)  }{\left(  x+1-N-\mathbf{a}\right)
\left(  x+1+\mathbf{b}\right)  \left(  x+1\right)  },
\]
and it follows that $\varrho_{1}\left(  x\right)  $ (and $\varrho_{\varsigma
})$ is of class $s_{1}$ with%
\[
s_{1}=\left\{
\begin{array}
[c]{c}%
p+q,\quad\text{if }q+p\text{ is even}\\
p+q-1,\quad\text{if }q+p\text{ is odd}%
\end{array}
\right.  .
\]

If $\mathbf{a}=\left(  \mathbf{a}_{1},-N\right)  ,$ then%
\[
\frac{\varrho_{1}\left(  x+1\right)  }{\varrho_{1}\left(  x\right)  }%
=\frac{\left(  -1\right)  ^{q+1+p}\left(  x+\mathbf{a}_{1}\right)  }{\left(
x+1+\mathbf{b}\right)  \left(  x+1\right)  }\frac{x-N}{x+1}\frac{\left(
x-N-\mathbf{b}\right)  \left(  x-N\right)  }{x+1-N-\mathbf{a}_{1}},
\]
and since the term $\frac{x-N}{x+1}$ is repeated, we consider the reduced
Pearson equation%
\begin{align*}
\frac{\varrho_{2}\left(  x+1\right)  }{\varrho_{2}\left(  x\right)  }  &
=\frac{\left(  -1\right)  ^{q+1+p}\left(  x+\mathbf{a}_{1}\right)  }{\left(
x+1+\mathbf{b}\right)  \left(  x+1\right)  }\frac{\left(  x-N-\mathbf{b}%
\right)  \left(  x-N\right)  }{x+1-N-\mathbf{a}_{1}}\\
&  =\frac{\left(  -1\right)  ^{q+1+p}\left(  x+\mathbf{a}\right)  }{\left(
x+1+\mathbf{b}\right)  }\frac{\left(  x-N-\mathbf{b}\right)  }%
{x+1-N-\mathbf{a}},
\end{align*}
and $\varrho_{\varsigma}\left(  x\right)  =\varrho_{2}\left(  x+m\right)  .$
In this case $\varrho_{2}\left(  x\right)  $ (and $\varrho_{\varsigma})$ is of
class $s_{2}$ with%
\[
s_{2}=\left\{
\begin{array}
[c]{c}%
p+q-1,\quad\text{if }q+p\text{ is even}\\
p+q-2,\quad\text{if }q+p\text{ is odd}%
\end{array}
\right.  .
\]
As a conclusion, if $\varrho\left(  x\right)  $ satisfies (\ref{pearson}),
then we have $\varrho_{\varsigma}\left(  x\right)  =\rho\left(  x+m\right)  $
where%
\[
\rho\left(  x\right)  =\left\{
\begin{array}
[c]{c}%
\dfrac{\left(  \mathbf{a}\right)  _{x}}{\left(  \mathbf{b}+1\right)  _{x}%
}\dfrac{\left(  z_{0}\right)  ^{x}}{x!}\dfrac{\left(  -N\right)  _{x}\left(
-N-\mathbf{b}\right)  _{x}}{\left(  1-N-\mathbf{a}\right)  _{x}},\quad
\eta\left(  N\right)  \neq0,\\
\dfrac{\left(  \mathbf{a}\right)  _{x}}{\left(  \mathbf{b}+1\right)  _{x}%
}\left(  z_{0}\right)  ^{x}\dfrac{\left(  -N-\mathbf{b}\right)  _{x}}{\left(
1-N-\mathbf{a}\right)  _{x}},\quad\eta\left(  N\right)  =0,
\end{array}
\right.
\]
and%
\[
z_{0}=\left(  -1\right)  ^{p+q+1},\quad N=2m.
\]

We define the \emph{symmetrized functional} $L_{\varsigma}$ of $L_{\rho}$ by%
\[
L_{\varsigma}\left[  r\right]  =%
%TCIMACRO{\dsum \limits_{x=-m}^{m}}%
%BeginExpansion
{\displaystyle\sum\limits_{x=-m}^{m}}
%EndExpansion
r\left(  x\right)  \varrho_{\varsigma}\left(  x\right)  =%
%TCIMACRO{\dsum \limits_{x=0}^{N}}%
%BeginExpansion
{\displaystyle\sum\limits_{x=0}^{N}}
%EndExpansion
r\left(  x-m\right)  \rho\left(  x\right)  =L_{\rho}\left[  r\left(
x-m\right)  \right]  .
\]
Hence, the moments of $L_{\varsigma}$ on the basis
\[
\overline{\phi}_{n}\left(  x\right)  =\phi_{n}\left(  x+m\right)  , \quad
n\in\mathbb{N}_{0},
\]
are
\begin{equation}
\nu_{n}^{\varsigma}=%
%TCIMACRO{\dsum \limits_{x=-m}^{m}}%
%BeginExpansion
{\displaystyle\sum\limits_{x=-m}^{m}}
%EndExpansion
\overline{\phi}_{n}\left(  x\right)  \varrho_{\varsigma}\left(  x\right)  =%
%TCIMACRO{\dsum \limits_{x=0}^{N}}%
%BeginExpansion
{\displaystyle\sum\limits_{x=0}^{N}}
%EndExpansion
\phi_{n}\left(  x\right)  \rho\left(  x\right)  =L_{\rho}\left[  \phi
_{n}\right]  , \quad n\in\mathbb{N}_{0} , \label{moments symm}%
\end{equation}
and the Stieltjes transform of $L_{\varsigma}$ is
\[
S_{\varsigma}(t)=%
%TCIMACRO{\dsum \limits_{x=-m}^{m}}%
%BeginExpansion
{\displaystyle\sum\limits_{x=-m}^{m}}
%EndExpansion
\frac{\varrho_{\varsigma}\left(  x\right)  }{t-x}=%
%TCIMACRO{\dsum \limits_{x=0}^{N}}%
%BeginExpansion
{\displaystyle\sum\limits_{x=0}^{N}}
%EndExpansion
\frac{\rho\left(  x\right)  }{t+m-x}=S_{\rho}(t+m).
\]

It follows that if the Stieltjes transform of $L_{\rho}$ satisfies the
difference equation%
\[
\sigma_{\rho}\left(  t+1\right)  S_{\rho}\left(  t+1\right)  -\eta_{\rho
}\left(  t\right)  S_{\rho}\left(  t\right)  =\xi_{\rho}\left(  t\right)  ,
\]
then
\[
\sigma_{\rho}\left(  t+m+1\right)  S_{\rho}\left(  t+m+1\right)  -\eta_{\rho
}\left(  t+m\right)  S_{\rho}(t+m)=\xi_{\rho}\left(  t+m\right)  .
\]
Therefore%
\begin{equation}
\sigma_{\varsigma}\left(  t+1\right)  S_{\varsigma}\left(  t+1\right)
-\eta_{\varsigma}\left(  t\right)  S_{\varsigma}(t)=\xi_{\rho}\left(
t+m\right)  , \label{DE S  symm}%
\end{equation}
since%
\[
\frac{\eta_{\varsigma}\left(  x\right)  }{\sigma_{\varsigma}\left(
x+1\right)  }=\frac{\varrho_{\varsigma}\left(  x+1\right)  }{\varrho
_{\varsigma}\left(  x\right)  }=\frac{\rho\left(  x+m+1\right)  }{\rho\left(
x+m\right)  }=\frac{\eta_{\rho}\left(  x+m\right)  }{\sigma_{\rho}\left(
x+m+1\right)  }.
\]

Symmetric orthogonal polynomials of a discrete variable were studied in
\cite{MR3360482}, \cite{MR2033351}, \cite{MR2993855}, \cite{MR3173503}.

\subsection{\strut Summary}

Let's list all the transformations that we will use in the sequel. For
simplicity, we will not consider compositions of them.

\begin{enumerate}
\item Uvarov transformation$:s\rightarrow s+2$%
\[
\frac{\varrho_{U}\left(  x+1\right)  }{\varrho_{U}\left(  x\right)  }%
=\frac{\eta\left(  x\right)  }{\sigma\left(  x+1\right)  }\times\frac{\left(
x-\omega\right)  \left(  x+1-\omega\right)  }{\left(  x-\omega\right)  \left(
x+1-\omega\right)  },\quad\eta\left(  \omega\right)  \sigma\left(
\omega\right)  \neq0.
\]

\item Reduced Uvarov transformation$:s\rightarrow s+1$

\begin{description}
\item[(a)]
\[
\frac{\varrho_{U}\left(  x+1\right)  }{\varrho_{U}\left(  x\right)  }%
=\frac{\eta\left(  x\right)  }{\sigma\left(  x+1\right)  }\times
\frac{x+1-\omega}{x+1-\omega},\quad\eta\left(  \omega\right)  =0.
\]

\item[(b)]
\[
\frac{\varrho_{U}\left(  x+1\right)  }{\varrho_{U}\left(  x\right)  }%
=\frac{\eta\left(  x\right)  }{\sigma\left(  x+1\right)  }\times\frac
{x-\omega}{x-\omega},\quad\sigma\left(  \omega\right)  =0.
\]

\end{description}

\item Christoffel transformation: $s\rightarrow s+1$%
\[
\frac{\varrho_{C}\left(  x+1\right)  }{\varrho_{C}\left(  x\right)  }%
=\frac{\eta\left(  x\right)  }{\sigma\left(  x+1\right)  }\times
\frac{x+1-\omega}{x-\omega},\quad\eta\left(  \omega\right)  \sigma\left(
\omega\right)  \neq0.
\]

\item Geronimus transformation: $s\rightarrow s+1$%
\[
\frac{\varrho_{G}\left(  x+1\right)  }{\varrho_{G}\left(  x\right)  }%
=\frac{\eta\left(  x\right)  }{\sigma\left(  x+1\right)  }\times\frac
{x-\omega}{x+1-\omega},\quad\eta\left(  \omega-1\right)  \sigma\left(
\omega+1\right)  \neq0.
\]

\item Truncation at $x=N:$ $s\rightarrow s+1$%
\[
\frac{\varrho_{T}\left(  x+1\right)  }{\varrho_{T}\left(  x\right)  }%
=\frac{\eta\left(  x\right)  }{\sigma\left(  x+1\right)  }\times\frac
{x-N}{x-N},\quad\eta\left(  N\right)  \neq0.
\]

\item Symmetrization on the interval $\left[  -m,m\right]  :$

\begin{description}
\item[(a)]
\[
\frac{\varrho_{\varsigma}\left(  x+1\right)  }{\varrho_{\varsigma}\left(
x\right)  }=\frac{\eta\left(  m+x\right)  }{\sigma\left(  m+x+1\right)
}\times\frac{\sigma\left(  m-x\right)  }{\eta\left(  m-x-1\right)  },\quad
\eta\left(  2m\right)  \neq0.
\]

\item[(b)]
\begin{equation}
\frac{\varrho_{\varsigma}\left(  x+1\right)  }{\varrho_{\varsigma}\left(
x\right)  }=\frac{\eta\left(  m+x\right)  }{\sigma\left(  m+x+1\right)
}\times\frac{\sigma_{1}\left(  m-x\right)  }{\eta_{1}\left(  m-x-1\right)
},\quad\eta\left(  2m\right)  =0, \label{symm2}%
\end{equation}
where%
\[
\eta\left(  x\right)  =\left(  x-2m\right)  \eta_{1}\left(  x\right)
,\quad\sigma\left(  x\right)  =x\sigma_{1}\left(  x\right)  .
\]

\end{description}
\end{enumerate}

\section{\strut{\protect\LARGE Semiclassical polynomials of class 0
(}{\protect\Large classical polynomials)}}

In this section, we consider the families of discrete classical polynomials.
We have $3$ main cases, corresponding to $\left(  p,q\right)  =\left(
0,0\right)  ,\left(  1,0\right)  ,\left(  2,1\right)  .$ There are also $3$
symmetrized subcases.

We use the notation $\left(  p,q;N\right)  $ to denote the family such that
one of the parameters in the numerator is a non-negative integer, and $\left(
p,q;N,1\right)  $ if in addition the value of $z$ is equal to $1.$

For each polynomial, we list the linear functional, the Pearson equation
satisfied by the weight function, the moments computed from (\ref{moment}),
and the difference equation satisfied by the Stieltjes transform, using
(\ref{xi}) and (\ref{DE S symm}).

\subsection{$(0,0):$ Charlier polynomials}

\strut Linear functional%
\[
L\left[  r\right]  =%
%TCIMACRO{\dsum \limits_{x=0}^{\infty}}%
%BeginExpansion
{\displaystyle\sum\limits_{x=0}^{\infty}}
%EndExpansion
r\left(  x\right)  \frac{z^{x}}{x!}.
\]

Pearson equation%
\[
\frac{\varrho\left(  x+1\right)  }{\varrho\left(  x\right)  }=\frac{z}{x+1}.
\]

Moments%
\[
\nu_{n}\left(  z\right)  =z^{n}\ _{0}F_{0}\left[
\begin{array}
[c]{c}%
-\\
-
\end{array}
;z\right]  =z^{n}e^{z}, \quad n\in\mathbb{N}_{0}.
\]

Stieltjes transform difference equation
\[
\left(  t+1\right)  S\left(  t+1\right)  -zS\left(  t\right)  =\nu_{0}.
\]

\subsection{$(1,0):$ Meixner polynomials}

\strut Linear functional%
\[
L\left[  r\right]  =%
%TCIMACRO{\dsum \limits_{x=0}^{\infty}}%
%BeginExpansion
{\displaystyle\sum\limits_{x=0}^{\infty}}
%EndExpansion
r\left(  x\right)  \ \left(  a\right)  _{x}\frac{z^{x}}{x!}.
\]

Pearson equation%
\[
\frac{\varrho\left(  x+1\right)  }{\varrho\left(  x\right)  }=\frac{z\left(
x+a\right)  }{x+1}.
\]

Moments%
\[
\nu_{n}\left(  z\right)  =z^{n}\left(  a\right)  _{n}\ _{1}F_{0}\left[
\begin{array}
[c]{c}%
a+n\\
-
\end{array}
;z\right]  =z^{n}\left(  a\right)  _{n}\left(  1-z\right)  ^{-a-n}, \quad
n\in\mathbb{N}_{0},
\]
where we choose the principal branch $z\in\mathbb{C}\setminus\lbrack
1,\infty).$

Stieltjes transform difference equation
\[
\left(  t+1\right)  S\left(  t+1\right)  -z\left(  t+a\right)  S(t)=\left(
1-z\right)  \nu_{0}.
\]

\subsection{$(1,0;N):$ Krawtchouk polynomials}

These polynomials are a particular case of the Meixner polynomials, with
$-a=N\in\mathbb{N}.$

Linear functional%
\[
L\left[  r\right]  =%
%TCIMACRO{\dsum \limits_{x=0}^{N}}%
%BeginExpansion
{\displaystyle\sum\limits_{x=0}^{N}}
%EndExpansion
r\left(  x\right)  \ \left(  -N\right)  _{x}\frac{z^{x}}{x!}.
\]

Pearson equation%
\[
\frac{\varrho\left(  x+1\right)  }{\varrho\left(  x\right)  }=\frac{z\left(
x-N\right)  }{x+1}.
\]

Moments%
\begin{equation}
\nu_{n}\left(  z\right)  =z^{n}\left(  -N\right)  _{n}\left(  1-z\right)
^{N-n},\quad z\neq1, \quad n\in\mathbb{N}_{0}. \label{moment Kraw}%
\end{equation}

Stieltjes transform difference equation
\[
\left(  t+1\right)  S\left(  t+1\right)  -z\left(  t-N\right)  S(t)=\left(
1-z\right)  \nu_{0}.
\]

\begin{remark}
Let's consider the symmetrized Krawtchouk polynomials. Since $\eta\left(
N\right)  =0,$ we use (\ref{symm2}) and obtain%
\[
\frac{\varrho_{\varsigma}\left(  x+1\right)  }{\varrho_{\varsigma}\left(
x\right)  }=\frac{z\left(  x-m\right)  }{x+m+1}\frac{\left(  x+m+1\right)
\left(  -x+m\right)  }{\left(  x-m\right)  z\left(  -x-1-m\right)  }%
=\frac{x-m}{x+m+1}.
\]
Hence, the symmetrized Krawtchouk polynomials are shifted Krawtchouk
polynomials with $z=1.$ But from (\ref{moments symm}) and (\ref{moment Kraw})
we see that%
\[
\nu_{n}^{\varsigma}=\nu_{n}\left(  1\right)  =0,\quad n\in\mathbb{N}_{0},
\]
and, therefore, we need to discard this example.
\end{remark}

\subsubsection{Symmetrized Charlier polynomials}

Special values%
\[
N\rightarrow2m,\quad z\rightarrow-1,\quad x\rightarrow x+m.
\]

Weight function%
\begin{align*}
\varrho_{\varsigma}\left(  x\right)   &  =\varrho\left(  x+m\right)  =\left(
-2m\right)  _{x+m}\frac{\left(  -1\right)  ^{x+m}}{\left(  x+m\right)  !}\\
&  =\left(  -2m\right)  _{m}\frac{\left(  -1\right)  ^{m}}{m!}\frac{\left(
-m\right)  _{x}}{\left(  m+1\right)  _{x}}\left(  -1\right)  ^{x}=\binom
{2m}{m}\frac{\left(  -m\right)  _{x}}{\left(  m+1\right)  _{x}}\left(
-1\right)  ^{x}.
\end{align*}

Linear functional%
\[
L_{\varsigma}\left[  r\right]  =\binom{2m}{m}%
%TCIMACRO{\dsum \limits_{x=-m}^{m}}%
%BeginExpansion
{\displaystyle\sum\limits_{x=-m}^{m}}
%EndExpansion
r\left(  x\right)  \ \frac{\left(  -m\right)  _{x}}{\left(  m+1\right)  _{x}%
}\left(  -1\right)  ^{x}.
\]

Pearson equation%
\begin{equation}
\frac{\varrho_{\varsigma}\left(  x+1\right)  }{\varrho_{\varsigma}\left(
x\right)  }=\frac{-\left(  x-m\right)  }{x+m+1}. \label{PE1}%
\end{equation}

Moments on the basis $\{\overline{\phi}_{n}\left(  x\right)  \}_{n\geq0} $%
\begin{equation}
\nu_{n}^{\varsigma}=L_{\varsigma}\left[  \phi\left(  x+m\right)  \right]
=\left(  -1\right)  ^{n}\left(  -2m\right)  _{n}\ 2^{2m-n}=\nu_{n}\left(
-1\right)  , \quad n\in\mathbb{N}_{0}, \label{Mom1}%
\end{equation}
where $\nu_{n}\left(  z\right)  $ are the moments of the Krawtchouk
polynomials defined in (\ref{moment Kraw}).

Stieltjes transform difference equation%
\[
\left(  t+m+1\right)  S_{\varsigma}\left(  t+1\right)  +\left(  t-m\right)
S_{\varsigma}(t)=2\nu_{0}^{\varsigma}.
\]

\begin{remark}
These polynomials were studied in \cite{MR2033351}. The Pearson equation
(\ref{PE1}) is the same as equation (11) in that paper, with $m=c.$ The
authors used the weight function%
\[
\rho\left(  x\right)  =\frac{1}{\Gamma\left(  x+m+1\right)  \Gamma\left(
-x+m+1\right)  }=\frac{1}{\left(  m!\right)  ^{2}}\frac{\left(  -m\right)
_{x}}{\left(  m+1\right)  _{x}}\left(  -1\right)  ^{x},
\]
and, therefore,
\begin{equation}
\varrho_{\varsigma}\left(  x\right)  =\left(  2m\right)  !\rho\left(
x\right)  . \label{rho1}%
\end{equation}
The moments (\ref{Mom1}) are the same as those appearing in equation (15), if
we use (\ref{equal}) and write%
\[
\nu_{n}^{\varsigma}=\left(  -1\right)  ^{n}\left(  -2m\right)  _{n}%
\ 2^{2m-n}=2^{2m-n}\phi_{n}\left(  2m\right)  ,
\]
after taking the scaling (\ref{rho1}) into account.
\end{remark}

\subsection{$(2,1;N,1):$ Hahn polynomials}

Linear functional%
\[
L\left[  r\right]  =%
%TCIMACRO{\dsum \limits_{x=0}^{N}}%
%BeginExpansion
{\displaystyle\sum\limits_{x=0}^{N}}
%EndExpansion
r\left(  x\right)  \ \frac{\left(  a\right)  _{x}\left(  -N\right)  _{x}%
}{\left(  b+1\right)  _{x}}\frac{1}{x!}.
\]

Pearson equation%
\[
\frac{\eta\left(  x\right)  }{\sigma\left(  x+1\right)  }=\frac{\varrho\left(
x+1\right)  }{\varrho\left(  x\right)  }=\frac{\left(  x+a\right)  \left(
x-N\right)  }{\left(  x+b+1\right)  \left(  x+1\right)  }.
\]

Moments%
\[
\nu_{n}=\frac{\left(  -N\right)  _{n}\left(  a\right)  _{n}}{\left(
b+1\right)  _{n}}\ _{2}F_{1}\left(
\begin{array}
[c]{c}%
-N+n,a+n\\
b+1+n
\end{array}
;1\right)  , n\in\mathbb{N}_{0} .
\]
Using the Chu--Vandermonde Identity \cite[15.4.24]{MR2723248},
\[
\ _{2}F_{1}\left(
\begin{array}
[c]{c}%
-n,b\\
c
\end{array}
;1\right)  =\frac{\left(  c-b\right)  _{n}}{\left(  c\right)  _{n}},
\]
we get%
\[
\nu_{n}=\frac{\left(  -N\right)  _{n}\left(  a\right)  _{n}}{\left(
b+1\right)  _{n}}\frac{\left(  b+1-a\right)  _{N-n}}{\left(  b+1+n\right)
_{N-n}}, \quad n\in\mathbb{N}_{0}.
\]

Stieltjes transform difference equation%
\[
\left(  t+1\right)  \left(  t+b+1\right)  S\left(  t+1\right)  -\left(
t-N\right)  \left(  t+a\right)  S\left(  t\right)  =\left(  b+1-a+N\right)
\nu_{0}.
\]

\subsubsection{Symmetrized Meixner polynomials}

Special values%
\[
b\rightarrow-N-a,\quad N\rightarrow2m,\quad x\rightarrow x+m.
\]

Linear functional%
\[
L_{\varsigma}\left[  r\right]  =\binom{2m}{m}\frac{\left(  a\right)  _{m}%
}{\left(  a+m\right)  _{m}}%
%TCIMACRO{\dsum \limits_{x=-m}^{m}}%
%BeginExpansion
{\displaystyle\sum\limits_{x=-m}^{m}}
%EndExpansion
r\left(  x\right)  \ \frac{\left(  a+m\right)  _{x}\left(  -m\right)  _{x}%
}{\left(  -m-a+1\right)  _{x}\left(  m+1\right)  _{x}}.
\]

Pearson equation%
\begin{equation}
\frac{\varrho_{\varsigma}\left(  x+1\right)  }{\varrho_{\varsigma}\left(
x\right)  }=\frac{\left(  x+a+m\right)  \left(  x-m\right)  }{\left(
x+1-m-a\right)  \left(  x+1+m\right)  }. \label{symm Meixner}%
\end{equation}

Moments%
\[
\nu_{n}^{\varsigma}=\frac{\left(  -N\right)  _{n}\left(  a\right)  _{n}%
}{\left(  1-N-a\right)  _{n}}\frac{\left(  1-N-2a\right)  _{N-n}}{\left(
n+1-N-a\right)  _{N-n}}, \quad n\in\mathbb{N}_{0}.
\]

Stieltjes transform difference equation%
\[
\left(  t+m+1\right)  \left(  t-m-a+1\right)  S_{\varsigma}\left(  t+1\right)
-\left(  t-m\right)  \left(  t+m+a\right)  S_{\varsigma}(t)=\left(
1-2a\right)  \nu_{0}^{\varsigma}.
\]

\begin{remark}
If we use (\ref{Gamma}) and (\ref{rho symm}) to write the weight function
$\varrho_{\varsigma}\left(  x\right)  $ in terms of Gamma functions, we have
\[
\varrho_{\varsigma}\left(  x\right)  =\varrho_{\varsigma}\left(  0\right)
\left[  \frac{m!}{\left(  a\right)  _{m}}\right]  ^{2}\frac{\left(  a\right)
_{x+m}}{\left(  x+m\right)  !}\frac{\left(  a\right)  _{-x+m}}{\left(
-x+m\right)  !}=\frac{\left(  2m\right)  !}{\left(  a\right)  _{2m}}%
\frac{\Gamma\left(  x+m+a\right)  \Gamma\left(  -x+m+a\right)  }{\Gamma\left(
x+m+1\right)  \Gamma\left(  -x+m+1\right)  }.
\]
This agrees (up to a normalization factor) with the weight function considered
by the authors in \cite{MR2033351} (equation 27), if we set $\delta_{1}=-m$
and $\delta_{2}=m+a.$ Note that the condition $\delta_{2}>-\delta_{1}$ is
satisfied if $a>0$ (the positive-definite case).
\end{remark}

\subsubsection{Symmetrized generalized Charlier polynomials}

See section \ref{GenChar} for a definition of the Generalized Charlier polynomials.

Special values%
\[
a\rightarrow-N-b,\quad N\rightarrow2m,\quad x\rightarrow x+m.
\]

Linear functional%
\[
L_{\varsigma}\left[  r\right]  =\binom{2m}{m}\frac{\left(  b+1+m\right)  _{m}%
}{\left(  b+1\right)  _{m}}%
%TCIMACRO{\dsum \limits_{x=-m}^{m}}%
%BeginExpansion
{\displaystyle\sum\limits_{x=-m}^{m}}
%EndExpansion
r\left(  x\right)  \ \frac{\left(  -m-b\right)  _{x}\left(  -m\right)  _{x}%
}{\left(  b+1+m\right)  _{x}\left(  m+1\right)  _{x}}.
\]

Pearson equation%
\[
\frac{\varrho_{\varsigma}\left(  x+1\right)  }{\varrho_{\varsigma}\left(
x\right)  }=\frac{\left(  x-m-b\right)  \left(  x-m\right)  }{\left(
x+m+b+1\right)  \left(  x+m+1\right)  }.
\]

\strut Moments%
\[
\nu_{n}^{\varsigma}=\frac{\left(  -2m\right)  _{n}\left(  -2m-b\right)  _{n}%
}{\left(  b+1\right)  _{n}}\frac{\left(  2m+2b+1\right)  _{2m-n}}{\left(
b+1+n\right)  _{2m-n}}, \quad n\in\mathbb{N}_{0}.
\]

Stieltjes transform difference equation%
\[
\left(  t+m+1\right)  \left(  t+m+b+1\right)  S_{\varsigma}\left(  t+1\right)
-\left(  t-m\right)  \left(  t-m-b\right)  S_{\varsigma}(t)=\left(
2b+1+4m\right)  \nu_{0}^{\varsigma}.
\]

\begin{remark}
If we use (\ref{Gamma}) and (\ref{rho symm}) to write the weight function
$\varrho_{\varsigma}\left(  x\right)  $ in terms of Gamma functions, we have
\begin{align*}
\varrho_{\varsigma}\left(  x\right)   &  =\varrho_{\varsigma}\left(  0\right)
\left[  \left(  b+1\right)  _{m}m!\right]  ^{2}\frac{1}{\left(  b+1\right)
_{x+m}\left(  x+m\right)  !}\frac{1}{\left(  b+1\right)  _{-x+m}\left(
-x+m\right)  !}\\
&  =\frac{\left(  b+1\right)  _{2m}\left(  2m\right)  !}{\Gamma\left(
x+m+b+1\right)  \Gamma\left(  x+m+1\right)  \Gamma\left(  -x+m+b+1\right)
\Gamma\left(  -x+m+1\right)  }.
\end{align*}
This agrees (up to a normalization factor) with the weight function considered
by the authors in \cite{MR2033351} (equation 26).
\end{remark}

\strut

\section{{\protect\LARGE Semiclassical polynomials of class 1 }}

In this section, we consider all families of polynomials of class 1. We have
$5$ main cases, corresponding to
\[
\left(  p,q\right)  =\left(  0,1\right)  ,\left(  1,1\right)  ,\left(
2,0\right)  ,\left(  2,1\right)  ,\left(  3,2\right)  .
\]
There are also $14$ subcases.

\subsection{$(0,1):$ Generalized Charlier polynomials \label{GenChar}}

Linear functional%
\[
L\left[  r\right]  =%
%TCIMACRO{\dsum \limits_{x=0}^{\infty}}%
%BeginExpansion
{\displaystyle\sum\limits_{x=0}^{\infty}}
%EndExpansion
r\left(  x\right)  \ \frac{1}{\left(  b+1\right)  _{x}}\frac{z^{x}}{x!}.
\]

Pearson equation%
\[
\frac{\varrho\left(  x+1\right)  }{\varrho\left(  x\right)  }=\frac{z}{\left(
x+b+1\right)  \left(  x+1\right)  }.
\]

Moments%
\[
\nu_{n}\left(  z\right)  =\frac{z^{n}}{\left(  b+1\right)  _{n}}\ _{0}%
F_{1}\left[
\begin{array}
[c]{c}%
-\\
b+1+n
\end{array}
;z\right]  , \quad n\in\mathbb{N}_{0} .
\]

Stieltjes transform difference equation%
\[
\left(  t+1\right)  \left(  t+b+1\right)  S\left(  t+1\right)  -zS\left(
t\right)  =\left(  t+b+1\right)  \nu_{0}+\nu_{1}.
\]

\subsection{$(1,1):$ Generalized Meixner polynomials}

Linear functional%
\[
L\left[  r\right]  =%
%TCIMACRO{\dsum \limits_{x=0}^{\infty}}%
%BeginExpansion
{\displaystyle\sum\limits_{x=0}^{\infty}}
%EndExpansion
r\left(  x\right)  \ \frac{\left(  a\right)  _{x}}{\left(  b+1\right)  _{x}%
}\frac{z^{x}}{x!}.
\]

Pearson equation%
\[
\frac{\varrho\left(  x+1\right)  }{\varrho\left(  x\right)  }=\frac{z\left(
x+a\right)  }{\left(  x+1\right)  \left(  x+b+1\right)  }.
\]

Moments%
\[
\nu_{n}\left(  z\right)  =z^{n}\frac{\left(  a\right)  _{n}}{\left(
b+1\right)  _{n}}\ _{1}F_{1}\left[
\begin{array}
[c]{c}%
a+n\\
b+1+n
\end{array}
;z\right]  , \quad n\in\mathbb{N}_{0} .
\]

Stieltjes transform difference equation%
\begin{equation}
\left(  t+1\right)  \left(  t+b+1\right)  S\left(  t+1\right)  -z\left(
t+a\right)  S\left(  t\right)  =\left(  t+b+1-z\right)  \nu_{0}+\nu_{1}.
\label{DES GM}%
\end{equation}

\subsubsection{Reduced-Uvarov Charlier polynomials}

Let $\omega=0.$

Special values%
\[
a\rightarrow0,\quad b\rightarrow-1.
\]

Linear functional%
\[
L_{U}\left[  r\right]  =%
%TCIMACRO{\dsum \limits_{x=0}^{\infty}}%
%BeginExpansion
{\displaystyle\sum\limits_{x=0}^{\infty}}
%EndExpansion
r\left(  x\right)  \ \frac{z^{x}}{x!}+Mr\left(  0\right)  .
\]

Pearson equation%
\[
\frac{\varrho_{U}\left(  x+1\right)  }{\varrho_{U}\left(  x\right)  }%
=\frac{zx}{\left(  x+1\right)  x}.
\]

Moments%
\[
\nu_{n}^{U}\left(  z\right)  =z^{n}e^{z}+M\phi_{n}\left(  0\right)  , \quad
n\in\mathbb{N}_{0} .
\]

Stieltjes transform difference equation%
\begin{equation}
t\left[  \left(  t+1\right)  S_{U}\left(  t+1\right)  -zS_{U}\left(  t\right)
\right]  =\left(  t-z\right)  \nu_{0}^{U}+\nu_{1}^{U}.
\label{DE Dirac Charlier}%
\end{equation}
Since%
\[
\left(  t-z\right)  \nu_{0}^{U}+\nu_{1}^{U}=\left(  t-z\right)  \left(
e^{z}+M\right)  +ze^{z}=\allowbreak t\nu_{0}+M\left(  t-z\right)  ,
\]
where $\nu_{0}=e^{z}$ is the first moment of the Charlier polynomials, we can
rewrite (\ref{DE Dirac Charlier}) in the $M-$dependent form
\[
\left(  t+1\right)  S_{U}\left(  t+1\right)  -zS_{U}\left(  t\right)
=\allowbreak\nu_{0}+M\left(  1-\frac{z}{t}\right)  ,
\]
which agrees with (\ref{reduced1}).

\subsubsection{Christoffel Charlier polynomials}

Let $\omega\notin\left\{  0,z\right\}  .$

Special values%
\[
a\rightarrow-\omega+1,\quad b\rightarrow-\omega-1.
\]

Linear functional%
\[
L_{C}\left[  r\right]  =%
%TCIMACRO{\dsum \limits_{x=0}^{\infty}}%
%BeginExpansion
{\displaystyle\sum\limits_{x=0}^{\infty}}
%EndExpansion
r\left(  x\right)  \left(  x-\omega\right)  \frac{z^{x}}{x!}.
\]

Pearson equation%
\[
\frac{\varrho_{C}\left(  x+1\right)  }{\varrho_{C}\left(  x\right)  }%
=\frac{z\left(  x+1-\omega\right)  }{\left(  x+1\right)  \left(
x-\omega\right)  }.
\]

Moments%
\begin{equation}
\nu_{n}^{C}\left(  z\right)  =\nu_{n+1}+\left(  n-\omega\right)  \nu
_{n}=\left(  n+z-\omega\right)  z^{n}e^{z}, \quad n\in\mathbb{N}_{0}.
\label{Moments CC}%
\end{equation}

Stieltjes transform difference equation%
\[
\left(  t-\omega\right)  \left(  t+1\right)  S_{C}\left(  t+1\right)  -\left(
t+1-\omega\right)  zS_{C}\left(  t\right)  =\nu_{0}\left[  \left(
z-\omega\right)  t+\left(  \omega^{2}-z\omega+z\right)  \right]  .
\]
Note that if we use (\ref{Moments CC}), we have%
\[
\nu_{0}^{C}=\left(  z-\omega\right)  \nu_{0},\quad\nu_{1}^{C}=\left(
1+z-\omega\right)  z\nu_{0},
\]
and therefore%
\[
\left(  t-\omega\right)  \left(  t+1\right)  S_{C}\left(  t+1\right)  -\left(
t+1-\omega\right)  zS_{C}\left(  t\right)  =\left(  t-\omega-z\right)  \nu
_{0}^{C}+\nu_{1}^{C},
\]
which agrees with (\ref{DES GM}) as $a\rightarrow1-\omega,\quad b\rightarrow
-\omega-1.$

\subsubsection{Geronimus Charlier polynomials}

Let $\omega\notin\left\{  -1\right\}  \cup\mathbb{N}_{0}.$

Special values%
\[
a\rightarrow-\omega,\quad b\rightarrow-\omega.
\]

Linear functional%
\[
L_{G}\left[  r\right]  =%
%TCIMACRO{\dsum \limits_{x=0}^{\infty}}%
%BeginExpansion
{\displaystyle\sum\limits_{x=0}^{\infty}}
%EndExpansion
\frac{r\left(  x\right)  }{x-\omega}\frac{z^{x}}{x!}+Mr\left(  \omega\right)
.
\]

Pearson equation%
\[
\frac{\varrho_{G}\left(  x+1\right)  }{\varrho_{G}\left(  x\right)  }%
=\frac{z\left(  x-\omega\right)  }{\left(  x+1\right)  \left(  x+1-\omega
\right)  }.
\]

Moments%
\[
\nu_{n}^{G}\left(  z\right)  =\phi_{n}\left(  \omega\right)  \left[
M-S\left(  \omega\right)  +e^{z}%
%TCIMACRO{\dsum \limits_{k=0}^{n-1}}%
%BeginExpansion
{\displaystyle\sum\limits_{k=0}^{n-1}}
%EndExpansion
\frac{z^{k}}{\phi_{k+1}\left(  \omega\right)  }\right]  ,\quad n\in
\mathbb{N}_{0}.
\]

Stieltjes transform difference equation%
\[
\left(  t+1-\omega\right)  \sigma\left(  t+1\right)  S_{G}\left(  t+1\right)
-\left(  t-\omega\right)  \eta\left(  t\right)  S_{G}\left(  t\right)
=\nu_{0}+\nu_{0}^{G}\left(  t+1-z\right)  .
\]
Note that if we use (\ref{DE MG}), we have%
\[
\nu_{1}^{G}-\omega\nu_{0}^{G}=\nu_{0},
\]
and therefore we can write%
\[
\left(  t+1\right)  \left(  t+1-\omega\right)  S_{G}\left(  t+1\right)
-z\left(  t-\omega\right)  S_{G}\left(  t\right)  =\left(  t+1-\omega
-z\right)  \nu_{0}^{G}+\nu_{1}^{G},
\]
which agrees with (\ref{DES GM}) as $a\rightarrow-\omega,\quad b\rightarrow
-\omega.$

\subsubsection{Truncated Charlier polynomials}

Special values%
\[
a\rightarrow-N,\quad b\rightarrow-N-1.
\]

Linear functional%
\[
L_{T}\left[  r\right]  =%
%TCIMACRO{\dsum \limits_{x=0}^{N}}%
%BeginExpansion
{\displaystyle\sum\limits_{x=0}^{N}}
%EndExpansion
r\left(  x\right)  \ \frac{z^{x}}{x!}.
\]

Pearson equation%
\[
\frac{\varrho_{T}\left(  x+1\right)  }{\varrho_{T}\left(  x\right)  }%
=\frac{z\left(  x-N\right)  }{\left(  x+1\right)  \left(  x-N\right)  }.
\]

Moments%
\[
\nu_{n}^{T}(z)=\frac{z^{N}}{\left(  N-n\right)  !}\ _{2}F_{0}\left[
\begin{array}
[c]{c}%
n-N,1\\
-
\end{array}
;-\frac{1}{z}\right]  , \quad n\in\mathbb{N}_{0} .
\]

Stieltjes transform difference equation%
\[
\left(  t-N\right)  \left[  \left(  t+1\right)  S_{T}\left(  t+1\right)
-zS_{T}\left(  t\right)  \right]  =\left(  t-N-z\right)  \nu_{0}^{T}+\nu
_{1}^{T}.
\]

\begin{remark}
The performance of the modified Chebyshev algorithm used to compute the
moments of these polynomials was studied in \cite{MR667829} (example 4.3).
\end{remark}

\subsection{$(2,0;N):$ Generalized Krawtchouk polynomials}

Linear functional%
\[
L\left[  r\right]  =%
%TCIMACRO{\dsum \limits_{x=0}^{N}}%
%BeginExpansion
{\displaystyle\sum\limits_{x=0}^{N}}
%EndExpansion
r\left(  x\right)  \ \left(  a\right)  _{x}\left(  -N\right)  _{x}\frac{z^{x}%
}{x!}.
\]

Pearson equation%
\[
\frac{\varrho\left(  x+1\right)  }{\varrho\left(  x\right)  }=\frac{z\left(
x+a\right)  \left(  x-N\right)  }{x+1}.
\]

\strut Moments%
\[
\nu_{n}\left(  z\right)  =z^{n}\left(  -N\right)  _{n}\left(  a\right)
_{n}\ _{2}F_{0}\left[
\begin{array}
[c]{c}%
-N+n,a+n\\
-
\end{array}
;z\right]  , \quad n\in\mathbb{N}_{0} .
\]

Stieltjes transform difference equation%
\[
\left(  t+1\right)  S\left(  t+1\right)  -z\left(  t+a\right)  \left(
t-N\right)  S\left(  t\right)  =\left[  1-\left(  t+a-N\right)  z\right]
\nu_{0}-z\nu_{1}.
\]

\subsection{$(2,1):$ Generalized Hahn polynomials of type I}

Linear functional%
\[
L\left[  r\right]  =%
%TCIMACRO{\dsum \limits_{x=0}^{\infty}}%
%BeginExpansion
{\displaystyle\sum\limits_{x=0}^{\infty}}
%EndExpansion
r\left(  x\right)  \ \frac{\left(  a_{1}\right)  _{x}\left(  a_{2}\right)
_{x}}{\left(  b+1\right)  _{x}}\frac{z^{x}}{x!}.
\]

Pearson equation%
\[
\frac{\varrho\left(  x+1\right)  }{\varrho\left(  x\right)  }=\frac{z\left(
x+a_{1}\right)  \left(  x+a_{2}\right)  }{\left(  x+b+1\right)  \left(
x+1\right)  }.
\]

\strut Moments%
\[
\nu_{n}\left(  z\right)  =z^{n}\frac{\left(  a_{1}\right)  _{n}\left(
a_{2}\right)  _{n}}{\left(  b+1\right)  _{n}}\ _{2}F_{1}\left[
\begin{array}
[c]{c}%
a_{1}+n,a_{2}+n\\
b+1+n
\end{array}
;z\right]  , \quad n\in\mathbb{N}_{0} ,
\]
where we choose the principal branch $z\in\mathbb{C}\setminus\lbrack
1,\infty).$ Note that since%
\[
b+1+n-\left(  a_{1}+n+a_{2}+n\right)  \leq0,\quad n\geq b-a_{1}-a_{2}+1,
\]
$\nu_{n}\left(  1\right)  $ is undefined for $n\geq b-a_{1}-a_{2}+1.$

Stieltjes transform difference equation%
\begin{align*}
&  \left(  t+1\right)  \left(  t+b+1\right)  S\left(  t+1\right)  -z\left(
t+a_{1}\right)  \left(  t+a_{2}\right)  S\left(  t\right) \\
&  =\left[  t+b+1-z\left(  t+a_{1}+a_{2}\right)  \right]  \nu_{0}+\left(
1-z\right)  \nu_{1}.
\end{align*}

\subsubsection{\strut Reduced-Uvarov Meixner polynomials}

Let $\left(  \omega,\Omega\right)  \in\left\{  \left(  0,0\right)  ,\left(
-a,a+1\right)  \right\}  .$

Special values%
\[
a_{1}\rightarrow\Omega,\quad a_{2}\rightarrow a,\quad b\rightarrow\Omega-1.
\]

Linear functional%
\[
L_{U}\left[  r\right]  =%
%TCIMACRO{\dsum \limits_{x=0}^{\infty}}%
%BeginExpansion
{\displaystyle\sum\limits_{x=0}^{\infty}}
%EndExpansion
r\left(  x\right)  \ \left(  a\right)  _{x}\frac{z^{x}}{x!}+Mr\left(
\omega\right)  .
\]

Pearson equation%
\[
\frac{\varrho_{U}\left(  x+1\right)  }{\varrho_{U}\left(  x\right)  }%
=\frac{z\left(  x+\Omega\right)  \left(  x+a\right)  }{\left(  x+\Omega
\right)  \left(  x+1\right)  }.
\]

Moments%
\[
\nu_{n}^{U}\left(  z\right)  =z^{n}\left(  a\right)  _{n}\left(  1-z\right)
^{-a-n}+M\phi_{n}\left(  \omega\right)  , \quad n\in\mathbb{N}_{0} .
\]

Stieltjes transform difference equation%
\[
\left(  t+\Omega\right)  \left[  \left(  t+1\right)  S_{U}\left(  t+1\right)
-z\left(  t+a\right)  S_{U}\left(  t\right)  \right]  =\left[  (1-z)\left(
t+\Omega\right)  -az\right]  \nu_{0}^{U}+\left(  1-z\right)  \nu_{1}^{U}.
\]

Stieltjes transform $M-$dependent difference equation%
\[
\left(  t+1\right)  S_{U}\left(  t+1\right)  -z\left(  t+a\right)
S_{U}\left(  t\right)  =\left(  1-z\right)  \nu_{0}+M\left[  \frac
{t+1}{t-\omega+1}-\frac{z\left(  t+a\right)  }{t-\omega}\right]  .
\]

\subsubsection{Christoffel Meixner polynomials}

Let $\omega\notin\left\{  -a,0,az\left(  1-z\right)  ^{-1}\right\}  .$

Special values%
\[
a_{1}\rightarrow a,\quad a_{2}\rightarrow-\omega+1,\quad b\rightarrow
-\omega-1.
\]

Linear functional%
\[
L_{C}\left[  r\right]  =%
%TCIMACRO{\dsum \limits_{x=0}^{\infty}}%
%BeginExpansion
{\displaystyle\sum\limits_{x=0}^{\infty}}
%EndExpansion
r\left(  x\right)  \left(  x-\omega\right)  \left(  a\right)  _{x}%
\ \frac{z^{x}}{x!}.
\]

Pearson equation%
\[
\frac{\varrho_{C}\left(  x+1\right)  }{\varrho_{C}\left(  x\right)  }%
=\frac{z\left(  x+a\right)  \left(  x+1-\omega\right)  }{\left(  x+1\right)
\left(  x-\omega\right)  }.
\]

Moments%
\begin{equation}
\nu_{n}^{C}\left(  z\right)  =\nu_{n+1}+\left(  n-\omega\right)  \nu_{n}%
=z^{n}\left(  a\right)  _{n}\left(  1-z\right)  ^{-a-n-1}\left(  az+\omega
z+n-\omega\right)  , \quad n\in\mathbb{N}_{0} . \label{Moments CM}%
\end{equation}

Stieltjes transform difference equation%
\begin{align*}
&  \left(  t-\omega\right)  \left(  t+1\right)  S_{C}\left(  t+1\right)
-\left(  t+1-\omega\right)  z\left(  t+a\right)  S_{C}\left(  t\right) \\
&  =\nu_{0}\left[  \allowbreak\left(  az-\omega+\omega z\right)  t+\left(
\omega^{2}+az+\omega z-az\omega-z\omega^{2}\right)  \right]  .
\end{align*}
Note that if we use (\ref{Moments CM}), we have%
\[
\nu_{0}^{C}=\frac{az+\omega z-\omega}{1-z}\nu_{0},\quad\nu_{1}^{C}%
=\frac{za(az+\omega z+1-\omega)}{\left(  1-z\right)  ^{2}}\nu_{0},
\]
and therefore we can also write%
\begin{align*}
&  \left(  t-\omega\right)  \left(  t+1\right)  S_{C}\left(  t+1\right)
-\left(  t+1-\omega\right)  z\left(  t+a\right)  S_{C}\left(  t\right) \\
&  =\left[  t-\omega-z\left(  t+a-\omega+1\right)  \right]  \nu_{0}%
^{C}+\left(  1-z\right)  \nu_{1}^{C}.
\end{align*}

\subsubsection{Geronimus Meixner polynomials}

Let $\omega\notin\left\{  -1,1-a\right\}  \cup\mathbb{N}_{0}.$

Special values%
\[
a_{1}\rightarrow a,\quad a_{2}\rightarrow-\omega,\quad b\rightarrow-\omega.
\]

Linear functional%
\[
L_{G}\left[  r\right]  =%
%TCIMACRO{\dsum \limits_{x=0}^{\infty}}%
%BeginExpansion
{\displaystyle\sum\limits_{x=0}^{\infty}}
%EndExpansion
\frac{r\left(  x\right)  }{x-\omega}\left(  a\right)  _{x}\ \frac{z^{x}}%
{x!}+Mr\left(  \omega\right)  .
\]

Pearson equation%
\[
\frac{\varrho_{G}\left(  x+1\right)  }{\varrho_{G}\left(  x\right)  }%
=\frac{z\left(  x+a\right)  \left(  x-\omega\right)  }{\left(  x+1\right)
\left(  x+1-\omega\right)  }.
\]

Moments%
\[
\nu_{n}^{G}\left(  z\right)  =\phi_{n}\left(  \omega\right)  \left[
M-S\left(  \omega\right)  +\left(  1-z\right)  ^{-a}%
%TCIMACRO{\dsum \limits_{k=0}^{n-1}}%
%BeginExpansion
{\displaystyle\sum\limits_{k=0}^{n-1}}
%EndExpansion
\frac{z^{k}\left(  a\right)  _{k}\left(  1-z\right)  ^{-k}}{\phi_{k+1}\left(
\omega\right)  }\right]  ,\quad n\in\mathbb{N}_{0}.
\]

Stieltjes transform difference equation%
\begin{align*}
&  \left(  t+1\right)  \left(  t-\omega+1\right)  S_{G}\left(  t+1\right)
-z\left(  t+a\right)  \left(  t-\omega\right)  S_{G}\left(  t\right) \\
&  =\left(  1-z\right)  \nu_{0}+\nu_{0}^{G}\left[  \allowbreak\left(
1-z\right)  t+\left(  1-az\right)  \right]  .
\end{align*}
Note that if we use (\ref{DE MG}), we have%
\[
\nu_{1}^{G}-\omega\nu_{0}^{G}=\nu_{0},
\]
and therefore we can also write%
\begin{align*}
&  \left(  t+1\right)  \left(  t-\omega+1\right)  S_{G}\left(  t+1\right)
-z\left(  t+a\right)  \left(  t-\omega\right)  S_{G}\left(  t\right) \\
&  =\left[  \allowbreak\left(  1-z\right)  t+z\omega-\omega-az+1\right]
\nu_{0}^{G}+\left(  1-z\right)  \nu_{1}^{G}.
\end{align*}

\subsubsection{Truncated Meixner polynomials}

Special values%
\[
a_{1}\rightarrow a,\quad a_{2}\rightarrow-N,\quad b\rightarrow-N-1.
\]

Linear functional%
\[
L_{T}\left[  r\right]  =%
%TCIMACRO{\dsum \limits_{x=0}^{N}}%
%BeginExpansion
{\displaystyle\sum\limits_{x=0}^{N}}
%EndExpansion
r\left(  x\right)  \ \left(  a\right)  _{x}\ \frac{z^{x}}{x!}.
\]

\strut Moments%
\[
\nu_{n}^{T}=\left(  a\right)  _{N}\ \frac{z^{N}}{\left(  N-n\right)  !}%
\ _{2}F_{1}\left[
\begin{array}
[c]{c}%
n-N,1\\
1-N-\ a
\end{array}
;\frac{1}{z}\right]  , \quad n\in\mathbb{N}_{0} .
\]

Pearson equation%
\[
\frac{\varrho_{T}\left(  x+1\right)  }{\varrho_{T}\left(  x\right)  }%
=\frac{z\left(  x+a\right)  \left(  x-N\right)  }{\left(  x-N\right)  \left(
x+1\right)  }.
\]

Stieltjes transform difference equation%
\[
\left(  t-N\right)  \left[  \left(  t+1\right)  S_{T}\left(  t+1\right)
-z\left(  t+a\right)  S_{T}(t)\right]  =\left[  \left(  t-N\right)  \left(
1-z\right)  -az\right]  \nu_{0}^{T}+\left(  1-z\right)  \nu_{1}^{T}.
\]

\subsubsection{Symmetrized Generalized Krawtchouk polynomials}

Special values%
\[
a_{1}\rightarrow a,\quad a_{2}\rightarrow-2m,\quad b\rightarrow-2m-a,\quad
z\rightarrow-1,\quad x\rightarrow x+m.
\]

Linear functional%
\[
L_{\varsigma}\left[  r\right]  =\varrho_{\varsigma}\left(  0\right)
%TCIMACRO{\dsum \limits_{x=-m}^{m}}%
%BeginExpansion
{\displaystyle\sum\limits_{x=-m}^{m}}
%EndExpansion
r\left(  x\right)  \ \frac{\left(  a+m\right)  _{x}\ \left(  -m\right)  _{x}%
}{\left(  -m-a+1\right)  _{x}\ \left(  m+1\right)  _{x}}\left(  -1\right)
^{x},
\]
where%
\[
\varrho_{\varsigma}\left(  0\right)  =\left(  -1\right)  ^{m}\binom{2m}%
{m}\frac{\left(  a\right)  _{m}\ }{\left(  a+m\right)  _{m}\ }.
\]

Pearson equation%
\[
\frac{\varrho_{\varsigma}\left(  x+1\right)  }{\varrho_{\varsigma}\left(
x\right)  }=\frac{-\left(  x+a+m\right)  \left(  x-m\right)  }{\left(
x+1-a-m\right)  \left(  x+1+m\right)  }.
\]

\strut Moments on the basis $\{\overline{\phi}_{n}\left(  x\right)
\}_{n\geq0} $%
\[
\nu_{n}^{\varsigma}=\left(  -1\right)  ^{n}\frac{\left(  a\right)  _{n}\left(
-2m\right)  _{n}}{\left(  -2m-a\right)  _{n}}\ _{2}F_{1}\left[
\begin{array}
[c]{c}%
-2m+n,a+n\\
-2m-a+n
\end{array}
;-1\right]  , \quad n\in\mathbb{N}_{0} .
\]

Stieltjes transform difference equation%
\[
\left(  t+m+1\right)  \left(  t-m-a+1\right)  S_{\varsigma}\left(  t+1\right)
+\left(  t-m\right)  \left(  t+m+a\right)  S_{\varsigma}\left(  t\right)
=\left(  2t-2m+1\right)  \nu_{0}^{\varsigma}+2\nu_{1}^{\varsigma}.
\]

\subsection{$(3,2;N,1):$ Generalized Hahn polynomials of type II}

Linear functional%
\[
L\left[  r\right]  =%
%TCIMACRO{\dsum \limits_{x=0}^{N}}%
%BeginExpansion
{\displaystyle\sum\limits_{x=0}^{N}}
%EndExpansion
r\left(  x\right)  \ \frac{\left(  a_{1}\right)  _{x}\left(  a_{2}\right)
_{x}\left(  -N\right)  _{x}}{\left(  b_{1}+1\right)  _{x}\left(
b_{2}+1\right)  _{x}}\frac{1}{x!}.
\]

Pearson equation%
\[
\frac{\varrho\left(  x+1\right)  }{\varrho\left(  x\right)  }=\frac{\left(
x+a_{1}\right)  \left(  x+a_{2}\right)  \left(  x-N\right)  }{\left(
x+b_{1}+1\right)  \left(  x+b_{2}+1\right)  \left(  x+1\right)  }.
\]

Moments%
\[
\nu_{n}=\frac{\left(  -N\right)  _{n}\left(  a_{1}\right)  _{n}\left(
a_{2}\right)  _{n}}{\left(  b_{1}+1\right)  _{n}\left(  b_{2}+1\right)  _{n}%
}\ _{3}F_{2}\left[
\begin{array}
[c]{c}%
-N+n,a_{1}+n,a_{2}+n\\
b_{1}+1+n,b_{2}+1+n
\end{array}
;1\right]  , \quad n\in\mathbb{N}_{0} .
\]

Stieltjes transform difference equation%
\begin{align*}
&  \left(  t+b_{1}+1\right)  \left(  t+b_{2}+1\right)  \left(  t+1\right)
S\left(  t+1\right)  -\left(  t+a_{1}\right)  \left(  t+a_{2}\right)  \left(
t-N\right)  S\left(  t\right) \\
&  =\left[  \left(  N+2+b_{1}+b_{2}-a_{1}-a_{2}\right)  t+b_{1}+b_{2}%
+b_{1}b_{2}-a_{1}-a_{2}-a_{1}a_{2}+Na_{1}+Na_{2}+1\right]  \nu_{0}\\
&  +\left(  N+b_{1}+b_{2}-a_{1}-a_{2}+1\right)  \nu_{1}.
\end{align*}

\subsubsection{Reduced-Uvarov Hahn polynomials}

Let $\left(  \omega,\Omega\right)  \in\left\{  \left(  0,0\right)  ,\left(
-a,a+1\right)  ,\left(  -b,b\right)  ,\left(  N,-N+1\right)  \right\}  .$

Special values%
\[
a_{1}\rightarrow a,\quad a_{2}\rightarrow\Omega,\quad b_{1}\rightarrow
\Omega-1,\quad b_{2}\rightarrow a.
\]

Linear functional%
\[
L_{U}\left[  r\right]  =%
%TCIMACRO{\dsum \limits_{x=0}^{N}}%
%BeginExpansion
{\displaystyle\sum\limits_{x=0}^{N}}
%EndExpansion
r\left(  x\right)  \ \frac{\left(  a\right)  _{x}\left(  -N\right)  _{x}%
}{\left(  b+1\right)  _{x}}\frac{1}{x!}+Mr\left(  \omega\right)  .
\]

Pearson equation%

\[
\frac{\varrho_{U}\left(  x+1\right)  }{\varrho_{U}\left(  x\right)  }%
=\frac{\left(  x+a\right)  \left(  x+\Omega\right)  \left(  x-N\right)
}{\left(  x+b+1\right)  \left(  x+\Omega\right)  \left(  x+1\right)  }.
\]

Moments%
\[
\nu_{n}^{U}=\frac{\left(  -N\right)  _{n}\left(  a\right)  _{n}}{\left(
b+1\right)  _{n}}\frac{\left(  b+1-a\right)  _{N-n}}{\left(  b+1+n\right)
_{N-n}}+M\phi_{n}\left(  \omega\right)  , \quad n\in\mathbb{N}_{0} .
\]

Stieltjes transform difference equation%
\begin{gather*}
\left(  t+\Omega\right)  \left[  \left(  t+b+1\right)  \left(  t+1\right)
S_{U}\left(  t+1\right)  -\left(  t+a\right)  \left(  t-N\right)  \right]
S_{U}\left(  t\right) \\
=\left[  (N-a+b+1)t+\left(  N-a+b+1\right)  \Omega+aN\right]  \nu_{0}%
^{U}+\left(  N+b-a\right)  \nu_{1}^{U}.
\end{gather*}

Stieltjes transform $M-$dependent difference equation%
\begin{align*}
&  \left(  t+b+1\right)  \left(  t+1\right)  S_{U}\left(  t+1\right)  -\left(
t-N\right)  \left(  t+a\right)  S_{U}\left(  t\right) \\
&  =\left(  N+b-a+1\right)  \nu_{0}+M\left[  \frac{\left(  t+b+1\right)
\left(  t+1\right)  }{t-\omega+1}-\frac{\left(  t-N\right)  \left(
t+a\right)  }{t-\omega}\right]  .
\end{align*}

\subsubsection{Christoffel Hahn polynomials}

Let $\omega\notin\left\{  -a,-b,0,N,-\dfrac{aN}{b-a+N}\right\}  .$

Special values%
\[
a_{1}\rightarrow a,\quad a_{2}\rightarrow-\omega+1,\quad b_{1}\rightarrow
b,\quad b_{2}\rightarrow-\omega-1.
\]

Linear functional%
\[
L_{C}\left[  r\right]  =%
%TCIMACRO{\dsum \limits_{x=0}^{\infty}}%
%BeginExpansion
{\displaystyle\sum\limits_{x=0}^{\infty}}
%EndExpansion
r\left(  x\right)  \left(  x-\omega\right)  \frac{\left(  a\right)
_{x}\left(  -N\right)  _{x}}{\left(  b+1\right)  _{x}}\ \frac{1}{x!}.
\]

Pearson equation%
\[
\frac{\varrho_{C}\left(  x+1\right)  }{\varrho_{C}\left(  x\right)  }%
=\frac{\left(  x+a\right)  \left(  x-N\right)  \left(  x+1-\omega\right)
}{\left(  x+b+1\right)  \left(  x+1\right)  \left(  x-\omega\right)  }.
\]

Moments%
\begin{equation}
\nu_{n}^{C}\left(  z\right)  =\nu_{n+1}+\left(  n-\omega\right)  \nu
_{n}=\left[  n-\omega-\frac{\left(  a+n\right)  \left(  N-n\right)  }%
{b-a+N-n}\right]  \nu_{n}, \quad n\in\mathbb{N}_{0}. \label{Moments CH}%
\end{equation}

Stieltjes transform difference equation%
\begin{align*}
&  \left(  t-\omega\right)  \left(  t+b+1\right)  \left(  t+1\right)
S_{C}\left(  t+1\right)  -\left(  t-\omega+1\right)  \left(  t-N\right)
\left(  t+a\right)  S_{C}\left(  t\right) \\
&  =\left[  \left(  a\omega-N\omega-b\omega-Na\right)  t+a\omega
-N\omega+\omega^{2}+N\omega^{2}-a\omega^{2}+b\omega^{2}-Na+Na\omega\right]
\nu_{0}.
\end{align*}
Note that if we use (\ref{Moments CH}), we have%
\[
\nu_{0}^{C}=-\frac{\left(  a+\omega\right)  N+\left(  b-a\right)  \omega
}{b-a+N}\nu_{0},\quad\nu_{1}^{C}=\frac{(Na+N\omega-a\omega+b\omega
-b-\omega)Na}{\left(  b-a+N\right)  \left(  b-a+N-1\right)  }\nu_{0},
\]
and therefore we can also write%
\begin{gather*}
\left(  t-\omega\right)  \left(  t+b+1\right)  \left(  t+1\right)
S_{C}\left(  t+1\right)  -\left(  t-\omega+1\right)  \left(  t-N\right)
\left(  t+a\right)  S_{C}\left(  t\right) \\
=\left[  \allowbreak\left(  N-a+b\right)  t+Na-N\omega+a\omega-b\omega
+N-a-\omega\right]  \nu_{0}^{C}+\left(  N-a+b-1\right)  \nu_{1}^{C}.
\end{gather*}

\subsubsection{Geronimus Hahn polynomials}

Let $\omega\notin\left\{  -b-1,-1,1-a,N+1\right\}  \cup\mathbb{N}_{0}.$

Special values%
\[
a_{1}\rightarrow a,\quad a_{2}\rightarrow-\omega,\quad b\rightarrow-\omega.
\]

Linear functional%
\[
L_{G}\left[  r\right]  =%
%TCIMACRO{\dsum \limits_{x=0}^{\infty}}%
%BeginExpansion
{\displaystyle\sum\limits_{x=0}^{\infty}}
%EndExpansion
\frac{r\left(  x\right)  }{x-\omega}\frac{\left(  a\right)  _{x}\left(
-N\right)  _{x}}{\left(  b+1\right)  _{x}}\ \frac{1}{x!}+Mr\left(
\omega\right)  .
\]

Pearson equation%
\[
\frac{\varrho_{G}\left(  x+1\right)  }{\varrho_{G}\left(  x\right)  }%
=\frac{\left(  x+a\right)  \left(  x-N\right)  \left(  x-\omega\right)
}{\left(  x+b+1\right)  \left(  x+1\right)  \left(  x+1-\omega\right)  }.
\]

Moments%
\[
\nu_{n}^{G}\left(  z\right)  =\phi_{n}\left(  \omega\right)  \left[
M-S\left(  \omega\right)  +%
%TCIMACRO{\dsum \limits_{k=0}^{n-1}}%
%BeginExpansion
{\displaystyle\sum\limits_{k=0}^{n-1}}
%EndExpansion
\frac{\nu_{k}}{\phi_{k+1}\left(  \omega\right)  }\right]  ,\quad
n\in\mathbb{N}_{0}.
\]

Stieltjes transform difference equation%
\[
\xi_{G}\left(  t\right)  =\xi\left(  t\right)  +\left[  \sigma\left(
t+1\right)  -\eta\left(  t\right)  \right]  \nu_{0}^{G},
\]%
\begin{align*}
&  \left(  t-\omega+1\right)  \left(  t+b+1\right)  \left(  t+1\right)
S_{G}\left(  t+1\right)  -\left(  t-\omega\right)  \left(  t-N\right)  \left(
t+a\right)  S_{G}\left(  t\right) \\
&  =\left(  b+1-a+N\right)  \nu_{0}+\left[  \allowbreak\left(  N-a+b+2\right)
t+b+Na+1\right]  \nu_{0}^{G}.
\end{align*}
Note that if we use (\ref{DE MG}), we have%
\[
\nu_{1}^{G}-\omega\nu_{0}^{G}=\nu_{0},
\]
and therefore we can also write%
\begin{align*}
&  \left(  t-\omega+1\right)  \left(  t+b+1\right)  \left(  t+1\right)
S_{G}\left(  t+1\right)  -\left(  t-\omega\right)  \left(  t-N\right)  \left(
t+a\right)  S_{G}\left(  t\right) \\
&  =\left[  \allowbreak(N-a+b+2)t+Na-N\omega+a\omega-b\omega+b-\omega
+1\right]  \nu_{0}^{G}+\left(  b-1-a+N\right)  \nu_{1}^{G}.
\end{align*}
$\allowbreak$

\subsubsection{Symmetrized Hahn polynomials}

Special values%
\[
a_{1}\rightarrow a,\quad a_{2}\rightarrow-N-b,\quad b_{1}\rightarrow b,\quad
b_{2}\rightarrow-N-a,\quad N\rightarrow2m,\quad x\rightarrow x+m.
\]

Linear functional%
\[
L_{\varsigma}\left[  r\right]  =\varrho_{\varsigma}\left(  0\right)
%TCIMACRO{\dsum \limits_{x=-m}^{m}}%
%BeginExpansion
{\displaystyle\sum\limits_{x=-m}^{m}}
%EndExpansion
r\left(  x\right)  \ \frac{\left(  m+a\right)  _{x}\left(  -m-b\right)
_{x}\left(  -m\right)  _{x}}{\left(  1-m-a\right)  _{x}\left(  m+b+1\right)
_{x}\left(  m+1\right)  _{x}},
\]
where%
\[
\varrho_{\varsigma}\left(  0\right)  =\left(  -1\right)  ^{m}\binom{2m}%
{m}\frac{\left(  a\right)  _{m}}{\left(  a+m\right)  _{m}}\frac{\left(
b+1+m\right)  _{m}}{\left(  b+1\right)  _{m}}.
\]

Pearson equation%
\[
\frac{\varrho_{\varsigma}\left(  x+1\right)  }{\varrho_{\varsigma}\left(
x\right)  }=\frac{\left(  x+m+a\right)  \left(  x-m-b\right)  \left(
x-m\right)  }{\left(  x+1-m-a\right)  \left(  x+1+m+b\right)  \left(
x+1+m\right)  }.
\]

\strut Moments on the basis $\{\overline{\phi}_{n}\left(  x\right)
\}_{n\geq0} $%
\[
\nu_{n}^{\varsigma}=\frac{\left(  -2m\right)  _{n}\left(  a\right)
_{n}\left(  -2m-b\right)  _{n}}{\left(  b+1\right)  _{n}\left(
-2m-a+1\right)  _{n}}\ _{3}F_{2}\left[
\begin{array}
[c]{c}%
-2m+n,a+n,-2m-b+n\\
b+1+n,-2m-a+1+n
\end{array}
;1\right]  , \quad n\in\mathbb{N}_{0} .
\]

Stieltjes transform difference equation%
\begin{align*}
&  \left(  t+1-m-a\right)  \left(  t+1+m+b\right)  \left(  t+1+m\right)
S_{\varsigma}\left(  t+1\right)  -\left(  t+m+a\right)  \left(  t-m-b\right)
\left(  t-m\right)  S_{\varsigma}(t)\\
&  =\left[  2\left(  1+b-a+m\right)  t+2am-2bm-2m^{2}-a+b+1\right]  \nu
_{0}^{\varsigma}+\left(  1+2m-2a+2b\right)  \nu_{1}^{\varsigma}.
\end{align*}

\subsubsection{Symmetrized polynomials of type $(3,0;N)$}

For a definition of the polynomials of type $(3,0;N),$ see Section \ref{(30N)}.

Special values%
\[
b_{1}\rightarrow-N-a_{1},\quad b_{2}\rightarrow-N-a_{2},\quad N\rightarrow
2m,\quad x\rightarrow x+m.
\]

Linear functional%
\[
L_{\varsigma}\left[  r\right]  =\varrho_{\varsigma}\left(  0\right)
%TCIMACRO{\dsum \limits_{x=-m}^{m}}%
%BeginExpansion
{\displaystyle\sum\limits_{x=-m}^{m}}
%EndExpansion
r\left(  x\right)  \ \frac{\left(  a_{1}+m\right)  _{x}\left(  a_{2}+m\right)
_{x}\left(  -m\right)  _{x}}{\left(  1-a_{1}-m\right)  _{x}\left(
1-a_{2}-m\right)  _{x}\left(  m+1\right)  _{x}},
\]
where%
\[
\varrho_{\varsigma}\left(  0\right)  =\left(  -1\right)  ^{m}\binom{2m}%
{m}\frac{\left(  a_{1}\right)  _{m}}{\left(  a_{1}+m\right)  _{m}}%
\frac{\left(  a_{2}\right)  _{m}}{\left(  a_{2}+m\right)  _{m}}.
\]

Pearson equation%
\[
\frac{\varrho_{\varsigma}\left(  x+1\right)  }{\varrho_{\varsigma}\left(
x\right)  }=\frac{\left(  x+a_{1}+m\right)  \left(  x+a_{2}+m\right)  \left(
x-m\right)  }{\left(  x+1-a_{1}-m\right)  \left(  x+1-a_{2}-m\right)  \left(
x+1+m\right)  }.
\]

\strut Moments on the basis $\{\overline{\phi}_{n}\left(  x\right)
\}_{n\geq0} $
\[
\nu_{n}^{\varsigma}=\frac{\left(  -2m\right)  _{n}\left(  a_{1}\right)
_{n}\left(  a_{2}\right)  _{n}}{\left(  1-a_{1}-2m\right)  _{n}\left(
1-a_{2}-2m\right)  _{n}}\ _{3}F_{2}\left[
\begin{array}
[c]{c}%
-2m+n,a_{1}+n,a_{2}+n\\
1-a_{1}-2m+n,1-a_{2}-2m+n
\end{array}
;1\right]  , \quad n\in\mathbb{N}_{0} .
\]

Stieltjes transform difference equation%
\begin{align*}
&  \left(  t+1-a_{1}-m\right)  \left(  t+1-a_{2}-m\right)  \left(
t+1+m\right)  S_{\varsigma}\left(  t+1\right) \\
&  -\left(  t+a_{1}+m\right)  \left(  t+a_{2}+m\right)  \left(  t-m\right)
S_{\varsigma}(t)\\
&  =\left[  -2\left(  m+a_{1}+a_{2}-1\right)  t+2m^{2}-2m+1+\left(
2m-1\right)  \left(  a_{1}+a_{2}\right)  \right]  \nu_{0}^{\varsigma}\\
&  +\left(  -2m+1-2a_{1}-2a_{2}\right)  \nu_{1}^{\varsigma}.
\end{align*}

\section{\strut Semiclassical polynomials of class 2}

In this section, we consider all families of polynomials of class $2$. We have
$7$ main cases, corresponding to
\[
\left(  p,q\right)  =\left(  0,2\right)  ,\left(  1,2\right)  ,\left(
2,2\right)  ,\left(  3,0\right)  ,\left(  3,1\right)  ,\left(  3,2\right)
,\left(  4,3\right)  .
\]
There are also $25$ subcases.

\subsection{Polynomials of type (0,2)}

\strut Linear functional%
\[
L\left[  r\right]  =%
%TCIMACRO{\dsum \limits_{x=0}^{\infty}}%
%BeginExpansion
{\displaystyle\sum\limits_{x=0}^{\infty}}
%EndExpansion
r\left(  x\right)  \ \frac{1}{\left(  b_{1}+1\right)  _{x}\left(
b_{2}+1\right)  _{x}}\frac{z^{x}}{x!}.
\]

Pearson equation%
\[
\frac{\varrho\left(  x+1\right)  }{\varrho\left(  x\right)  }=\frac{z}{\left(
x+b_{1}+1\right)  \left(  x+b_{2}+1\right)  \left(  x+1\right)  }.
\]

Moments%
\[
\nu_{n}\left(  z\right)  =\frac{z^{n}}{\left(  b_{1}+1\right)  _{n}\left(
b_{2}+1\right)  _{n}}\ _{0}F_{2}\left[
\begin{array}
[c]{c}%
-\\
b_{1}+1+n,b_{2}+1+n
\end{array}
;z\right]  , \quad n\in\mathbb{N}_{0} .
\]

Stieltjes transform difference equation%
\begin{align*}
&  \left(  t+b_{1}+1\right)  \left(  t+b_{2}+1\right)  \left(  t+1\right)
S\left(  t+1\right)  -zS\left(  t\right) \\
&  =\left(  t+b_{1}+1\right)  \left(  t+b_{2}+1\right)  \nu_{0}+\left(
t+b_{1}+b_{2}+2\right)  \nu_{1}+\nu_{2}.
\end{align*}

\subsection{Polynomials of type (1,2)}

\strut Linear functional%
\[
L\left[  r\right]  =%
%TCIMACRO{\dsum \limits_{x=0}^{\infty}}%
%BeginExpansion
{\displaystyle\sum\limits_{x=0}^{\infty}}
%EndExpansion
r\left(  x\right)  \ \frac{\left(  a\right)  _{x}}{\left(  b_{1}+1\right)
_{x}\left(  b_{2}+1\right)  _{x}}\frac{z^{x}}{x!}.
\]

Pearson equation%
\[
\frac{\varrho\left(  x+1\right)  }{\varrho\left(  x\right)  }=\frac{z\left(
x+a\right)  }{\left(  x+b_{1}+1\right)  \left(  x+b_{2}+1\right)  \left(
x+1\right)  }.
\]

Moments%
\[
\nu_{n}\left(  z\right)  =z^{n}\frac{\left(  a\right)  _{n}}{\left(
b_{1}+1\right)  _{n}\left(  b_{2}+1\right)  _{n}}\ _{1}F_{2}\left[
\begin{array}
[c]{c}%
a+n\\
b_{1}+1+n,b_{2}+1+n
\end{array}
;z\right]  , \quad n\in\mathbb{N}_{0} .
\]

Stieltjes transform difference equation%
\begin{align*}
&  \left(  t+b_{1}+1\right)  \left(  t+b_{2}+1\right)  \left(  t+1\right)
S\left(  t+1\right)  -z\left(  t+a\right)  S\left(  t\right) \\
&  =\left[  \left(  t+b_{1}+1\right)  \left(  t+b_{2}+1\right)  -z\right]
\nu_{0}+\left(  t+b_{1}+b_{2}+2\right)  \nu_{1}+\nu_{2}.
\end{align*}
\strut

\subsubsection{Reduced-Uvarov Generalized Charlier polynomials}

Let $\omega\in\left\{  0,-b\right\}  .$

\strut Special values%
\[
a\rightarrow-\omega,\quad b_{1}\rightarrow b,\quad b_{2}\rightarrow-\omega-1.
\]

Linear functional%
\[
L_{U}\left[  r\right]  =%
%TCIMACRO{\dsum \limits_{x=0}^{\infty}}%
%BeginExpansion
{\displaystyle\sum\limits_{x=0}^{\infty}}
%EndExpansion
r\left(  x\right)  \ \frac{1}{\left(  b+1\right)  _{x}}\frac{z^{x}}%
{x!}+Mr\left(  \omega\right)  .
\]

Pearson equation%
\[
\frac{\varrho_{U}\left(  x+1\right)  }{\varrho_{U}\left(  x\right)  }%
=\frac{z\left(  x-\omega\right)  }{\left(  x+b+1\right)  \left(
x-\omega\right)  \left(  x+1\right)  }.
\]

Moments%
\[
\nu_{n}^{U}\left(  z\right)  =\frac{z^{n}}{\left(  b+1\right)  _{n}}%
\ _{0}F_{1}\left[
\begin{array}
[c]{c}%
-\\
b+1+n
\end{array}
;z\right]  +M\phi_{n}\left(  \omega\right)  , \quad n\in\mathbb{N}_{0} .
\]

Stieltjes transform difference equation%
\begin{align*}
&  \left(  t-\omega\right)  \left[  \left(  t+b+1\right)  \left(  t+1\right)
S_{U}\left(  t+1\right)  -zS_{U}\left(  t\right)  \right] \\
&  =\left[  \left(  t+b+1\right)  \left(  t+b\right)  -z\right]  \nu_{0}%
^{U}+\left(  t+2b+1\right)  \nu_{1}^{U}+\nu_{2}^{U}.
\end{align*}

Stieltjes transform $M-$dependent difference equation%
\[
\left(  t+b+1\right)  \left(  t+1\right)  S_{U}\left(  t+1\right)
-zS_{U}\left(  t\right)  =\left(  t+b+1\right)  \nu_{0}+\nu_{1}+M\left[
\frac{\left(  t+b+1\right)  \left(  t+1\right)  }{t-\omega+1}-\frac
{z}{t-\omega}\right]  .
\]

\subsubsection{Christoffel generalized Charlier polynomials}

Let $\omega\notin\left\{  -b,0,\dfrac{\nu_{1}}{\nu_{0}}\right\}  .$

Special values%
\[
a\rightarrow-\omega+1,\quad b_{1}\rightarrow b,\quad b_{2}\rightarrow
-\omega-1.
\]

Linear functional%
\[
L_{C}\left[  r\right]  =%
%TCIMACRO{\dsum \limits_{x=0}^{\infty}}%
%BeginExpansion
{\displaystyle\sum\limits_{x=0}^{\infty}}
%EndExpansion
r\left(  x\right)  \left(  x-\omega\right)  \frac{1}{\left(  b+1\right)  _{x}%
}\frac{z^{x}}{x!}.
\]

Pearson equation%
\[
\frac{\varrho_{C}\left(  x+1\right)  }{\varrho_{C}\left(  x\right)  }%
=\frac{z\left(  x+1-\omega\right)  }{\left(  x+b+1\right)  \left(  x+1\right)
\left(  x-\omega\right)  }.
\]

Moments%
\[
\nu_{n}^{C}\left(  z\right)  =\nu_{n+1}+\left(  n-\omega\right)  \nu_{n},
\quad n\in\mathbb{N}_{0}.
\]

Stieltjes transform difference equation%
\begin{align*}
&  \left(  t-\omega\right)  \left(  t+b+1\right)  \left(  t+1\right)
S_{C}\left(  t+1\right)  -z\left(  t-\omega+1\right)  S_{C}\left(  t\right) \\
&  =\left[  -\omega t^{2}+\left(  z-\omega-b\omega+\omega^{2}\right)
t+\left(  z-z\omega+\omega^{2}+b\omega^{2}\right)  \right]  \nu_{0}+\left(
t-\omega\right)  \left(  t-\omega+1\right)  \nu_{1}.
\end{align*}

\subsubsection{Geronimus generalized Charlier polynomials}

Let $\omega\notin\left\{  -b-1,-1\right\}  \cup\mathbb{N}_{0}.$

Special values%
\[
a\rightarrow-\omega,\quad b_{1}\rightarrow b,\quad b_{2}\rightarrow-\omega.
\]

Linear functional%
\[
L_{G}\left[  r\right]  =%
%TCIMACRO{\dsum \limits_{x=0}^{\infty}}%
%BeginExpansion
{\displaystyle\sum\limits_{x=0}^{\infty}}
%EndExpansion
\frac{r\left(  x\right)  }{x-\omega}\frac{1}{\left(  b+1\right)  _{x}}%
\frac{z^{x}}{x!}+Mr\left(  \omega\right)  .
\]

Pearson equation%
\[
\frac{\varrho_{G}\left(  x+1\right)  }{\varrho_{G}\left(  x\right)  }%
=\frac{z\left(  x-\omega\right)  }{\left(  x+b+1\right)  \left(  x+1\right)
\left(  x+1-\omega\right)  }.
\]

Moments
\[
\nu_{n}^{G}\left(  z\right)  =\phi_{n}\left(  \omega\right)  \left[
M-S\left(  \omega\right)  +%
%TCIMACRO{\dsum \limits_{k=0}^{n-1}}%
%BeginExpansion
{\displaystyle\sum\limits_{k=0}^{n-1}}
%EndExpansion
\frac{\nu_{k}}{\phi_{k+1}\left(  \omega\right)  }\right]  ,\quad
n\in\mathbb{N}_{0}.
\]

Stieltjes transform difference equation%
\begin{align*}
&  \left(  t-\omega+1\right)  \left(  t+b+1\right)  \left(  t+1\right)
S_{G}\left(  t+1\right)  -\left(  t-\omega\right)  zS_{G}\left(  t\right) \\
&  =\left(  t+b+1\right)  \nu_{0}+\nu_{1}+\left[  \left(  t+b+1\right)
\left(  t+1\right)  -z\right]  \nu_{0}^{G}.
\end{align*}

\subsubsection{Truncated generalized Charlier polynomials}

Special values%
\[
a\rightarrow-N,\quad b_{1}\rightarrow b,\quad b_{2}\rightarrow-N-1.
\]

Linear functional%
\[
L_{T}\left[  r\right]  =%
%TCIMACRO{\dsum \limits_{x=0}^{N}}%
%BeginExpansion
{\displaystyle\sum\limits_{x=0}^{N}}
%EndExpansion
r\left(  x\right)  \ \frac{1}{\left(  b+1\right)  _{x}}\frac{z^{x}}{x!}.
\]

Pearson equation%
\[
\frac{\varrho_{T}\left(  x+1\right)  }{\varrho_{T}\left(  x\right)  }%
=\frac{z\left(  x-N\right)  }{\left(  x+b+1\right)  \left(  x-N\right)
\left(  x+1\right)  }.
\]

\strut Moments%
\[
\nu_{n}^{T}\left(  z\right)  =\frac{1}{\left(  b+1\right)  _{N}}\frac{z^{N}%
}{\left(  N-n\right)  !}\ _{3}F_{0}\left[
\begin{array}
[c]{c}%
n-N,1,\ -N-b\\
-
\end{array}
;\frac{1}{z}\right]  , \quad n\in\mathbb{N}_{0} .
\]

Stieltjes transform difference equation%
\begin{align*}
&  \left(  t-N\right)  \left[  \left(  t+b+1\right)  \left(  t+1\right)
S_{T}\left(  t+1\right)  -zS_{T}\left(  t\right)  \right] \\
&  =\left[  \left(  t+b+1\right)  \left(  t-N\right)  -z\right]  \nu_{0}%
^{T}+\left(  t+b+1-N\right)  \nu_{1}^{T}+\nu_{2}^{T}.
\end{align*}

\subsection{Polynomials of type (2,2)}

\strut Linear functional%
\[
L\left[  r\right]  =%
%TCIMACRO{\dsum \limits_{x=0}^{\infty}}%
%BeginExpansion
{\displaystyle\sum\limits_{x=0}^{\infty}}
%EndExpansion
r\left(  x\right)  \ \frac{\left(  a_{1}\right)  _{x}\left(  a_{2}\right)
_{x}}{\left(  b_{1}+1\right)  _{x}\left(  b_{2}+1\right)  _{x}}\frac{z^{x}%
}{x!}.
\]

Pearson equation%
\[
\frac{\varrho\left(  x+1\right)  }{\varrho\left(  x\right)  }=\frac{z\left(
x+a_{1}\right)  \left(  x+a_{2}\right)  }{\left(  x+b_{1}+1\right)  \left(
x+b_{2}+1\right)  \left(  x+1\right)  }.
\]

Moments%
\[
\nu_{n}\left(  z\right)  =z^{n}\frac{\left(  a_{1}\right)  _{n}\left(
a_{2}\right)  _{n}}{\left(  b_{1}+1\right)  _{n}\left(  b_{2}+1\right)  _{n}%
}\ _{2}F_{2}\left[
\begin{array}
[c]{c}%
a_{1}+n,a_{2}+n\\
b_{1}+1+n,b_{2}+1+n
\end{array}
;z\right]  , \quad n\in\mathbb{N}_{0} .
\]

Stieltjes transform difference equation%
\begin{align*}
&  \left(  t+b_{1}+1\right)  \left(  t+b_{2}+1\right)  \left(  t+1\right)
S\left(  t+1\right)  -z\left(  t+a_{1}\right)  \left(  t+a_{2}\right)
S\left(  t\right) \\
&  =\left[  \left(  t+b_{1}+1\right)  \left(  t+b_{2}+1\right)  -\left(
t+a_{1}+a_{2}\right)  z\right]  \nu_{0}+\left(  t+b_{1}+b_{2}+2-z\right)
\nu_{1}+\nu_{2}.
\end{align*}

\subsubsection{Uvarov Charlier polynomials}

Let $\omega\neq0.$

Special values%
\[
a_{1}\rightarrow-\omega,\quad a_{2}\rightarrow-\omega+1,\quad b_{1}%
\rightarrow-\omega-1,\quad b_{2}\rightarrow-\omega.
\]

Linear functional%
\[
L_{U}\left[  r\right]  =%
%TCIMACRO{\dsum \limits_{x=0}^{\infty}}%
%BeginExpansion
{\displaystyle\sum\limits_{x=0}^{\infty}}
%EndExpansion
r\left(  x\right)  \ \frac{z^{x}}{x!}+Mr\left(  \omega\right)  .
\]

Pearson equation%
\[
\frac{\varrho_{U}\left(  x+1\right)  }{\varrho_{U}\left(  x\right)  }%
=\frac{z\left(  x-\omega\right)  \left(  x-\omega+1\right)  }{\left(
x-\omega\right)  \left(  x-\omega+1\right)  \left(  x+1\right)  }.
\]

Moments%
\[
\nu_{n}^{U}\left(  z\right)  =z^{n}e^{z}+M\phi_{n}\left(  \omega\right)  ,
\quad n\in\mathbb{N}_{0} .
\]

Stieltjes transform difference equation%
\begin{align*}
&  \left(  t-\omega\right)  \left(  t-\omega+1\right)  \left[  \left(
t+1\right)  S_{U}\left(  t+1\right)  -zS_{U}\left(  t\right)  \right] \\
&  =\left[  \left(  t-\omega\right)  \left(  t-\omega+1\right)  -\left(
t-2\omega+1\right)  z\right]  \nu_{0}^{U}+\left(  t-z-2\omega+1\right)
\nu_{1}^{U}+\nu_{2}^{U}.
\end{align*}

Stieltjes transform $M-$dependent difference equation%
\[
\left(  t+1\right)  S_{U}\left(  t+1\right)  -zS_{U}\left(  t\right)  =\nu
_{0}+M\left(  \frac{t+1}{t+1-\omega}-\frac{z}{t-\omega}\right)  .
\]

\subsubsection{Reduced-Uvarov generalized Meixner polynomials}

Let $\left(  \omega,\Omega\right)  \in\left\{  \left(  -a,a+1\right)  ,\left(
-b,b\right)  ,\left(  0,0\right)  \right\}  .$

Special values%
\[
a_{1}\rightarrow a,\quad a_{2}\rightarrow\Omega,\quad b_{1}\rightarrow b,\quad
b_{2}\rightarrow\Omega-1.
\]

Linear functional%
\[
L_{U}\left[  r\right]  =%
%TCIMACRO{\dsum \limits_{x=0}^{\infty}}%
%BeginExpansion
{\displaystyle\sum\limits_{x=0}^{\infty}}
%EndExpansion
r\left(  x\right)  \ \frac{\left(  a\right)  _{x}}{\left(  b+1\right)  _{x}%
}\frac{z^{x}}{x!}+Mr\left(  \omega\right)  .
\]

Pearson equation%
\[
\frac{\varrho_{U}\left(  x+1\right)  }{\varrho_{U}\left(  x\right)  }%
=\frac{z\left(  x+a\right)  \left(  x+\Omega\right)  }{\left(  x+b+1\right)
\left(  x+\Omega\right)  \left(  x+1\right)  }.
\]

\strut

Moments%
\[
\nu_{n}^{U}\left(  z\right)  =z^{n}\frac{\left(  a\right)  _{n}}{\left(
b+1\right)  _{n}}\ _{1}F_{1}\left[
\begin{array}
[c]{c}%
a+n\\
b+1+n
\end{array}
;z\right]  +M\phi_{n}\left(  \omega\right)  , \quad n\in\mathbb{N}_{0} .
\]

Stieltjes transform difference equation%
\begin{align*}
&  \left(  t+\Omega\right)  \left[  \left(  t+b+1\right)  \left(  t+1\right)
S_{U}\left(  t+1\right)  -z\left(  t+a\right)  S_{U}\left(  t\right)  \right]
\\
&  =\left[  \left(  t+b+1\right)  \left(  t+\Omega\right)  -\left(
t+a+\Omega\right)  z\right]  \nu_{0}^{U}+\left(  t+b+1+\Omega-z\right)
\nu_{1}^{U}+\nu_{2}^{U}.
\end{align*}

Stieltjes transform $M-$dependent difference equation%
\begin{align*}
&  \left(  t+b+1\right)  \left(  t+1\right)  S_{U}\left(  t+1\right)
-z\left(  t+a\right)  S_{U}\left(  t\right) \\
&  =\left(  t+b+1-z\right)  \nu_{0}+\nu_{1}+M\left[  \frac{\left(
t+b+1\right)  \left(  t+1\right)  }{t+1-\omega}-\frac{z\left(  t+a\right)
}{t-\omega}\right]  .
\end{align*}

\subsubsection{Christoffel generalized Meixner polynomials}

Let $\omega\notin\left\{  -a,-b,0,\dfrac{\nu_{1}}{\nu_{0}}\right\}  .$

Special values%
\[
a_{1}\rightarrow a,\quad a_{2}\rightarrow-\omega+1,\quad b_{1}\rightarrow
b,\quad b_{2}\rightarrow-\omega-1.
\]

Linear functional%
\[
L_{C}\left[  r\right]  =%
%TCIMACRO{\dsum \limits_{x=0}^{\infty}}%
%BeginExpansion
{\displaystyle\sum\limits_{x=0}^{\infty}}
%EndExpansion
r\left(  x\right)  \left(  x-\omega\right)  \frac{\left(  a\right)  _{x}%
}{\left(  b+1\right)  _{x}}\frac{z^{x}}{x!}.
\]

Pearson equation%
\[
\frac{\varrho_{C}\left(  x+1\right)  }{\varrho_{C}\left(  x\right)  }%
=\frac{z\left(  x+a\right)  \left(  x+1-\omega\right)  }{\left(  x+1\right)
\left(  x+b+1\right)  \left(  x-\omega\right)  }.
\]

Moments%
\[
\nu_{n}^{C}\left(  z\right)  =\nu_{n+1}+\left(  n-\omega\right)  \nu_{n},
\quad n\in\mathbb{N}_{0}.
\]

Stieltjes transform difference equation%
\begin{align*}
&  \left(  t-\omega\right)  \left(  t+b+1\right)  \left(  t+1\right)
S_{C}\left(  t+1\right)  -z\left(  t-\omega+1\right)  \left(  t+a\right)
S_{C}\left(  t\right) \\
&  =\left[  -\omega t^{2}+\left(  z\omega-b\omega-\omega+\omega^{2}+az\right)
t+z\omega+\omega^{2}+b\omega^{2}-z\omega^{2}+az-az\omega\right]  \nu_{0}\\
&  +\left(  t-\omega\right)  \left(  t-\omega+1\right)  \nu_{1}.
\end{align*}

\subsubsection{Geronimus generalized Meixner polynomials}

Let $\omega\notin\left\{  -b-1,-1,1-a\right\}  \cup\mathbb{N}_{0}.$

Special values%
\[
a_{1}\rightarrow a,\quad a_{2}\rightarrow-\omega,\quad b_{1}\rightarrow
b,\quad b_{2}\rightarrow-\omega.
\]

Linear functional%
\[
L_{G}\left[  r\right]  =%
%TCIMACRO{\dsum \limits_{x=0}^{\infty}}%
%BeginExpansion
{\displaystyle\sum\limits_{x=0}^{\infty}}
%EndExpansion
\frac{r\left(  x\right)  }{x-\omega}\frac{\left(  a\right)  _{x}}{\left(
b+1\right)  _{x}}\frac{z^{x}}{x!}+Mr\left(  \omega\right)  .
\]

Pearson equation%
\[
\frac{\varrho_{G}\left(  x+1\right)  }{\varrho_{G}\left(  x\right)  }%
=\frac{z\left(  x+a\right)  \left(  x-\omega\right)  }{\left(  x+1\right)
\left(  x+b+1\right)  \left(  x+1-\omega\right)  }.
\]

Moments%
\[
\nu_{n}^{G}\left(  z\right)  =\phi_{n}\left(  \omega\right)  \left[
M-S\left(  \omega\right)  +%
%TCIMACRO{\dsum \limits_{k=0}^{n-1}}%
%BeginExpansion
{\displaystyle\sum\limits_{k=0}^{n-1}}
%EndExpansion
\frac{\nu_{k}}{\phi_{k+1}\left(  \omega\right)  }\right]  ,\quad
n\in\mathbb{N}_{0}.
\]

Stieltjes transform difference equation%
\begin{align*}
&  \left(  t-\omega+1\right)  \left(  t+b+1\right)  \left(  t+1\right)
S_{G}\left(  t+1\right)  -z\left(  t-\omega\right)  \left(  t+a\right)
S_{G}\left(  t\right) \\
&  =\left(  t+b+1-z\right)  \nu_{0}+\nu_{1}+\left[  \left(  t+b+1\right)
\left(  t+1\right)  -z\left(  t+a\right)  \right]  \nu_{0}^{G}.
\end{align*}

\subsubsection{Truncated generalized Meixner polynomials}

Special values%
\[
a_{1}\rightarrow a,\quad a_{2}\rightarrow-N,\quad b_{1}\rightarrow b,\quad
b_{2}\rightarrow-N-1.
\]

Linear functional%
\[
L_{T}\left[  r\right]  =%
%TCIMACRO{\dsum \limits_{x=0}^{N}}%
%BeginExpansion
{\displaystyle\sum\limits_{x=0}^{N}}
%EndExpansion
r\left(  x\right)  \ \frac{\left(  a\right)  _{x}}{\left(  b+1\right)  _{x}%
}\frac{z^{x}}{x!}.
\]

Pearson equation%
\[
\frac{\varrho_{T}\left(  x+1\right)  }{\varrho_{T}\left(  x\right)  }%
=\frac{z\left(  x+a\right)  \left(  x-N\right)  }{\left(  x+b+1\right)
\left(  x-N\right)  \left(  x+1\right)  }.
\]

Moments%
\[
\nu_{n}^{T}\left(  z\right)  =\frac{\left(  a\right)  _{N}\ }{\left(
b+1\right)  _{N}}\frac{z^{N}}{\left(  N-n\right)  !}\ _{3}F_{1}\left[
\begin{array}
[c]{c}%
n-N,1,\ -N-b\\
1-N-\ a
\end{array}
;\frac{1}{z}\right]  , \quad n\in\mathbb{N}_{0} .
\]

Stieltjes transform difference equation%
\begin{align*}
&  \left(  t-N\right)  \left[  \left(  t+b+1\right)  \left(  t+1\right)
S_{T}\left(  t+1\right)  -z\left(  t+a\right)  S_{T}\left(  t\right)  \right]
\\
&  =\left[  \left(  t+b+1\right)  \left(  t-N\right)  -\left(  t+a-N\right)
z\right]  \nu_{0}^{T}+\left(  t+b-N+1-z\right)  \nu_{1}^{T}+\nu_{2}^{T}.
\end{align*}

\subsection{Polynomials of type (3,0;N) \label{(30N)}}

\strut Linear functional%
\[
L\left[  r\right]  =%
%TCIMACRO{\dsum \limits_{x=0}^{N}}%
%BeginExpansion
{\displaystyle\sum\limits_{x=0}^{N}}
%EndExpansion
r\left(  x\right)  \ \left(  a_{1}\right)  _{x}\left(  a_{2}\right)
_{x}\left(  -N\right)  _{x}\frac{z^{x}}{x!}.
\]

Pearson equation%
\[
\frac{\varrho\left(  x+1\right)  }{\varrho\left(  x\right)  }=\frac{z\left(
x+a_{1}\right)  \left(  x+a_{2}\right)  \left(  x-N\right)  }{x+1}.
\]

Moments%
\[
\nu_{n}\left(  z\right)  =z^{n}\left(  -N\right)  _{n}\left(  a_{1}\right)
_{n}\left(  a_{2}\right)  _{n}\ _{3}F_{0}\left[
\begin{array}
[c]{c}%
-N+n,a_{1}+n,a_{2}+n\\
-
\end{array}
;z\right]  , \quad n\in\mathbb{N}_{0} .
\]

Stieltjes transform difference equation%
\begin{align*}
&  \left(  t+1\right)  S\left(  t+1\right)  -z\left(  t+a_{1}\right)  \left(
t+a_{2}\right)  \left(  t-N\right)  S\left(  t\right) \\
&  =\left[  \allowbreak-zt^{2}-\left(  a_{1}+a_{2}-N\right)  zt+z\left(
Na_{1}+Na_{2}-a_{1}a_{2}\right)  +1\right]  \nu_{0}\\
&  -\left(  t+a_{1}+a_{2}+1-N\right)  z\nu_{1}-z\nu_{2}.
\end{align*}

\subsection{Polynomials of type (3,1;N)}

Linear functional%
\[
L\left[  r\right]  =%
%TCIMACRO{\dsum \limits_{x=0}^{N}}%
%BeginExpansion
{\displaystyle\sum\limits_{x=0}^{N}}
%EndExpansion
r\left(  x\right)  \ \frac{\left(  a_{1}\right)  _{x}\left(  a_{2}\right)
_{x}\left(  -N\right)  _{x}}{\left(  b+1\right)  _{x}}\frac{z^{x}}{x!}.
\]

Pearson equation%
\[
\frac{\varrho\left(  x+1\right)  }{\varrho\left(  x\right)  }=\frac{z\left(
x+a_{1}\right)  \left(  x+a_{2}\right)  \left(  x-N\right)  }{\left(
x+b+1\right)  \left(  x+1\right)  }.
\]

Moments%
\[
\nu_{n}\left(  z\right)  =z^{n}\frac{\left(  -N\right)  _{n}\left(
a_{1}\right)  _{n}\left(  a_{2}\right)  _{n}}{\left(  b+1\right)  _{n}}%
\ _{3}F_{1}\left[
\begin{array}
[c]{c}%
-N+n,a_{1}+n,a_{2}+n\\
b+1+n
\end{array}
;z\right]  , \quad n\in\mathbb{N}_{0} .
\]

Stieltjes transform difference equation%
\begin{align*}
&  \left(  x+b+1\right)  \left(  t+1\right)  S\left(  t+1\right)  -z\left(
t+a_{1}\right)  \left(  t+a_{2}\right)  \left(  t-N\right)  S\left(  t\right)
\\
&  =\left[  \allowbreak-zt^{2}+t-\left(  a_{1}+a_{2}-N\right)  zt+\left(
Na_{1}+Na_{2}-a_{1}a_{2}\right)  z+b+1\right]  \nu_{0}\\
&  +\left[  1-\left(  t+a_{1}+a_{2}+1-N\right)  z\right]  \nu_{1}-z\nu_{2}.
\end{align*}
\strut

\subsubsection{Reduced-Uvarov generalized Krawtchouk polynomials}

Let $\left(  \omega,\Omega\right)  \in\left\{  \left(  -a,a+1\right)  ,\left(
0,0\right)  ,\left(  N,-N+1\right)  \right\}  .$

Special values%
\[
a_{1}\rightarrow a,\quad a_{2}\rightarrow\Omega,\quad b\rightarrow\Omega-1.
\]

Linear functional%
\[
L_{U}\left[  r\right]  =%
%TCIMACRO{\dsum \limits_{x=0}^{N}}%
%BeginExpansion
{\displaystyle\sum\limits_{x=0}^{N}}
%EndExpansion
r\left(  x\right)  \ \left(  a\right)  _{x}\left(  -N\right)  _{x}\frac{z^{x}%
}{x!}+Mr\left(  \omega\right)  .
\]

Pearson equation%
\[
\frac{\varrho_{U}\left(  x+1\right)  }{\varrho_{U}\left(  x\right)  }%
=\frac{z\left(  x+a\right)  \left(  x+\Omega\right)  \left(  x-N\right)
}{\left(  x+\Omega\right)  \left(  x+1\right)  }.
\]

Moments%
\[
\nu_{n}^{U}\left(  z\right)  =z^{n}\left(  -N\right)  _{n}\left(  a\right)
_{n}\ _{2}F_{0}\left[
\begin{array}
[c]{c}%
-N+n,a+n\\
-
\end{array}
;z\right]  +M\phi_{n}\left(  \omega\right)  , \quad n\in\mathbb{N}_{0} .
\]

Stieltjes transform difference equation%
\begin{align*}
&  \left(  t+\Omega\right)  \left[  \left(  t+1\right)  S_{U}\left(
t+1\right)  -z\left(  t+a\right)  \left(  t-N\right)  S_{U}\left(  t\right)
\right] \\
&  =\left[  \allowbreak-zt^{2}+t-\left(  a+\Omega-N\right)  zt+\left(
Na+N\Omega-a\Omega\right)  z+\Omega\right]  \nu_{0}^{U}\\
&  +\left[  1-\left(  t+a+\Omega+1-N\right)  z\right]  \nu_{1}^{U}-z\nu
_{2}^{U}.
\end{align*}

Stieltjes transform $M-$dependent difference equation%
\begin{align*}
&  \left(  t+1\right)  S_{U}\left(  t+1\right)  -z\left(  t+a\right)  \left(
t-N\right)  S_{U}\left(  t\right) \\
&  =\left[  1-\left(  t+a-N\right)  z\right]  \nu_{0}-z\nu_{1}+M\left[
\frac{t+1}{t-\omega+1}-\frac{z\left(  t+a\right)  \left(  t-N\right)
}{t-\omega}\right]  .
\end{align*}

\subsubsection{Christoffel generalized Krawtchouk polynomials}

Let $\omega\notin\left\{  -a,0,N,\dfrac{\nu_{1}}{\nu_{0}}\right\}  .$

Special values%
\[
a_{1}\rightarrow a,\quad a_{2}\rightarrow-\omega+1,\quad b\rightarrow
-\omega-1.
\]

Linear functional%
\[
L_{C}\left[  r\right]  =%
%TCIMACRO{\dsum \limits_{x=0}^{\infty}}%
%BeginExpansion
{\displaystyle\sum\limits_{x=0}^{\infty}}
%EndExpansion
r\left(  x\right)  \left(  x-\omega\right)  \left(  a\right)  _{x}\left(
-N\right)  _{x}\frac{z^{x}}{x!}.
\]

Pearson equation%
\[
\frac{\varrho_{C}\left(  x+1\right)  }{\varrho_{C}\left(  x\right)  }%
=\frac{z\left(  x+a\right)  \left(  x-N\right)  \left(  x+1-\omega\right)
}{\left(  x+1\right)  \left(  x-\omega\right)  }.
\]

Moments%
\[
\nu_{n}^{C}\left(  z\right)  =\nu_{n+1}+\left(  n-\omega\right)  \nu_{n},
\quad n\in\mathbb{N}_{0}.
\]

Stieltjes transform difference equation%
\begin{align*}
&  \left(  t-\omega\right)  \left(  t+1\right)  S_{C}\left(  t+1\right)
-z\left(  t-\omega+1\right)  \left(  t+a\right)  \left(  t-N\right)
S_{C}\left(  t\right) \\
&  =\left[  z\omega t^{2}-\left(  \omega-z\omega+z\omega^{2}+Nz\omega
-az\omega+Naz\right)  t\right]  \nu_{0}\\
&  +\left(  \omega^{2}-Nz\omega+az\omega+Nz\omega^{2}-az\omega^{2}%
-Naz+Naz\omega\right)  \nu_{0}-z\left(  t-\omega\right)  \left(
t-\omega+1\right)  \nu_{1}.
\end{align*}

\subsubsection{Geronimus generalized Krawtchouk polynomials}

Let $\omega\notin\left\{  -1,1-a,1+N\right\}  \cup\mathbb{N}_{0}.$

Special values%
\[
a_{1}\rightarrow a,\quad a_{2}\rightarrow-\omega,\quad b\rightarrow-\omega.
\]

Linear functional%
\[
L_{G}\left[  r\right]  =%
%TCIMACRO{\dsum \limits_{x=0}^{\infty}}%
%BeginExpansion
{\displaystyle\sum\limits_{x=0}^{\infty}}
%EndExpansion
\frac{r\left(  x\right)  }{x-\omega}\left(  a\right)  _{x}\left(  -N\right)
_{x}\frac{z^{x}}{x!}+Mr\left(  \omega\right)  .
\]

Pearson equation%
\[
\frac{\varrho_{G}\left(  x+1\right)  }{\varrho_{G}\left(  x\right)  }%
=\frac{z\left(  x+a\right)  \left(  x-N\right)  \left(  x-\omega\right)
}{\left(  x+1\right)  \left(  x+1-\omega\right)  }.
\]

Moments%
\[
\nu_{n}^{G}\left(  z\right)  =\phi_{n}\left(  \omega\right)  \left[
M-S\left(  \omega\right)  +%
%TCIMACRO{\dsum \limits_{k=0}^{n-1}}%
%BeginExpansion
{\displaystyle\sum\limits_{k=0}^{n-1}}
%EndExpansion
\frac{\nu_{k}}{\phi_{k+1}\left(  \omega\right)  }\right]  ,\quad
n\in\mathbb{N}_{0}.
\]

Stieltjes transform difference equation%
\begin{align*}
&  \left(  t-\omega+1\right)  \left(  t+1\right)  S_{G}\left(  t+1\right)
-z\left(  t-\omega\right)  \left(  t+a\right)  \left(  t-N\right)
S_{G}\left(  t\right) \\
&  =\left(  -zt+Nz-az+1\right)  \nu_{0}-z\nu_{1}+\left[  t+1-z\left(
t+a\right)  \left(  t-N\right)  \right]  \nu_{0}^{G}.
\end{align*}

\subsection{Polynomials of type (3,2)}

\strut Linear functional%
\[
L\left[  r\right]  =%
%TCIMACRO{\dsum \limits_{x=0}^{\infty}}%
%BeginExpansion
{\displaystyle\sum\limits_{x=0}^{\infty}}
%EndExpansion
r\left(  x\right)  \ \frac{\left(  a_{1}\right)  _{x}\left(  a_{2}\right)
_{x}\left(  a_{3}\right)  _{x}}{\left(  b_{1}+1\right)  _{x}\left(
b_{2}+1\right)  _{x}}\frac{z^{x}}{x!}.
\]

Pearson equation%
\[
\frac{\varrho\left(  x+1\right)  }{\varrho\left(  x\right)  }=\frac{z\left(
x+a_{1}\right)  \left(  x+a_{2}\right)  \left(  x+a_{3}\right)  }{\left(
x+b_{1}+1\right)  \left(  x+b_{2}+1\right)  \left(  x+1\right)  }.
\]

Moments%

\[
\nu_{n}\left(  z\right)  =z^{n}\frac{\left(  a_{1}\right)  _{n}\left(
a_{2}\right)  _{n}\left(  a_{3}\right)  _{n}}{\left(  b_{1}+1\right)
_{n}\left(  b_{2}+1\right)  _{n}}\ _{3}F_{2}\left[
\begin{array}
[c]{c}%
a_{1}+n,a_{2}+n,a_{3}+n\\
b_{1}+1+n,b_{2}+1+n
\end{array}
;z\right]  , \quad n\in\mathbb{N}_{0} ,
\]
where we choose the principal branch $z\in\mathbb{C}\setminus\lbrack
1,\infty).$ Note that since%
\[
b_{1}+1+n+b_{2}+1+n-\left(  a_{1}+n+a_{2}+n+a_{3}+n\right)  \leq0,\quad n\geq
b_{1}+b_{2}-a_{2}-a_{3}-a_{1}+2,
\]
$\nu_{n}\left(  1\right)  $ is undefined for $n\geq b_{1}+b_{2}-a_{2}%
-a_{3}-a_{1}+2.$

Stieltjes transform difference equation%

\[
\left(  t+b_{1}+1\right)  \left(  t+b_{2}+1\right)  \left(  t+1\right)
S\left(  t+1\right)  -z\left(  t+a_{1}\right)  \left(  t+a_{2}\right)  \left(
t+a_{3}\right)  S\left(  t\right)  =%
%TCIMACRO{\dsum \limits_{k=0}^{2}}%
%BeginExpansion
{\displaystyle\sum\limits_{k=0}^{2}}
%EndExpansion
\xi_{k}t^{k},
\]
$\allowbreak$where%
\begin{align*}
&  \xi_{2}=\left(  1-z\right)  \nu_{0},\quad\xi_{1}=\left[  e_{1}\left(
\mathbf{b}+1\right)  -e_{1}\left(  \mathbf{a}\right)  z\right]  \nu
_{0}+\left(  1-z\right)  \nu_{1},\\
&  \xi_{0}=\left[  e_{2}\left(  \mathbf{b}+1\right)  -e_{2}\left(
\mathbf{a}\right)  z\right]  \nu_{0}+\left[  e_{1}\left(  \mathbf{b}+1\right)
-e_{1}\left(  \mathbf{a}\right)  z-z\right]  \nu_{1}+\left(  1-z\right)
\nu_{2},
\end{align*}
and $e_{1},e_{2},e_{3}$ denote the \emph{elementary symmetric polynomials}
defined by%
\[
e_{1}\left(  \mathbf{x}\right)  =%
%TCIMACRO{\dsum \limits_{i}}%
%BeginExpansion
{\displaystyle\sum\limits_{i}}
%EndExpansion
x_{i},\quad e_{2}\left(  \mathbf{x}\right)  =%
%TCIMACRO{\dsum \limits_{i<j}}%
%BeginExpansion
{\displaystyle\sum\limits_{i<j}}
%EndExpansion
x_{i}x_{j},\quad e_{3}\left(  \mathbf{x}\right)  =%
%TCIMACRO{\dsum \limits_{i<j<k}}%
%BeginExpansion
{\displaystyle\sum\limits_{i<j<k}}
%EndExpansion
x_{i}x_{j}x_{k}.
\]
$\allowbreak$

\subsubsection{Uvarov Meixner polynomials}

Let $\omega\notin\left\{  -a,0\right\}  .$

Special values%
\[
a_{1}\rightarrow a,\quad a_{2}\rightarrow-\omega,\quad a_{3}\rightarrow
-\omega+1,\quad b_{1}\rightarrow-\omega-1,\quad b_{2}\rightarrow-\omega.
\]

Linear functional%
\[
L_{U}\left[  r\right]  =%
%TCIMACRO{\dsum \limits_{x=0}^{\infty}}%
%BeginExpansion
{\displaystyle\sum\limits_{x=0}^{\infty}}
%EndExpansion
r\left(  x\right)  \ \left(  a\right)  _{x}\ \frac{z^{x}}{x!}+Mr\left(
\omega\right)  .
\]

Pearson equation%
\[
\frac{\varrho_{U}\left(  x+1\right)  }{\varrho_{U}\left(  x\right)  }%
=\frac{z\left(  x+a\right)  \left(  x-\omega\right)  \left(  x-\omega
+1\right)  }{\left(  x-\omega\right)  \left(  x-\omega+1\right)  \left(
x+1\right)  }.
\]

Moments%
\[
\nu_{n}^{U}\left(  z\right)  =z^{n}\left(  a\right)  _{n}\left(  1-z\right)
^{-a-n}+M\phi_{n}\left(  \omega\right)  , \quad n\in\mathbb{N}_{0} .
\]

Stieltjes transform difference equation%
\begin{align*}
&  \left(  t-\omega\right)  \left(  t-\omega+1\right)  \left[  \left(
t+1\right)  S_{U}\left(  t+1\right)  -z\left(  t+a\right)  S_{U}\left(
t\right)  \right] \\
&  =\left[  \allowbreak(1-z)t^{2}+(-az+2\omega z-2\omega-z+1)t+\omega\left(
z+\omega-z\omega+2az-1\right)  -az\right]  \nu_{0}^{U\varsigma}\\
&  +\left[  (1-z)t-2z+2z\omega-2\omega-az+1\right]  \nu_{1}^{U}+\left(
1-z\right)  \nu_{2}^{U}.
\end{align*}

Stieltjes transform $M-$dependent difference equation%
\[
\left(  t+1\right)  S_{U}\left(  t+1\right)  -z\left(  t+a\right)
S_{U}\left(  t\right)  =\left(  1-z\right)  \nu_{0}+M\left[  \frac
{t+1}{t+1-\omega}-\frac{z\left(  t+a\right)  }{t-\omega}\right]  .
\]
\strut

\subsubsection{Symmetrized generalized Meixner polynomials}

Special values%
\begin{align*}
a_{1}  &  \rightarrow a,\quad a_{2}\rightarrow-b-2m,\quad a_{3}\rightarrow
-2m,\\
b_{1}  &  \rightarrow-a-2m,\quad b_{2}\rightarrow b-2m,\quad z\rightarrow
-1,\quad x\rightarrow x+m.
\end{align*}

Linear functional%
\[
L_{\varsigma}\left[  r\right]  =\varrho_{\varsigma}\left(  0\right)
%TCIMACRO{\dsum \limits_{x=-m}^{m}}%
%BeginExpansion
{\displaystyle\sum\limits_{x=-m}^{m}}
%EndExpansion
r\left(  x\right)  \ \frac{\left(  a+m\right)  _{x}\left(  -b-m\right)
_{x}\left(  -m\right)  _{x}}{\left(  -a-m+1\right)  _{x}\left(  b+1-m\right)
_{x}}\frac{\left(  -1\right)  ^{x}}{\left(  m+1\right)  _{x}},
\]
where%
\[
\varrho_{\varsigma}\left(  0\right)  =\binom{2m}{m}\frac{\left(  a\right)
_{m}\ \left(  b+1+m\right)  _{m}}{\left(  a+m\right)  _{m}\ \left(
b+1-2m\right)  _{m}}.
\]

Pearson equation%
\[
\frac{\varrho_{\varsigma}\left(  x+1\right)  }{\varrho_{\varsigma}\left(
x\right)  }=\frac{-\left(  x+a+m\right)  \left(  x-b-m\right)  \left(
x-m\right)  }{\left(  x-a-m+1\right)  \left(  x+b+1+m\right)  \left(
x+1+m\right)  }.
\]

\strut Moments on the basis $\{\overline{\phi}_{n}\left(  x\right)
\}_{n\geq0} $%

\[
\nu_{n}^{\varsigma}=\frac{\left(  -1\right)  ^{n}\left(  a\right)  _{n}\left(
-b-2m\right)  _{n}\left(  -2m\right)  _{n}}{\left(  -a-2m+1\right)
_{n}\left(  b-2m+1\right)  _{n}}\ _{3}F_{2}\left[
\begin{array}
[c]{c}%
a+n,-b-2m+n,-2m+n\\
-a-2m+1+n,b-2m+1+n
\end{array}
;-1\right]  , \quad n\in\mathbb{N}_{0} .
\]

Stieltjes transform difference equation%
\begin{align*}
&  \left(  t-a-m+1\right)  \left(  t+b+1+m\right)  \left(  t+1+m\right)
S_{\varsigma}\left(  t+1\right)  +\left(  t+a+m\right)  \left(  t-b-m\right)
\left(  t-m\right)  S_{\varsigma}\left(  t\right) \\
&  =\left[  2t^{2}+2\left(  1-m\right)  t+b-a-2ab-4am+1\right]  \nu
_{0}^{\varsigma}+\left(  2t-4m+3\right)  \nu_{1}^{\varsigma}+2\nu
_{2}^{\varsigma}.
\end{align*}

\subsubsection{Reduced-Uvarov generalized Hahn polynomials of type I}

Let $\left\{  \omega,\Omega\right\}  \in\left\{  \left(  -a_{1},a_{1}%
+1\right)  ,\left(  -a_{2},a_{2}+1\right)  ,\left(  -b,b\right)  ,\left(
0,0\right)  \right\}  .$

Special values%
\[
a_{3}\rightarrow\Omega,\quad b_{1}\rightarrow b,\quad b_{2}\rightarrow
\Omega-1.
\]

Linear functional%
\[
L_{U}\left[  r\right]  =%
%TCIMACRO{\dsum \limits_{x=0}^{\infty}}%
%BeginExpansion
{\displaystyle\sum\limits_{x=0}^{\infty}}
%EndExpansion
r\left(  x\right)  \ \frac{\left(  a_{1}\right)  _{x}\left(  a_{2}\right)
_{x}}{\left(  b+1\right)  _{x}}\frac{z^{x}}{x!}+Mr\left(  \omega\right)  .
\]

Pearson equation%
\[
\frac{\varrho_{U}\left(  x+1\right)  }{\varrho_{U}\left(  x\right)  }%
=\frac{z\left(  x+a_{1}\right)  \left(  x+a_{2}\right)  \left(  x+\Omega
\right)  }{\left(  x+b+1\right)  \left(  x+\Omega\right)  \left(  x+1\right)
}.
\]

Moments%
\[
\nu_{n}^{U}\left(  z\right)  =z^{n}\frac{\left(  a_{1}\right)  _{n}\left(
a_{2}\right)  _{n}}{\left(  b+1\right)  _{n}}\ _{2}F_{1}\left[
\begin{array}
[c]{c}%
a_{1}+n,a_{2}+n\\
b+1+n
\end{array}
;z\right]  +M\phi_{n}\left(  \omega\right)  , \quad n\in\mathbb{N}_{0} .
\]

Stieltjes transform difference equation%
\begin{align*}
&  \left(  t+\Omega\right)  \left[  \left(  t+b+1\right)  \left(  t+1\right)
S_{U}\left(  t+1\right)  -z\left(  t+a_{1}\right)  \left(  t+a_{2}\right)
S_{U}\left(  t\right)  \right] \\
&  =\left[  (1-z)t^{2}+\left(  b+\Omega-z\Omega-za_{1}-za_{2}+1\right)
t-\left(  \Omega a_{1}+\Omega a_{2}+a_{1}a_{2}\right)  z+(b+1)\Omega\right]
\nu_{0}^{U}\\
&  +\left[  (1-z)t+1-z\left(  \Omega+a_{1}+a_{2}+1\right)  +\Omega+b\right]
\nu_{1}^{U}+\left(  1-z\right)  \nu_{2}^{U}.
\end{align*}

Stieltjes transform $M-$dependent difference equation%
\begin{align*}
&  \left(  t+b+1\right)  \left(  t+1\right)  S_{U}\left(  t+1\right)
-z\left(  t+a_{1}\right)  \left(  t+a_{2}\right)  S_{U}\left(  t\right)
=\left[  t+b+1-z\left(  t+a_{1}+a_{2}\right)  \right]  \nu_{0}\\
&  +\left(  1-z\right)  \nu_{1}+M\left[  \frac{\left(  t+b+1\right)  \left(
t+1\right)  }{t-\omega+1}-\frac{z\left(  t+a_{1}\right)  \left(
t+a_{2}\right)  }{t-\omega}\right]  .
\end{align*}

\subsubsection{Christoffel generalized Hahn polynomials of type I}

Let $\omega\notin\left\{  -a_{1},-a_{2},-b,0,\dfrac{\nu_{1}}{\nu_{0}}\right\}
.$

Special values%
\[
a_{3}\rightarrow-\omega+1,\quad b_{1}\rightarrow b,\quad b_{2}\rightarrow
-\omega-1.
\]

Linear functional%
\[
L_{C}\left[  r\right]  =%
%TCIMACRO{\dsum \limits_{x=0}^{\infty}}%
%BeginExpansion
{\displaystyle\sum\limits_{x=0}^{\infty}}
%EndExpansion
r\left(  x\right)  \left(  x-\omega\right)  \frac{\left(  a_{1}\right)
_{x}\left(  a_{2}\right)  _{x}}{\left(  b+1\right)  _{x}}\frac{z^{x}}{x!}.
\]

Pearson equation%
\[
\frac{\varrho_{C}\left(  x+1\right)  }{\varrho_{C}\left(  x\right)  }%
=\frac{z\left(  x+a_{1}\right)  \left(  x+a_{2}\right)  \left(  x+1-\omega
\right)  }{\left(  x+b+1\right)  \left(  x+1\right)  \left(  x-\omega\right)
}, \quad n\in\mathbb{N}_{0}.
\]

Moments%
\[
\nu_{n}^{C}\left(  z\right)  =\nu_{n+1}+\left(  n-\omega\right)  \nu_{n}.
\]

Stieltjes transform difference equation%
\begin{align*}
&  \left(  t-\omega\right)  \left(  t+b+1\right)  \left(  t+1\right)
S_{C}\left(  t+1\right)  -z\left(  t-\omega+1\right)  \left(  t+a_{1}\right)
\left(  t+a_{2}\right)  S_{C}\left(  t\right) \\
&  =\left[  \left(  z-1\right)  \omega t^{2}-\left(  \omega+b\omega
-z\omega-\omega^{2}+z\omega^{2}-z\omega a_{1}-z\omega a_{2}-za_{1}%
a_{2}\right)  t\right]  \nu_{0}\\
&  +\left[  \omega^{2}\left(  b+1\right)  -z\left(  \omega-1\right)  \left(
\omega a_{1}+\omega a_{2}+a_{1}a_{2}\right)  \right]  \nu_{0}-\left(
z-1\right)  \left(  t-\omega\right)  \left(  t-\omega+1\right)  \nu_{1}.
\end{align*}

\subsubsection{Geronimus generalized Hahn polynomials of type I}

Let $\omega\notin\left\{  -b-1,-1,1-a_{1},1-a_{2}\right\}  \cup\mathbb{N}%
_{0}.$

Special values%
\[
a_{3}\rightarrow-\omega,\quad b_{1}\rightarrow b,\quad b_{2}\rightarrow
-\omega.
\]

Linear functional%
\[
L_{G}\left[  r\right]  =%
%TCIMACRO{\dsum \limits_{x=0}^{\infty}}%
%BeginExpansion
{\displaystyle\sum\limits_{x=0}^{\infty}}
%EndExpansion
\frac{r\left(  x\right)  }{x-\omega}\frac{\left(  a_{1}\right)  _{x}\left(
a_{2}\right)  _{x}}{\left(  b+1\right)  _{x}}\frac{z^{x}}{x!}+Mr\left(
\omega\right)  .
\]

Pearson equation%
\[
\frac{\varrho_{G}\left(  x+1\right)  }{\varrho_{G}\left(  x\right)  }%
=\frac{z\left(  x+a_{1}\right)  \left(  x+a_{2}\right)  \left(  x-\omega
\right)  }{\left(  x+b+1\right)  \left(  x+1\right)  \left(  x+1-\omega
\right)  }.
\]

Moments%
\[
\nu_{n}^{G}\left(  z\right)  =\phi_{n}\left(  \omega\right)  \left[
M-S\left(  \omega\right)  +%
%TCIMACRO{\dsum \limits_{k=0}^{n-1}}%
%BeginExpansion
{\displaystyle\sum\limits_{k=0}^{n-1}}
%EndExpansion
\frac{\nu_{k}}{\phi_{k+1}\left(  \omega\right)  }\right]  ,\quad
n\in\mathbb{N}_{0}.
\]

Stieltjes transform difference equation%
\begin{align*}
&  \left(  t-\omega+1\right)  \left(  t+b+1\right)  \left(  t+1\right)
S_{G}\left(  t+1\right)  -z\left(  t-\omega\right)  \left(  t+a_{1}\right)
\left(  t+a_{2}\right)  S_{G}\left(  t\right) \\
&  =\left[  \left(  1-z\right)  t+\left(  b-za_{1}-za_{2}+1\right)  \right]
\nu_{0}+\left(  1-z\right)  \nu_{1}\\
&  +\left[  \left(  t+b+1\right)  \left(  t+1\right)  -z\left(  t+a_{1}%
\right)  \left(  t+a_{2}\right)  \right]  \nu_{0}^{G}.
\end{align*}

\subsubsection{Truncated generalized Hahn polynomials of type I}

Special values%
\[
a_{3}\rightarrow-N,\quad b_{1}\rightarrow b,\quad b_{2}\rightarrow-N-1.
\]

Linear functional%
\[
L_{T}\left[  r\right]  =%
%TCIMACRO{\dsum \limits_{x=0}^{N}}%
%BeginExpansion
{\displaystyle\sum\limits_{x=0}^{N}}
%EndExpansion
r\left(  x\right)  \ \frac{\left(  a_{1}\right)  _{x}\left(  a_{2}\right)
_{x}}{\left(  b+1\right)  _{x}}\frac{z^{x}}{x!}.
\]

Pearson equation%
\[
\frac{\varrho_{T}\left(  x+1\right)  }{\varrho_{T}\left(  x\right)  }%
=\frac{z\left(  x+a_{1}\right)  \left(  x+a_{2}\right)  \left(  x-N\right)
}{\left(  x+b+1\right)  \left(  x-N\right)  \left(  x+1\right)  }.
\]

Moments%
\[
\nu_{n}^{T}\left(  z\right)  =\frac{\left(  a_{1}\right)  _{N}\ \left(
a_{2}\right)  _{N}\ }{\left(  b+1\right)  _{N}}\frac{z^{N}}{\left(
N-n\right)  !}\ _{3}F_{2}\left[
\begin{array}
[c]{c}%
n-N,1,\ -N-b\\
1-N-\ a_{1},1-N-\ a_{2}%
\end{array}
;\frac{1}{z}\right]  , \quad n\in\mathbb{N}_{0} .
\]

Stieltjes transform difference equation%
\begin{align*}
&  \left(  t-N\right)  \left[  \left(  t+b+1\right)  \left(  t+1\right)
S_{U}\left(  t+1\right)  -z\left(  t+a_{1}\right)  \left(  t+a_{2}\right)
S_{U}\left(  t\right)  \right] \\
&  =\left[  (1-z)t^{2}+\left(  b-N+zN-za_{1}-za_{2}+1\right)  t-\left(
-Na_{1}-Na_{2}+a_{1}a_{2}\right)  z-(b+1)N\right]  \nu_{0}^{U}\\
&  +\left[  (1-z)t+1-z\left(  -N+a_{1}+a_{2}+1\right)  -N+b\right]  \nu
_{1}^{U}+\left(  1-z\right)  \nu_{2}^{U}.
\end{align*}

\subsubsection{Symmetrized polynomials of type $(0,2)$}

Special values%
\[
a_{1}\rightarrow-b_{1}-2m,\quad a_{2}\rightarrow-b_{2}-2m,\quad a_{3}%
\rightarrow-2m,\quad z\rightarrow-1,\quad x\rightarrow x+m.
\]

Linear functional%
\[
L_{\varsigma}\left[  r\right]  =\varrho_{\varsigma}\left(  0\right)
%TCIMACRO{\dsum \limits_{x=-m}^{m}}%
%BeginExpansion
{\displaystyle\sum\limits_{x=-m}^{m}}
%EndExpansion
r\left(  x\right)  \ \frac{\left(  -b_{1}-m\right)  _{x}\left(  -b_{2}%
-m\right)  _{x}\left(  -m\right)  _{x}}{\left(  b_{1}+1+m\right)  _{x}\left(
b_{2}+1+m\right)  _{x}}\frac{\left(  -1\right)  ^{x}}{\left(  m+1\right)
_{x}},
\]
where%
\[
\varrho_{\varsigma}\left(  0\right)  =\binom{2m}{m}\frac{\ \left(
b_{1}+1+m\right)  _{m}\left(  b_{1}+1+m\right)  _{m}}{\ \left(  b_{1}%
+1\right)  _{m}\left(  b_{2}+1\right)  _{m}}.
\]

Pearson equation%
\[
\frac{\varrho_{\varsigma}\left(  x+1\right)  }{\varrho_{\varsigma}\left(
x\right)  }=\frac{-\left(  x-b_{1}-m\right)  \left(  x-b_{2}-m\right)  \left(
x-m\right)  }{\left(  x+b_{1}+1+m\right)  \left(  x+b_{2}+1+m\right)  \left(
x+1+m\right)  }.
\]

\strut Moments on the basis $\{\overline{\phi}_{n}\left(  x\right)
\}_{n\geq0}. $

For $n\in\mathbb{N}_{0}$%

\[
\nu_{n}^{\varsigma}=\frac{\left(  -1\right)  ^{n}\left(  -2m\right)
_{n}\left(  -b_{1}-2m\right)  _{n}\left(  -b_{2}-2m\right)  _{n}}{\left(
b_{1}+1\right)  _{n}\left(  b_{2}+1\right)  _{n}}\ _{3}F_{2}\left[
\begin{array}
[c]{c}%
-2m+n,-b_{1}-2m,-b_{2}-2m+n,\\
b_{1}+1+n,b_{2}+1+n
\end{array}
;-1\right]  .
\]

Stieltjes transform difference equation%
\begin{align*}
&  \left(  t+b_{1}+1+m\right)  \left(  t+b_{2}+1+m\right)  \left(
t+1+m\right)  S_{\varsigma}\left(  t+1\right) \\
&  +\left(  t-b_{1}-m\right)  \left(  t-b_{2}-m\right)  \left(  t-m\right)
S_{\varsigma}\left(  t\right) \\
&  =\left[  2t^{2}+2\left(  1-m\right)  t+8m^{2}+2m\left(  2b_{1}%
+2b_{2}+1\right)  +b_{1}+b_{2}+2b_{1}b_{2}+1\right]  \nu_{0}^{\varsigma}\\
&  +\left(  2t-4m+3\right)  \nu_{1}^{\varsigma}+2\nu_{2}^{\varsigma}.
\end{align*}

\subsection{\strut Polynomials of type (4,3;N,1)}

\strut Linear functional%
\[
L\left[  r\right]  =%
%TCIMACRO{\dsum \limits_{x=0}^{N}}%
%BeginExpansion
{\displaystyle\sum\limits_{x=0}^{N}}
%EndExpansion
r\left(  x\right)  \ \frac{\left(  a_{1}\right)  _{x}\left(  a_{2}\right)
_{x}\left(  a_{3}\right)  _{x}\left(  -N\right)  _{x}}{\left(  b_{1}+1\right)
_{x}\left(  b_{2}+1\right)  _{x}\left(  b_{3}+1\right)  _{x}}\frac{1}{x!}.
\]

Pearson equation%
\[
\frac{\varrho\left(  x+1\right)  }{\varrho\left(  x\right)  }=\frac{\left(
x+a_{1}\right)  \left(  x+a_{2}\right)  \left(  x+a_{3}\right)  \left(
x-N\right)  }{\left(  x+b_{1}+1\right)  \left(  x+b_{2}+1\right)  \left(
x+b_{3}+1\right)  \left(  x+1\right)  }.
\]

Moments

For $n\in\mathbb{N}_{0}$%
\[
\nu_{n}=\frac{\left(  -N\right)  _{n}\left(  a_{1}\right)  _{n}\left(
a_{2}\right)  _{n}\left(  a_{3}\right)  _{n}}{\left(  b_{1}+1\right)
_{n}\left(  b_{2}+1\right)  _{n}\left(  b_{3}+1\right)  _{n}}\ _{4}%
F_{3}\left[
\begin{array}
[c]{c}%
-N+n,a_{1}+n,a_{2}+n,a_{3}+n\\
b_{1}+1+n,b_{2}+1+n,b_{3}+1+n
\end{array}
;1\right]  .
\]

Stieltjes transform difference equation%
\begin{gather*}
\left(  t+b_{1}+1\right)  \left(  t+b_{2}+1\right)  \left(  t+b_{3}+1\right)
\left(  t+1\right)  S\left(  t+1\right) \\
-\left(  t+a_{1}\right)  \left(  t+a_{2}\right)  \left(  t+a_{3}\right)
\left(  t-N\right)  S(t)=%
%TCIMACRO{\dsum \limits_{k=0}^{2}}%
%BeginExpansion
{\displaystyle\sum\limits_{k=0}^{2}}
%EndExpansion
\xi_{k}t^{k},
\end{gather*}
where%
\begin{align*}
\xi_{2}  &  =\left[  e_{1}\left(  \mathbf{b}+1\right)  -e_{1}\left(
\mathbf{a}\right)  \right]  \nu_{0},\quad\xi_{1}=\left[  e_{2}\left(
\mathbf{b}+1\right)  -e_{2}\left(  \mathbf{a}\right)  \right]  \nu_{0}+\left[
e_{1}\left(  \mathbf{b}+1\right)  -e_{1}\left(  \mathbf{a}\right)  -1\right]
\nu_{1},\\
\xi_{0}  &  =\left[  e_{3}\left(  \mathbf{b}+1\right)  -e_{3}\left(
\mathbf{a}\right)  \right]  \nu_{0}+\left[  e_{2}\left(  \mathbf{b}+1\right)
-e_{2}\left(  \mathbf{a}\right)  -e_{1}\left(  \mathbf{a}\right)  -1\right]
\nu_{1}+\left[  e_{1}\left(  \mathbf{b}+1\right)  -e_{1}\left(  \mathbf{a}%
\right)  -2\right]  \nu_{2},
\end{align*}
with $a_{4}=-N.$

If we use the relations%
\[
e_{1}\left(  \mathbf{a},-N\right)  =e_{1}\left(  \mathbf{a}\right)  -N,\quad
e_{2}\left(  \mathbf{a},-N\right)  =e_{2}\left(  \mathbf{a}\right)
-Ne_{1}\left(  \mathbf{a}\right)  ,\quad e_{3}\left(  \mathbf{a},-N\right)
=e_{3}\left(  \mathbf{a}\right)  -Ne_{2}\left(  \mathbf{a}\right)  ,
\]
we obtain expressions for the coefficients $\xi_{k}$ explicitly depending on
$N$
\begin{align*}
\xi_{2}  &  =\left[  e_{1}\left(  \mathbf{b}+1\right)  -e_{1}\left(
\mathbf{a}\right)  +N\right]  \nu_{0},\\
\xi_{1}  &  =\left[  e_{2}\left(  \mathbf{b}+1\right)  -e_{2}\left(
\mathbf{a}\right)  +Ne_{1}\left(  \mathbf{a}\right)  \right]  \nu_{0}+\left[
e_{1}\left(  \mathbf{b}+1\right)  -e_{1}\left(  \mathbf{a}\right)
+N-1\right]  \nu_{1},\\
\xi_{0}  &  =\left[  e_{3}\left(  \mathbf{b}+1\right)  -e_{3}\left(
\mathbf{a}\right)  +Ne_{2}\left(  \mathbf{a}\right)  \right]  \nu_{0}+\left[
e_{2}\left(  \mathbf{b}+1\right)  -e_{2}\left(  \mathbf{a}\right)  +\left(
N-1\right)  \left(  e_{1}\left(  \mathbf{a}\right)  +1\right)  \right]
\nu_{1}\\
&  +\left[  e_{1}\left(  \mathbf{b}+1\right)  -e_{1}\left(  \mathbf{a}\right)
+N-2\right]  \nu_{2}.
\end{align*}

\subsubsection{Uvarov Hahn polynomials}

Let $\omega\notin\left\{  -a,-b,0,N\right\}  .$

Special values%

\begin{align*}
a_{1}  &  \rightarrow a,\quad a_{2}\rightarrow-\omega,\quad a_{3}%
\rightarrow-\omega+1,\\
b_{1}  &  \rightarrow b,\quad b_{2}\rightarrow-\omega-1,\quad b_{3}%
\rightarrow-\omega.
\end{align*}

Linear functional%

\[
L_{U}\left[  r\right]  =%
%TCIMACRO{\dsum \limits_{x=0}^{N}}%
%BeginExpansion
{\displaystyle\sum\limits_{x=0}^{N}}
%EndExpansion
r\left(  x\right)  \ \frac{\left(  a\right)  _{x}\left(  -N\right)  _{x}%
}{\left(  b+1\right)  _{x}}\ \frac{1}{x!}+Mr\left(  \omega\right)  .
\]

Pearson equation%
\[
\frac{\varrho_{U}\left(  x+1\right)  }{\varrho_{U}\left(  x\right)  }%
=\frac{\left(  x+a\right)  \left(  x-\omega\right)  \left(  x-\omega+1\right)
\left(  x-N\right)  }{\left(  x+b+1\right)  \left(  x-\omega\right)  \left(
x-\omega+1\right)  \left(  x+1\right)  }.
\]

Moments%
\[
\nu_{n}^{U}=\frac{\left(  -N\right)  _{n}\left(  a\right)  _{n}}{\left(
b+1\right)  _{n}}\frac{\left(  b+1-a\right)  _{N-n}}{\left(  b+1+n\right)
_{N-n}}+M\phi_{n}\left(  \omega\right)  , \quad n\in\mathbb{N}_{0} .
\]

Stieltjes transform difference equation%
\begin{align*}
&  \left(  t-\omega\right)  \left(  t-\omega+1\right)  \left[  \left(
t+b+1\right)  \left(  t+1\right)  S_{U}\left(  t+1\right)  -\left(
t+a\right)  \left(  t-N\right)  S_{U}\left(  t\right)  \right] \\
&  =\left[  \allowbreak\allowbreak\left(  N-a+b+1\right)  t^{2}+\left(
N-a+b-2\omega-2N\omega+2a\omega-2b\omega+Na+1\right)  t\right]  \nu
_{0}^{U\varsigma}\\
&  +\left(  a\omega-N\omega-\omega-b\omega+\omega^{2}+N\omega^{2}-a\omega
^{2}+b\omega^{2}+Na-2\allowbreak Na\omega\right)  \nu_{0}^{U\varsigma}\\
&  +\left[  \allowbreak\left(  N-a+b\right)  t+2N-2a+b-2N\omega+2a\omega
-2b\omega+Na-1\right]  \nu_{1}^{U}+\left(  -a-1+b+N\right)  \nu_{2}^{U}.
\end{align*}

Stieltjes transform $M-$dependent difference equation%
\begin{align*}
&  \left(  t+b+1\right)  \left(  t+1\right)  S_{U}\left(  t+1\right)  -\left(
t+a\right)  \left(  t-N\right)  S_{U}\left(  t\right) \\
&  =\left(  b+1-a+N\right)  \nu_{0}+M\left[  \frac{\left(  t+b+1\right)
\left(  t+1\right)  }{t+1-\omega}-\frac{\left(  t+a\right)  \left(
t-N\right)  }{t-\omega}\right]  .
\end{align*}

\subsubsection{Symmetrized generalized Hahn polynomials of type I}

Special values%

\begin{align*}
a_{3}  &  \rightarrow-b-N,\quad b_{1}\rightarrow-a_{1}-N,\quad b_{2}%
\rightarrow-a_{2}-N,\\
b_{3}  &  \rightarrow b,\quad N\rightarrow2m,\quad x\rightarrow x+m.
\end{align*}

Linear functional%

\[
L_{\varsigma}\left[  r\right]  =\varrho_{\varsigma}\left(  0\right)
%TCIMACRO{\dsum \limits_{x=-m}^{m}}%
%BeginExpansion
{\displaystyle\sum\limits_{x=-m}^{m}}
%EndExpansion
r\left(  x\right)  \ \frac{\left(  a_{1}+m\right)  _{x}\left(  a_{2}+m\right)
_{x}\left(  -b-m\right)  _{x}\left(  -m\right)  _{x}}{\left(  -a_{1}%
-m+1\right)  _{x}\left(  -a_{2}-m+1\right)  _{x}\left(  b+1+m\right)  _{x}%
}\frac{1}{\left(  m+1\right)  _{x}},
\]
where%
\[
\varrho_{\varsigma}\left(  0\right)  =\binom{2m}{m}\frac{\ \left(
a_{1}\right)  _{m}\left(  a_{2}\right)  _{m}\left(  b+1+m\right)  _{m}%
}{\left(  a_{1}+m\right)  _{m}\left(  a_{2}+m\right)  _{m}\ \left(
b+1\right)  _{m}}.
\]

Pearson equation%
\[
\frac{\varrho_{\varsigma}\left(  x+1\right)  }{\varrho_{\varsigma}\left(
x\right)  }=\frac{\left(  x+a_{1}+m\right)  \left(  x+a_{2}+m\right)  \left(
x-b-m\right)  \left(  x-m\right)  }{\left(  x+1-a_{1}-m\right)  \left(
x+1-a_{2}-m\right)  \left(  x+1+b+m\right)  \left(  x+1+m\right)  }.
\]

\strut Moments on the basis $\{\overline{\phi}_{n}\left(  x\right)
\}_{n\geq0} . $

For $n\in\mathbb{N}_{0}$%

\begin{align*}
\nu_{n}^{\varsigma}  &  =\frac{\left(  a_{1}\right)  _{n}\left(  a_{2}\right)
_{n}\left(  -b-2m\right)  _{n}\left(  -2m\right)  _{n}}{\left(  -a_{1}%
-2m+1\right)  _{n}\left(  -a_{2}-2m+1\right)  _{n}\left(  b+1\right)  _{n}}\\
&  \times\ _{4}F_{3}\left[
\begin{array}
[c]{c}%
-2m+n,a_{1}+n,a_{2}+n,-b-2m+n\\
-a_{1}-2m+1+n,-a_{2}-2m+1+n,b+1+n
\end{array}
;1\right]  .
\end{align*}

Stieltjes transform difference equation%
\begin{align*}
&  \left(  t+1-a_{1}-m\right)  \left(  t+1-a_{2}-m\right)  \left(
t+1+b+m\right)  \left(  t+1+m\right)  S_{\varsigma}\left(  t+1\right) \\
&  -\left(  t+a_{1}+m\right)  \left(  t+a_{2}+m\right)  \left(  t-b-m\right)
\left(  t-m\right)  S_{\varsigma}\left(  t\right) \\
&  =\left[  2\left(  b-a_{1}-a_{2}+1\right)  \left(  t+1-m\right)  t+t\left(
t+1\right)  \right]  \nu_{0}^{\varsigma}\\
&  +\left[  b-m-2bm-m^{2}+1-\left(  b+1\right)  \left(  a_{1}+a_{2}\right)
+\left(  2b+4m+1\right)  a_{1}a_{2}\right]  \nu_{0}^{\varsigma}\\
&  +\left[  \left(  2t+3-4m\right)  \left(  b-a_{1}-a_{2}\right)  +2\left(
t+1-m\right)  \right]  \nu_{1}^{\varsigma}+\left(  2b+1-2a_{1}-2a_{2}\right)
\nu_{2}^{\varsigma}.
\end{align*}

\subsubsection{Reduced-Uvarov generalized Hahn polynomials of type II}

Let $\left\{  \omega,\Omega\right\}  \in\left\{  \left(  -a_{1},a_{1}%
+1\right)  ,\left(  -a_{2},a_{2}+1\right)  ,\left(  -b_{1},b_{1}\right)
,\left(  -b_{2},b_{2}\right)  ,\left(  0,0\right)  ,\left(  N,-N+1\right)
\right\}  .$

Special values%
\[
a_{3}\rightarrow\Omega,\quad b_{3}\rightarrow\Omega-1.
\]

Linear functional%

\[
L_{U}\left[  r\right]  =%
%TCIMACRO{\dsum \limits_{x=0}^{N}}%
%BeginExpansion
{\displaystyle\sum\limits_{x=0}^{N}}
%EndExpansion
r\left(  x\right)  \ \frac{\left(  -N\right)  _{x}\left(  a_{1}\right)
_{x}\left(  a_{2}\right)  _{x}}{\left(  b_{1}+1\right)  _{x}\left(
b_{2}+1\right)  _{x}}\frac{1}{x!}+Mr\left(  \omega\right)  .
\]

Pearson equation%
\[
\frac{\varrho_{U}\left(  x+1\right)  }{\varrho_{U}\left(  x\right)  }%
=\frac{\left(  x+a_{1}\right)  \left(  x+a_{2}\right)  \left(  x+\Omega
\right)  \left(  x-N\right)  }{\left(  x+b_{1}+1\right)  \left(
x+b_{2}+1\right)  \left(  x+\Omega\right)  \left(  x+1\right)  }.
\]

Moments%
\[
\nu_{n}^{U}=\frac{\left(  -N\right)  _{n}\left(  a_{1}\right)  _{n}\left(
a_{2}\right)  _{n}}{\left(  b_{1}+1\right)  _{n}\left(  b_{2}+1\right)  _{n}%
}\ _{3}F_{2}\left[
\begin{array}
[c]{c}%
-N+n,a_{1}+n,a_{2}+n\\
b_{1}+1+n,b_{2}+1+n
\end{array}
;1\right]  +M\phi_{n}\left(  \omega\right)  , \quad n\in\mathbb{N}_{0} .
\]

Stieltjes transform difference equation%
\begin{align*}
&  \left(  t+\Omega\right)  \left[  \left(  t+b_{1}+1\right)  \left(
t+b_{2}+1\right)  \left(  t+1\right)  S_{U}\left(  t+1\right)  -\left(
t+a_{1}\right)  \left(  t+a_{2}\right)  \left(  t-N\right)  S_{U}\left(
t\right)  \right] \\
&  =\left[  (N-a_{1}-a_{2}+b_{1}+b_{2}+2)\left(  t+\Omega\right)
t+(b_{1}+1)(b_{2}+1)t+N\left(  a_{1}+a_{2}\right)  t-a_{1}a_{2}t\right]
\nu_{0}^{U}\\
&  +\left[  (b_{1}+1)(b_{2}+1)\Omega+a_{1}a_{2}\left(  N-\Omega\right)
+\left(  a_{1}+a_{2}\right)  \Omega N\right]  \nu_{0}^{U}\\
&  +\left[  \left(  N-a_{1}-a_{2}+b_{1}+b_{2}+1\right)  \left(  t+\Omega
\right)  +N+b_{1}+b_{2}+b_{1}b_{2}+\allowbreak\left(  N-1\right)  \left(
a_{1}+a_{2}\right)  -a_{1}a_{2}\right]  \nu_{1}^{U}\\
&  +\left(  N+b_{1}+b_{2}-a_{1}-a_{2}\right)  \nu_{2}^{U}.
\end{align*}

Stieltjes transform $M-$dependent difference equation%
\begin{align*}
&  \left(  t+b_{1}+1\right)  \left(  t+b_{2}+1\right)  \left(  t+1\right)
S_{U}\left(  t+1\right)  -\left(  t+a_{1}\right)  \left(  t+a_{2}\right)
\left(  t-N\right)  S_{U}\left(  t\right) \\
&  =\left[  \left(  N+2+b_{1}+b_{2}-a_{1}-a_{2}\right)  t+b_{1}+b_{2}%
+b_{1}b_{2}+\left(  N-1\right)  \left(  a_{1}+a_{2}\right)  -a_{1}%
a_{2}+1\right]  \nu_{0}\\
&  +\left(  N+b_{1}+b_{2}-a_{1}-a_{2}+1\right)  \nu_{1}\\
&  +M\left[  \frac{\left(  t+b_{1}+1\right)  \left(  t+b_{2}+1\right)  \left(
t+1\right)  }{t-\omega+1}-\frac{\left(  t+a_{1}\right)  \left(  t+a_{2}%
\right)  \left(  t-N\right)  }{t-\omega}\right]  .
\end{align*}

\subsubsection{Christoffel generalized Hahn polynomials of type II}

Let $\omega\notin\left\{  -a_{1},-a_{2},-b_{1},-b_{2},0,N,\dfrac{\nu_{1}}%
{\nu_{0}}\right\}  .$

Special values%
\[
a_{3}\rightarrow-\omega+1,\quad b_{3}\rightarrow-\omega-1.
\]

Linear functional%
\[
L_{C}\left[  r\right]  =%
%TCIMACRO{\dsum \limits_{x=0}^{\infty}}%
%BeginExpansion
{\displaystyle\sum\limits_{x=0}^{\infty}}
%EndExpansion
r\left(  x\right)  \left(  x-\omega\right)  \frac{\left(  a_{1}\right)
_{x}\left(  a_{2}\right)  _{x}\left(  -N\right)  _{x}}{\left(  b_{1}+1\right)
_{x}\left(  b_{2}+1\right)  _{x}}\frac{1}{x!}.
\]

Pearson equation%
\[
\frac{\varrho_{C}\left(  x+1\right)  }{\varrho_{C}\left(  x\right)  }%
=\frac{\left(  x+a_{1}\right)  \left(  x+a_{2}\right)  \left(  x-N\right)
\left(  x+1-\omega\right)  }{\left(  x+b_{1}+1\right)  \left(  x+b_{2}%
+1\right)  \left(  x+1\right)  \left(  x-\omega\right)  }.
\]

Moments%
\[
\nu_{n}^{C}\left(  z\right)  =\nu_{n+1}+\left(  n-\omega\right)  \nu_{n},
\quad n\in\mathbb{N}_{0}.
\]

Stieltjes transform difference equation%

\begin{align*}
&  \left(  t-\omega\right)  \left(  t+b_{1}+1\right)  \left(  t+b_{2}%
+1\right)  \left(  t+1\right)  S_{C}\left(  t+1\right)  -\left(
t-\omega+1\right)  \left(  t+a_{1}\right)  \left(  t+a_{2}\right)  \left(
t-N\right)  S_{C}\left(  t\right) \\
&  =-\left[  \omega\left(  N+b_{1}+b_{2}+1\right)  -\left(  \omega-1\right)
\left(  a_{1}+a_{2}\right)  \right]  t^{2}\nu_{0}+\omega\left(  \omega
-1\right)  \left(  N+b_{1}+b_{2}+1\right)  t\nu_{0}\\
&  +\left[  \left(  -N+\omega\right)  a_{1}a_{2}+\left(  3\omega
-N\omega-\omega^{2}-1\right)  \left(  a_{1}+a_{2}\right)  +\omega\left(
\omega-b_{1}b_{2}\right)  \right]  t\nu_{0}\\
&  +\left[  \omega\left(  \omega-1\right)  \left(  N-1\right)  \left(
a_{1}+a_{2}\right)  +\left(  \omega-1\right)  \left(  N-\omega\right)
a_{1}a_{2}+\omega^{2}\left(  b_{2}+1\right)  \left(  b_{1}+1\right)  \right]
\nu_{0}\\
&  +\left(  N-a_{1}-a_{2}+b_{1}+b_{2}+1\right)  \left(  t-\omega\right)
\left(  t-\omega+1\right)  \nu_{1}.
\end{align*}

\subsubsection{Geronimus generalized Hahn polynomials of type II}

Let $\omega\notin\left\{  -b_{1}-1,-b_{2}-1,-1,1-a_{1},1-a_{2},1+N\right\}
\cup\mathbb{N}_{0}.$

Special values%
\[
a_{3}\rightarrow-\omega,\quad b_{3}\rightarrow-\omega.
\]

Linear functional%
\[
L_{G}\left[  r\right]  =%
%TCIMACRO{\dsum \limits_{x=0}^{\infty}}%
%BeginExpansion
{\displaystyle\sum\limits_{x=0}^{\infty}}
%EndExpansion
\frac{r\left(  x\right)  }{x-\omega}\frac{\left(  a_{1}\right)  _{x}\left(
a_{2}\right)  _{x}\left(  -N\right)  _{x}}{\left(  b_{1}+1\right)  _{x}\left(
b_{2}+1\right)  _{x}}\frac{1}{x!}+Mr\left(  \omega\right)  .
\]

Pearson equation%
\[
\frac{\varrho_{G}\left(  x+1\right)  }{\varrho_{G}\left(  x\right)  }%
=\frac{\left(  x+a_{1}\right)  \left(  x+a_{2}\right)  \left(  x-N\right)
\left(  x-\omega\right)  }{\left(  x+b_{1}+1\right)  \left(  x+b_{2}+1\right)
\left(  x+1\right)  \left(  x+1-\omega\right)  }.
\]

Moments%
\[
\nu_{n}^{G}\left(  z\right)  =\phi_{n}\left(  \omega\right)  \left[
M-S\left(  \omega\right)  +%
%TCIMACRO{\dsum \limits_{k=0}^{n-1}}%
%BeginExpansion
{\displaystyle\sum\limits_{k=0}^{n-1}}
%EndExpansion
\frac{\nu_{k}}{\phi_{k+1}\left(  \omega\right)  }\right]  ,\quad
n\in\mathbb{N}_{0}.
\]

Stieltjes transform difference equation%

\begin{align*}
&  \left(  t-\omega+1\right)  \left(  t+b_{1}+1\right)  \left(  t+b_{2}%
+1\right)  \left(  t+1\right)  S_{G}\left(  t+1\right)  -\left(
t-\omega\right)  \left(  t+a_{1}\right)  \left(  t+a_{2}\right)  \left(
t-N\right)  S_{G}\left(  t\right) \\
&  =\left[  \left(  N+2+b_{1}+b_{2}-a_{1}-a_{2}\right)  t+\left(
b_{1}+1\right)  \left(  b_{2}+1\right)  +\left(  N-1\right)  \left(
a_{1}+a_{2}\right)  -a_{1}a_{2}\right]  \nu_{0}\\
&  +\left(  N+b_{1}+b_{2}-a_{1}-a_{2}+1\right)  \nu_{1}\\
&  +\left[  \allowbreak\left(  N-a_{1}-a_{2}+b_{1}+b_{2}+3\right)
t^{2}+\left(  2b_{1}+2b_{2}+Na_{1}+Na_{2}-a_{1}a_{2}+b_{1}b_{2}+3\right)
t\right]  \nu_{0}^{G}\\
&  +\left(  b_{1}+b_{2}+b_{1}b_{2}+Na_{1}a_{2}+1\right)  \nu_{0}^{G}.
\end{align*}

\subsubsection{Symmetrized polynomials of type $(1,2)$}

Special values%
\[
a_{1}\rightarrow a,\quad a_{2}\rightarrow-b_{1}-N,\quad a_{3}\rightarrow
-b_{2}-N,\quad b_{3}\rightarrow-a-N,\quad N\rightarrow2m,\quad x\rightarrow
x+m.
\]

Linear functional%
\[
L_{\varsigma}\left[  r\right]  =\varrho_{\varsigma}\left(  0\right)
%TCIMACRO{\dsum \limits_{x=-m}^{m}}%
%BeginExpansion
{\displaystyle\sum\limits_{x=-m}^{m}}
%EndExpansion
r\left(  x\right)  \ \frac{\left(  a+m\right)  _{x}\left(  -b_{1}-m\right)
_{x}\left(  -b_{2}-m\right)  _{x}\left(  -m\right)  _{x}}{\left(
b_{1}+1+m\right)  _{x}\left(  b_{2}+1+m\right)  _{x}\left(  -a-m+1\right)
_{x}}\frac{1}{\left(  m+1\right)  _{x}},
\]
where%
\[
\varrho_{\varsigma}\left(  0\right)  =\binom{2m}{m}\frac{\left(  a\right)
_{m}\ \left(  b_{1}+1+m\right)  _{m}\left(  b_{1}+1+m\right)  _{m}}{\left(
a+m\right)  _{m}\ \left(  b_{1}+1\right)  _{m}\left(  b_{2}+1\right)  _{m}}.
\]

Pearson equation%
\[
\frac{\varrho_{\varsigma}\left(  x+1\right)  }{\varrho_{\varsigma}\left(
x\right)  }=\frac{\left(  x+a+m\right)  \left(  x-b_{1}-m\right)  \left(
x-b_{2}-m\right)  \left(  x-m\right)  }{\left(  x+b_{1}+1+m\right)  \left(
x+b_{2}+1+m\right)  \left(  x+1-a-m\right)  \left(  x+1+m\right)  }.
\]

\strut Moments on the basis $\{\overline{\phi}_{n}\left(  x\right)
\}_{n\geq0} $%

\begin{align*}
\nu_{n}^{\varsigma}  &  =\frac{\left(  -2m\right)  _{n}\left(  a\right)
_{n}\left(  -b_{1}-2m\right)  _{n}\left(  -b_{2}-2m\right)  _{n}}{\left(
1-a-2m\right)  _{n}\left(  b_{1}+1\right)  _{n}\left(  b_{2}+1\right)  _{n}}\\
&  \times\ _{4}F_{3}\left[
\begin{array}
[c]{c}%
-2m+n,a+n,-b_{1}-2m,-b_{2}-2m+n,\\
1-a-2m+n,b_{1}+1+n,b_{2}+1+n
\end{array}
;1\right]  , \quad n\in\mathbb{N}_{0} .
\end{align*}

Stieltjes transform difference equation%

\begin{gather*}
\left(  t+1-a-m\right)  \left(  t+b_{1}+1+m\right)  \left(  t+b_{2}%
+1+m\right)  \left(  t+1+m\right)  S_{\varsigma}\left(  t+1\right) \\
-\left(  t+a+m\right)  \left(  t-b_{1}-m\right)  \left(  t-b_{2}-m\right)
\left(  t-m\right)  S_{\varsigma}\left(  t\right) \\
=\left[  \allowbreak\left(  4m+3+2b_{1}+2b_{2}-2a\right)  t^{2}+2\left(
m-1\right)  \left(  a-b_{1}-b_{2}\right)  t\right]  \nu_{0}^{\varsigma}\\
+\left[  \left(  -4\allowbreak m^{2}+2m+3\right)  t-\left(  18a+4ab_{1}%
+4ab_{2}+1\right)  m-\left(  b_{1}+b_{2}+1\right)  \left(  a-1\right)
-b_{1}b_{2}\left(  2a-1\right)  \right]  \nu_{0}^{\varsigma}\\
+\left[  \allowbreak\left(  2t-4m+3\right)  \left(  b_{1}+b_{2}-a\right)
+2\left(  t+2mt+1\right)  +4m\left(  1-2m\right)  \right]  \nu_{1}^{\varsigma
}\\
+\left(  4m+1+2b_{1}+2b_{2}-2a\right)  \nu_{2}^{\varsigma}.
\end{gather*}

\section{Conclusions}

We have considered all solutions of the Pearson equation%
\[
\frac{\varrho\left(  x+1\right)  }{\varrho\left(  x\right)  }=\frac
{\eta\left(  x\right)  }{\sigma\left(  x+1\right)  },
\]
with $\deg\left(  \sigma\right)  \leq4$ and $\deg\left(  \sigma-\eta\right)
\leq3.$ We deduce $15$ canonical cases, from which $42$ subcases can be
obtained and we relate those to rational spectral transformations (Uvarov,
Christoffel, and Geronimus), truncations, and symmetrizations.

For each family, we have listed a representation of the moments and the first
order linear difference equations satisfied by their Stieltjes transform.

Extensions to the class $s=3$ are possible, but the complexity of the formulas
will demand new forms of notation in order to be able to type the results in a
comprehensive way.

\begin{acknowledgement}
This work of the first author was done while visiting the Johannes Kep\-ler
Universit\"{a}t Linz and supported by the strategic program "Innovatives
O\"{O}-- 2010 plus" from the Upper Austrian Government. The second author
acknowledges financial support of Direcci\'on general de Investigaci\'on,
Ministerio de Econom\'ia, Industria y Competitividad of Spain, grant
MTM2015-65888. C4- 2-P.
\end{acknowledgement}


\begin{thebibliography}{99}                                                                                               %


\bibitem {MR1464669}F.~Abdelkarim and P.~Maroni. \newblock The {$D_{\omega}$%
}-classical orthogonal polynomials. \newblock {\em Results Math.},
32(1-2):1--28, 1997.

\bibitem {MR2061464}R.~\'{A}lvarez Nodarse, J.~Arves\'{u}, and
F.~Marcell\'{a}n. \newblock Modifications of quasi-definite linear functionals
via addition of delta and derivatives of delta {D}irac functions.
\newblock {\em Indag. Math. (N.S.)}, 15(1):1--20, 2004.

\bibitem {MR1379116}R.~\'{A}lvarez Nodarse, A.~G. Garc\'{\i}a, and
F.~Marcell\'{a}n. \newblock On the properties for modifications of classical
orthogonal polynomials of discrete variables. \newblock In \emph{Proceedings
of the {I}nternational {C}onference on {O}rthogonality, {M}oment {P}roblems
and {C}ontinued {F}ractions ({D}elft, 1994)}, volume~65, pages 3--18, 1995.

\bibitem {MR1353079}R.~\'{A}lvarez Nodarse and F.~Marcell\'{a}n.
\newblock Difference equation for modifications of {M}eixner polynomials.
\newblock {\em J. Math. Anal. Appl.}, 194(1):250--258, 1995.

\bibitem {MR1410602}R.~\'{A}lvarez Nodarse and F.~Marcell\'{a}n. \newblock The
modification of classical {H}ahn polynomials of a discrete variable.
\newblock {\em Integral Transform. Spec. Funct.}, 3(4):243--262, 1995.

\bibitem {MR2064407}R.~\'{A}lvarez Nodarse and J.~Petronilho. \newblock On the
{K}rall-type discrete polynomials. \newblock {\em J. Math. Anal. Appl.},
295(1):55--69, 2004.

\bibitem {MR3360482}I.~Area, D.~K. Dimitrov, E.~Godoy, and V.~Paschoa.
\newblock Bounds for the zeros of symmetric {K}ravchuk polynomials.
\newblock {\em Numer. Algorithms}, 69(3):611--624, 2015.

\bibitem {MR2033351}I.~Area, E.~Godoy, A.~Ronveaux, and A.~Zarzo.
\newblock Classical symmetric orthogonal polynomials of a discrete variable.
\newblock {\em Integral Transforms Spec. Funct.}, 15(1):1--12, 2004.

\bibitem {MR1186737}S.~Belmehdi. \newblock On semi-classical linear
functionals of class {$s=1$}. {C}lassification and integral representations.
\newblock {\em Indag. Math. (N.S.)}, 3(3):253--275, 1992.

\bibitem {MR3173496}L.~Boelen, G.~Filipuk, C.~Smet, W.~Van~Assche, and
L.~Zhang. \newblock The generalized {K}rawtchouk polynomials and the fifth
{P}ainlev\'{e} equation. \newblock {\em J. Difference Equ. Appl.},
19(9):1437--1451, 2013.

\bibitem {MR2749070}L.~Boelen, G.~Filipuk, and W.~Van~Assche.
\newblock Recurrence coefficients of generalized {M}eixner polynomials and
{P}ainlev\'{e} equations. \newblock {\em J. Phys. A}, 44(3):035202, 19, 2011.

\bibitem {MR2324861}M.~I. Bueno and F.~M. Dopico. \newblock A more accurate
algorithm for computing the {C}hristoffel transformation.
\newblock {\em J. Comput. Appl. Math.}, 205(1):567--582, 2007.

\bibitem {MR2055354}M.~I. Bueno and F.~Marcell\'{a}n. \newblock Darboux
transformation and perturbation of linear functionals.
\newblock {\em Linear Algebra Appl.}, 384:215--242, 2004.

\bibitem {MR1579059}E.~B. Christoffel. \newblock \"{U}ber die {G}au\ss ische
{Q}uadratur und eine {V}erallgemeinerung derselben.
\newblock {\em J. Reine Angew. Math.}, 55:61--82, 1858.

\bibitem {MR3208415}M.~Derevyagin, J.~C. Garc\'{\i}a-Ardila, and
F.~Marcell\'{a}n. \newblock Multiple {G}eronimus transformations.
\newblock {\em Linear Algebra Appl.}, 454:158--183, 2014.

\bibitem {MR3264577}M.~Derevyagin and F.~Marcell\'{a}n. \newblock A note on
the {G}eronimus transformation and {S}obolev orthogonal polynomials.
\newblock {\em Numer. Algorithms}, 67(2):271--287, 2014.

\bibitem {MR3813279}D.~Dominici. \newblock Laguerre-{F}reud equations for
generalized {H}ahn polynomials of type {I}.
\newblock {\em J. Difference Equ. Appl.}, 24(6):916--940, 2018.

\bibitem {MR3492864}D.~Dominici. \newblock Polynomial sequences associated
with the moments of hypergeometric  weights.
\newblock {\em SIGMA Symmetry Integrability Geom. Methods Appl.}, 12:Paper
No.  044, 18, 2016.

\bibitem {MR3227440}D.~Dominici and F.~Marcell\'{a}n. \newblock Discrete
semiclassical orthogonal polynomials of class one.
\newblock {\em Pacific J. Math.}, 268(2):389--411, 2014.

\bibitem {MR2128146}S.~Elaydi.
\newblock {\em An introduction to difference equations}.
\newblock Undergraduate Texts in Mathematics. Springer, New York, third
edition, 2005.

\bibitem {MR3846851}G.~Filipuk and W.~Van~Assche. \newblock Discrete
orthogonal polynomials with hypergeometric weights and {P}ainlev\'{e} {VI}.
\newblock {\em SIGMA Symmetry Integrability Geom. Methods Appl.}, 14:Paper No.
088, 19, 2018.

\bibitem {MR2861208}G.~Filipuk and W.~Van~Assche. \newblock Recurrence
coefficients of a new generalization of the {M}eixner polynomials.
\newblock {\em SIGMA Symmetry Integrability Geom. Methods Appl.}, 7:Paper 068,
11, 2011.

\bibitem {MR1140648}D.~Galant. \newblock Algebraic methods for modified
orthogonal polynomials. \newblock {\em Math. Comp.}, 59(200):541--546, 1992.

\bibitem {MR1340932}A.~G. Garc\'{\i}a, F.~Marcell\'{a}n, and L.~Salto.
\newblock A distributional study of discrete classical orthogonal polynomials.
\newblock In \emph{Proceedings of the {F}ourth {I}nternational {S}ymposium on
{O}rthogonal {P}olynomials and their {A}pplications ({E}vian-{L}es-{B}ains,
1992)}, volume~57, pages 147--162, 1995.

\bibitem {MR667829}W.~Gautschi. \newblock On generating orthogonal
polynomials. \newblock {\em SIAM J. Sci. Statist. Comput.}, 3(3):289--317, 1982.

\bibitem {MR0004339}J.~Geronimus. \newblock On polynomials orthogonal with
regard to a given sequence of numbers.
\newblock {\em Comm. Inst. Sci. Math. M\'{e}c. Univ. Kharkoff [Zapiski Inst.
Mat. Mech.] (4)}, 17:3--18, 1940.

\bibitem {MR0004340}J.~Geronimus. \newblock Sur quelques propri\'{e}t\'{e}s
des polynomes orthogonaux g\'{e}n\'{e}ralis\'{e}s.
\newblock {\em Rec. Math. [Mat. Sbornik] N.S.}, 9 (51):121--135, 1941.

\bibitem {MR1467146}E.~Godoy, F.~Marcell\'{a}n, L.~Salto, and A.~Zarzo.
\newblock Perturbations of discrete semiclassical functionals by {D}irac
masses. \newblock {\em Integral Transform. Spec. Funct.}, 5(1-2):19--46, 1997.

\bibitem {MR1737084}M.~N. Hounkonnou, C.~Hounga, and A.~Ronveaux.
\newblock Discrete semi-classical orthogonal polynomials: generalized
{C}harlier. \newblock {\em J. Comput. Appl. Math.}, 114(2):361--366, 2000.

\bibitem {MR2191786}M.~E.~H. Ismail.
\newblock {\em Classical and quantum orthogonal polynomials in one variable},
volume~98 of \emph{Encyclopedia of Mathematics and its Applications}.
\newblock Cambridge University Press, Cambridge, 2005.

\bibitem {MR2656096}R.~Koekoek, P.~A. Lesky, and R.~F. Swarttouw.
\newblock {\em Hypergeometric orthogonal polynomials and their
{$q$}-analogues}. \newblock Springer Monographs in Mathematics.
Springer-Verlag, Berlin, 2010.

\bibitem {MR1934914}J.~H. Lee and K.~H. Kwon. \newblock Division problem of
moment functionals. \newblock {\em Rocky Mountain J. Math.}, 32(2):739--758,
2002. \newblock Conference on Special Functions (Tempe, AZ, 2000).

\bibitem {MR1199643}F.~Marcell\'{a}n and P.~Maroni. \newblock Sur l'adjonction
d'une masse de {D}irac \`a une forme r\'{e}guli\`ere et semi-classique.
\newblock {\em Ann. Mat. Pura Appl. (4)}, 162:1--22, 1992.

\bibitem {MR1665164}F.~Marcell\'{a}n and L.~Salto. \newblock Discrete
semi-classical orthogonal polynomials.
\newblock {\em J. Differ. Equations Appl.}, 4(5):463--496, 1998.

\bibitem {MR2929632}F.~Marcell\'{a}n, M.~Sghaier, and M.~Zaatra. \newblock On
semiclassical linear functionals of class {$s=2$}: classification and integral
representations. \newblock {\em J. Difference Equ. Appl.}, 18(6):973--1000, 2012.

\bibitem {MR803215}P.~Maroni. \newblock Une caract\'{e}risation des
polyn\^{o}mes orthogonaux semi-classiques.
\newblock {\em C. R. Acad. Sci. Paris S\'{e}r. I Math.}, 301(6):269--272, 1985.

\bibitem {MR932783}P.~Maroni. \newblock Prol\'{e}gom\`enes \`a l'\'{e}tude des
polyn\^{o}mes orthogonaux semi-classiques.
\newblock {\em Ann. Mat. Pura Appl. (4)}, 149:165--184, 1987.

\bibitem {MR1105709}P.~Maroni. \newblock Sur la suite de polyn\^{o}mes
orthogonaux associ\'{e}e \`a la forme {$u=\delta_{c}+\lambda(x-c)^{-1}L$}.
\newblock {\em Period. Math. Hungar.}, 21(3):223--248, 1990.

\bibitem {MR1270222}P.~Maroni. \newblock Une th\'eorie alg\'ebrique des
polyn\^omes orthogonaux. {A}pplication aux polyn\^omes orthogonaux
semi-classiques. \newblock In \emph{Orthogonal polynomials and their
applications ({E}rice, 1990)}, volume~9 of \emph{IMACS Ann. Comput. Appl.
Math.}, pages 95--130. Baltzer, Basel, 1991.

\bibitem {MR2993855}M.~Masjed-Jamei and I.~Area. \newblock A basic class of
symmetric orthogonal polynomials of a discrete variable.
\newblock {\em J. Math. Anal. Appl.}, 399(1):291--305, 2013.

\bibitem {MR3173503}M.~Masjed-Jamei and I.~Area. \newblock A symmetric
generalization of {S}turm-{L}iouville problems in discrete spaces.
\newblock {\em J. Difference Equ. Appl.}, 19(9):1544--1562, 2013.

\bibitem {MR1149380}A.~F. Nikiforov, S.~K. Suslov, and V.~B. Uvarov.
\newblock {\em Classical orthogonal polynomials of a discrete variable}.
\newblock Springer Series in Computational Physics. Springer-Verlag, Berlin, 1991.

\bibitem {MR2466333}K.~Oldham, J.~Myland, and J.~Spanier.
\newblock {\em An atlas of functions}. \newblock Springer, New York, second
edition, 2009.

\bibitem {MR2723248}F.~W.~J. Olver, D.~W. Lozier, R.~F. Boisvert, and C.~W.
Clark, editors. \newblock {\em N{IST} handbook of mathematical functions}.
\newblock U.S. Department of Commerce National Institute of Standards and
Technology, Washington, DC, 2010.

\bibitem {MR1199259}F.~Peherstorfer. \newblock Finite perturbations of
orthogonal polynomials. \newblock {\em J. Comput. Appl. Math.},
44(3):275--302, 1992.

\bibitem {MR1939588}A.~Ronveaux and L.~Salto. \newblock Discrete orthogonal
polynomials---polynomial modification of a classical functional.
\newblock {\em J. Differ. Equations Appl.}, 7(3):323--344, 2001.

\bibitem {MR3765964}M.~Sghaier and M.~Zaatra. \newblock A family of discrete
semi-classical orthogonal polynomials of class one.
\newblock {\em Period. Math. Hungar.}, 76(1):68--87, 2018.

\bibitem {MR0262764}V.~B. Uvarov. \newblock The connection between systems of
polynomials that are orthogonal with respect to different distribution
functions. \newblock {\em \v{Z}. Vy\v{c}isl. Mat. i Mat. Fiz.}, 9:1253--1262, 1969.

\bibitem {MR3729446}W.~Van~Assche.
\newblock {\em Orthogonal polynomials and {P}ainlev\'{e} equations}, volume~27
of \emph{Australian Mathematical Society Lecture Series}. \newblock Cambridge
University Press, Cambridge, 2018.

\bibitem {MR1482157}A.~Zhedanov. \newblock Rational spectral transformations
and orthogonal polynomials. \newblock {\em J. Comput. Appl. Math.},
85(1):67--86, 1997.
\end{thebibliography}
\end{document}